\documentclass[11pt,a4paper,twoside,reqno]{amsart}

\usepackage[backref=page]{hyperref}
\usepackage{hyperref}
\usepackage{url,xcolor}
\usepackage{amssymb,latexsym}
\usepackage{amsmath,amsthm}
\usepackage[latin1]{inputenc}
\usepackage{tikz-cd}
\usepackage{enumerate}
\usepackage[all]{xy} \CompileMatrices
\usepackage{slashed} %% for the Dirac operator
\usepackage{youngtab}
\usepackage{tensor}

\SelectTips{cm}{12} % computer modern fonts of size 12 for arrowheads
\theoremstyle{plain}
\newtheorem{theorem}{Theorem}[section]
\newtheorem{lemma}[theorem]{Lemma}
\newtheorem{corollary}[theorem]{Corollary}
\newtheorem{proposition}[theorem]{Proposition}
\newtheorem{prop}[theorem]{Proposition}
\newtheorem{conjecture}[theorem]{Conjecture}
\theoremstyle{definition}
\newtheorem{definition}[theorem]{Definition}
\newtheorem{defn}[theorem]{Definition}
\newtheorem{definition-theorem}[theorem]{Definition-Theorem}

\newtheorem{remark}[theorem]{Remark}

\newtheorem{question}[theorem]{Question}

\numberwithin{equation}{section} 

\newcommand{\la}{\langle} \newcommand{\ra}{\rangle}

\newcommand{\tr}{\operatorname{tr}}
\newcommand{\Id}{\operatorname{Id}}

\newcommand{\Hom}{\operatorname{Hom}}
\newcommand{\Ker}{\operatorname{Ker}}
\newcommand{\ad}{\operatorname{ad}}

\newcommand{\Aut}{\operatorname{Aut}}
\newcommand{\dbar}{\bar{\partial}}

\newcommand{\CC}{{\mathbb C}}

\newcommand{\RR}{{\mathbb R}}
\newcommand{\ZZ}{{\mathbb Z}}

\renewcommand{\(}{\left(}
\renewcommand{\)}{\right)}

\newcommand{\surj}{\to\kern-1.8ex\to}

\newcommand{\cM}{\mathcal{M}}

\newcommand{\cG}{\mathcal{G}}

\newcommand{\Lie}{\operatorname{Lie}}

\newcommand{\dt}{\frac{\partial}{\partial t}}

\newcommand{\del}{\partial}
\newcommand{\delb}{\bar{\partial}}
\newcommand{\brs}[1]{\left| #1 \right|}

\newcommand{\gD}{\Delta}
\newcommand{\gd}{\delta}

\newcommand{\gU}{\Upsilon}
\newcommand{\gL}{\Lambda}

\newcommand{\gl}{\lambda}

\newcommand{\gtot}{\bar{g}}
\newcommand{\Htot}{\bar{H}}
\newcommand{\Jtot}{\bar{J}}
\newcommand{\otot}{\bar{\omega}}

\newcommand{\gw}{\omega}
\newcommand{\ga}{\alpha}
\newcommand{\gb}{\beta}
\renewcommand{\ge}{\epsilon}
\newcommand{\N}{\nabla}
\newcommand{\FF}{\mathcal F}

\newcommand{\til}[1]{\widetilde{#1}}

\renewcommand{\bar}[1]{\overline{#1}}

\renewcommand{\i}{\sqrt{-1}}

\newcommand{\bga}{\bar{\alpha}}
\newcommand{\bgb}{\bar{\beta}}

\newcommand{\bi}{\bar{i}}
\newcommand{\bj}{\bar{j}}
\newcommand{\bk}{\bar{k}}
\newcommand{\bl}{\bar{l}}

\newcommand{\bp}{\bar{p}}

\newcommand{\IP}[1]{\left<#1\right>}

\newcommand{\bgg}{\bar{\gamma}}

\DeclareMathOperator{\Rc}{Rc}
\DeclareMathOperator{\Rm}{Rm}
\DeclareMathOperator{\inj}{inj}

\DeclareMathOperator{\osc}{osc}

\begin{document}

\title[Pluriclosed flow and the Hull-Strominger system]{Pluriclosed flow and the Hull-Strominger system}

\author{Mario Garcia-Fernandez}
\address{Instituto de Ciencias Matem\'aticas (CSIC-UAM-UC3M-UCM)\\ Nicol\'as Cabrera 13--15, Cantoblanco\\ 28049 Madrid, Spain}
\email{mario.garcia@icmat.es}

\author{Raul Gonzalez Molina}
\address{Instituto de Ciencias Matem\'aticas (CSIC-UAM-UC3M-UCM)\\ Nicol\'as Cabrera 13--15, Cantoblanco\\ 28049 Madrid, Spain}
\email{raul.gonzalez@icmat.es}

\author{Jeffrey Streets}
\address{Rowland Hall\\
        University of California, Irvine\\
         Irvine, CA 92617}
\email{\href{mailto:jstreets@uci.edu}{jstreets@uci.edu}}

\begin{abstract} 
We define a natural extension of pluriclosed flow aiming at constructing solutions of the Hull-Strominger system.  We give several geometric formulations of this flow, which yield a series of a priori estimates for the flow and also for the Hull-Strominger system. The evolution equations are derived using the theory of string algebroids, a class of Courant algebroids which occur naturally in higher gauge theory. Using this, we interpret the flow as generalized Ricci flow and also as a higher/coupled version of Hermitian-Yang-Mills flow, proving furthermore that it is compatible with symmetry reduction. Regarding our main analytical results, we prove a priori $C^{\infty}$ estimates for uniformly parabolic solutions.  This in particular settles the question of smooth regularity of uniformly elliptic solutions of the Hull-Strominger system, generalizing Yau's $C^3$ estimate for the complex Monge-Amp\`ere equation.  We prove global existence and convergence results for the flow on special backgrounds, and discuss a conjectural relationship of the flow to the geometrization of Reid's fantasy.
\end{abstract}

%\date{\today}

\thanks{
MGF and RGM are partially supported by the Spanish Ministry of Science and Innovation, through the `Severo Ochoa Programme for Centres of Excellence in R\&D' (CEX2019-000904-S) and under grants PID2022-141387NB-C22 and CNS2022-135784. JS gratefully acknowledges support from a Simons Fellowship and from the NSF via DMS-2203536.}

\maketitle

\section{Introduction}

Let $(M,J)$ be a compact complex manifold of dimension $n$ with a holomorphic volume form $\Omega$. Let $K^c$ be a complex reductive Lie group. Fix a holomorphic principal $K^c$-bundle $P^c$  over $M$. The Hull-Strominger system, for a Hermitian metric $g$ on $M$ and a reduction $h$ to a maximal compact subgroup $K \subset K^c$ of $P^c$, is
\begin{gather}\label{f:HS}
\begin{split}
 F_h \wedge \gw^{n-1} =&\ 0,\\
 d^{\star} \omega - d^c\log \|\Omega\|_\gw =&\ 0,\\
 d d^c \gw + \IP{F_h \wedge F_h}_{\mathfrak k} =&\ 0,
\end{split}
\end{gather}
where $\gw$ is the K\"ahler form of $g$ and $F_h$ is the Chern curvature of $h$. Here $\IP{,}_{\mathfrak k}$ denotes a fixed biinvariant inner product on the Lie algebra $\mathfrak{k}$ of $K$. The first line is the Hermitian-Einstein equation for $h$, while the second line is equivalent to the conformally balanced condition $d(\|\Omega\|_\gw \omega^{n-1}) = 0$, and implies that the holonomy of the Bismut connection of $g$ lies in $SU(n)$. The final line is known as the \emph{heterotic Bianchi identity} and has its roots in the physical origins of the system in the heterotic string theory \cite{Hull, StromingerSST}. In this setup, $P^c$ is the bundle of split frames of $T^{1,0} \oplus V$, for a choice of auxiliary holomorphic vector bundle $V$ related to the gauge degrees of freedom of the physical theory.

The Hull-Strominger system has generated strong interest in mathematics. As originally proposed in the seminal works \cite{FuYau2,FuYau1,LiYau2005}, it is expected that in complex dimension three the Hull-Strominger system plays a key role in the geometrization of Reid's fantasy \cite{ Collins2024stability,FuLiYau2012,Reid1987moduli}, whose key idea is to connect the moduli space of complex threefolds with trivial canonical bundle via conifold transitions.  A parabolic flow approach to constructing solutions of the Hull-Strominger system has been developed, called the anomaly flow \cite{PPZ0} (cf. the recent surveys \cite{phong2023geometric,picard2024strominger}). An alternative variational approach, using generalized geometry and the dilaton functional, has been proposed in \cite{garciafern2018canonical}. Observe that, due to its physical origins, in some references the choice of Chern connection in the tangent bundle is taken to be determined by $\omega$ (whenever $P^c$ is constructed from $T^{1,0} \oplus V$), but we shall not adopt this convention here (cf. \cite{GFGM,picard2024strominger}).

Building on \cite{de2017restrictions}, it was shown in \cite{GFGM} that the Hull-Strominger system implies the existence of a $SU(n)$-instanton on an associated holomorphic string algebroid, a class of Courant algebroids introduced in \cite{garciafern2018holomorphic} which occur naturally in higher gauge theory \cite{Tellez}.  A full analysis of this instanton condition reveals that it is equivalent to a more general system of equations, called the coupled Hermitian-Einstein system \cite{GFGM}:
\begin{gather} \label{f:CHE}
    \begin{split}
        %F_h \wedge \gw^{n-1} =&\ \frac{z}{n} \gw^n,\\
        \Lambda_\omega F_h =&\ z,\\
        \rho_B + \IP{z, F_h}_{\mathfrak k} =&\ 0,\\
        d d^c \gw + \IP{F_h \wedge F_h}_{\mathfrak k} = &\ 0,
    \end{split}
\end{gather}
where $z$ is a central element of $\mathfrak k$ and $\Lambda_\omega F_h \in \Gamma(\ad P^c)$ is the smooth section of the adjoint bundle $\ad P^c$ defined by $\Lambda_\omega F_h \omega^n = n F_h \wedge \omega^{n-1}$.  Notice that the conformally balanced condition has now been replaced by an equation for the Bismut-Ricci form $\rho_B$.

Our goal in this paper is to develop a new approach to understanding the existence and regularity properties of the Hull-Strominger/coupled Hermitian Einstein system via generalized Ricci flow and pluriclosed flow \cite{GRFBook,PCF}.  More specifically, we propose the study of the parabolic flow equations
\begin{gather} \label{f:PCFSA}
\begin{split}
    \dt \omega & = -\rho_B(\omega)^{1,1} + \i\IP{S^h_g, F_h}_{\mathfrak{k}},\\
h^{-1} \dt h  & = - S^h_g,
\end{split}
\end{gather}
for pairs $(g,h)$, where $S_g^h = \i \Lambda_\gw F_h$. In what follows we will motivate this flow through the geometry of string algebroids. We give several different geometric formulations of this system of equations, showing that it induces generalized Ricci flow on a smooth string algebroid as in \cite{GF14}, as well as a Hermitian-Yang-Mills (HYM) type flow on a holomorphic string algebroid, building on \cite{JFS, garciafern2020gauge, JordanStreets}, and finally that is a compatible with a symmetry reduction principle for generalized Ricci flow \cite{GRFBook} and pluriclosed flow \cite{PCF}.
These different geometric formulations lead to a series of a priori estimates for the flow and for the Hull-Strominger system.  These results point towards the use of this flow as a more general tool for geometrization problems, and we discuss this in the final section.

\subsection{Pluriclosed flow on string algebroids}

To begin we briefly recall fundamental aspects of generalized Ricci flow and pluriclosed flow.  Given a smooth manifold $M$ and closed three-form $H_0$, a one-parameter family of pairs $(g_t, b_t)$ of Riemannian metrics and two-forms is a solution to \emph{generalized Ricci flow} if
\begin{gather} \label{f:GRF}
    \begin{split}
    \dt g =&\ -2 \Rc^g + \tfrac{1}{2} H^2,\\
    \dt b =&\ - d^\star_g H, \qquad H = H_0 + db,    
    \end{split}
\end{gather}
where $H^2(X,Y) = \IP{i_X H, i_Y H}$.  This flow arises in mathematical physics as the leading order equations of the renormalization group flow for a certain sigma model \cite{Polchinski}.  Short-time existence and regularity properties were established in \cite{Streetsexpent}.  The flow can be interpreted as the canonical `Ricci flow' for generalized metrics on an exact Courant algebroid \cite{StreetsTdual,GF19,GRFBook}, which leads for instance to the proof that these equations are preserved by T-duality \cite{GF19, StreetsTdual}.  This equation is linked to complex geometry through the pluriclosed flow \cite{PCF} and generalized K\"ahler-Ricci flow \cite{GKRF}.  In particular, on a complex manifold $(M^{2n}, J)$ a one-parameter family $(\gw_t, \gb_t)$ of pluriclosed metrics and $(2,0)$-forms is a solution of pluriclosed flow if
\begin{gather} \label{f:PCF}
\begin{split}
    \dt \gw =&\ - \rho_B^{1,1}\\
    \dt \gb =&\ - \rho_B^{2,0},
\end{split}
\end{gather}
where $\rho_B$ is the Ricci form of the associated Bismut connection.  In \cite{PCFReg} it was shown that pluriclosed flow is gauge-equivalent to generalized Ricci flow, thus providing a link to the geometry of exact Courant algebroids.  Furthermore, in \cite{JFS,JordanStreets,SBIPCF} it was shown that pluriclosed flow admits an interpretation as a kind of Donaldson HYM flow \cite{DonaldsonHYM} on a holomorphic Courant algebroid, leading to global existence and convergence results \cite{JFS}.

Prior works \cite{GF14,GF19} shown that the Hull-Strominger system implies the vanishing of the Ricci curvature of an associated generalized metric on a real string algebroid. Motivated by this, the first named author defined a Ricci flow equation for generalized metrics on string algebroids \cite{GF19}.  This equation, which is our starting point, couples the generalized Ricci flow further to Yang-Mills flow.  In particular, fix $K$ a compact semisimple Lie group whose Lie algebra $\mathfrak k$ is endowed with a nondegenerate bilinear form $\IP{,}_{\mathfrak k}$ and fix $P \to M$ a principal $K$-bundle. We say that a one-parameter family of triples $(g_t, b_t, A_t)$ of Riemannian metrics $g_t$ on $M$, two-forms $b_t$ on $M$, and principal connections $A_t$ on $P$ satisfy the \emph{(string algebroid) generalized Ricci flow} if
\begin{equation}\label{eq:GRFgbAintro}
\begin{split}
\dt g & = -2\operatorname{Rc} + \frac{1}{2} H^2 + 2 F_A^2,\\
\dt b & = - d^{\star}H - \langle a \wedge (  d_A^\star F_A - F_A \lrcorner H ) \rangle_{\mathfrak{k}} ,\\
\dt A & = - d_A^{\star} F_A + F_A \lrcorner \ H,
\end{split}
\end{equation}
where $A_t = A_0 + a_t$ and $H_t$ is dictated by the \emph{anomally cancellation equation}
\begin{equation*}
H_t := H_0 + db + 2 \langle a_t \wedge F_{A_0} \rangle_{\mathfrak{k}} + \langle a_t  \wedge d_{A_0}a_t \rangle_{\mathfrak{k}} + \frac{1}{3}\langle a_t \wedge [a_t \wedge a_t] \rangle_{\mathfrak{k}}.
\end{equation*}
The delicate relationship between $H, b$ and $a$ is predicted by the action of natural morphisms on the Dorfman bracket of a string algebroid (cf. \S \ref{s:SA}).  We review and expand upon the fundamental theory of (\ref{eq:GRFgbAintro}) in \S \ref {s:SA} and \S \ref{s:GRFSA}.

We emphasize here an important point in our discussion: the inner product $\IP{,}_{\mathfrak k}$ is allowed to be \emph{indefinite}.  Indeed, in the original physical setup this inner product is chosen to have mixed signature, and this is related to the anomaly cancellation mechanism.  For our discussion here there will be important conceptual and analytic distinctions between the definite and indefinite cases.

In view of the discussion above we may expect \eqref{eq:GRFgbAintro} to be a tool for constructing solutions of the Hull-Strominger system.  However, it is not obvious that this evolution equation preserves any structure specific to the underlying complex geometry of the system.  We thus define a new parabolic flow for the coupled Hermitian-Einstein system motivated by the theory of pluriclosed flow, and eventually show it is gauge-equivalent to \eqref{eq:GRFgbAintro}.  Our initial derivation of this flow is based on a novel gauge-theoretical interpretation of pluriclosed flow in terms of generalized geometry.  As explained in \cite[\S 7.3.2]{GRFBook}, a pluriclosed metric can be understood as an integrable lifting of $T^{0,1}$ to a complexified Courant algebroid
$$
T^{0,1} \cong \overline{\ell} \subset E \otimes \CC.
$$
Building on this, we understand the relevant data of a complex structure on $M$, Hermitian metric $g$, and a principal connection $A$ in terms of an integrable lifting of $T^{0,1}$ to a complexified string algebroid. Crucially, the Bianchi identity equation for the pair $(\omega,A)$ is implied by the integrability of the lifting. Following recent developments on the relation between the Hull-Strominger system and the theory of vertex algebras \cite{AAG2}, in \S \ref{s:PCFSA} we associate a canonical infinitesimal complex symmetry to such a lifting,
$$
\zeta^{\overline{\ell}} = \i \rho_B + \i \IP{F_A,\Lambda_\omega F_A}_\mathfrak{k} + \frac{\i}{2} \Lambda_\omega F_A,
$$
which in turn defines a flow of liftings, which we refer to as \emph{(string algebroid) pluriclosed flow}
\begin{gather}\label{f:CHEabs}
\dt \overline{\ell} = -  \zeta^{\overline{\ell}} \cdot \overline{\ell}.
\end{gather}
These equations exhibit a striking structural parallel with Donaldson HYM flow for connections on a bundle \cite{DonaldsonHYM} (see Proposition \ref{p:PCFholU}), and can be indeed regarded as an extension of such flow to higher gauge theory following the recent work \cite{Tellez}. As in \cite{AAG2}, the quantity $\zeta^{\overline{\ell}}$ is defined in terms of a local potential solving the so called \emph{D-term equation} (see Lemma \ref{lem:braketsum0}). This potential is strongly reminiscent of \emph{Getzler's element} $\rho$ in the construction of representations of the $N=2$ superconformal vertex algebra from Manin pairs \cite{Getzler}, and hence we expect an interesting relationship between the flow and mirror symmetry \cite{AAG2}.

When the flow \eqref{f:CHEabs} is unwound in terms of classical tensors, the data is determined by a triple $(\gw_t,\beta_t, h_t)$, given by one-parameter families of Hermitian metrics and $(2,0)$ forms on $M$ and reductions on $P^c$, which satisfy $d d^c \gw + \IP{F_h \wedge F_h}_{\mathfrak k} = 0$ and evolve by
\begin{gather} \label{f:PCFSAbis}
\begin{split}
    \dt \omega & = -\rho_B(\omega)^{1,1} + \i\IP{S^h_g, F_h}_{\mathfrak{k}},\\
\dt \beta & = - \rho_B(\omega)^{2,0} - \frac{\i}{2} \langle \alpha \wedge \partial^h(S^h_g) \rangle_{\mathfrak{k}},\\
h^{-1} \dt h  & = - S^h_g,
\end{split}
\end{gather}
where $A^{h_t} = A^{h_0} + \alpha_t$. Informally, this is a generalization of pluriclosed flow further coupled to HYM flow, where instead of the pluriclosed condition the flow rather preserves the Bianchi identity equation.  We give several  further geometric interpretations of this equation.  First, to connect to the discussion above, we show in Proposition \ref{prop:UPCFRiemannian} that solutions to string algebroid pluriclosed flow are gauge-equivalent to solutions of string algebroid generalized Ricci flow, generalizing the result of \cite{PCFReg}.  

Furthermore, building on \cite{JFS} we show that solutions to string algebroid pluriclosed flow are equivalent to solutions of a coupled HYM flow of Hermitian metrics on an associated \emph{holomorphic string algebroid} \cite{garciafern2018holomorphic}.  As explained in \cite{garciafern2020gauge}, the data of $(\gw, h)$ satisfying the Bianchi identity determines a natural holomorphic structure on the bundle
\begin{align*}
    \mathcal Q = T^{1,0} \oplus \ad P^c \oplus T^*_{1,0}.
\end{align*}
We show in Proposition \ref{p:PCFhol} that the data of a solution $(\gw_t,\beta_t,h_t)$ to string algebroid pluriclosed flow is equivalent to a one-parameter family $\mathbf{G}_t$ of (pseudo)Hermitian metrics on the bundle $\mathcal Q$ which satisfies the coupled Donaldson HYM flow
\begin{equation}\label{eq:CDYHMintro}
    \mathbf{G}^{-1} \dt \mathbf{G} = - S_g^{\mathbf{G}}.
\end{equation}
This formulation of the flow is the main tool used for some of the analytic results below.

\subsection{Dimensional reduction}

In addition to the different faces of pluriclosed flow arising from generalized geometry, there is a dimensional reduction principle which directly relates flow lines of generalized Ricci/pluriclosed flow on string algebroids to corresponding flow lines on \emph{exact} Courant algebroids, i.e. the systems (\ref{f:GRF}) and (\ref{f:PCF}), defined on an associated principal bundle.  Beyond a formal construction, this correspondence allows us to immediately apply results from the study of (\ref{f:GRF}) and (\ref{f:PCF}), such as various monotonicity formulas, to equations (\ref{eq:GRFgbAintro}) and (\ref{f:PCFSAbis}).  To begin, fix the data $K \to P \to M$ as above and $(g_t, b_t, A_t)$ a solution to (\ref{eq:GRFgbAintro}).  Given this, define a metric and three-form on $P$ via
\begin{align*}
\gtot =&\ p^* g + g_K(A, A), \qquad \Htot = p^* H - CS(A),
\end{align*}
where $g_K$ is a bi-invariant metric on $K$ associated to the pairing on $\mathfrak k$, and $CS(A)$ is the Chern-Simons form associated to $A$.  Note that this metric will only be Riemannian in the case the pairing on $\mathfrak k$ is negative-definite.  We show in Proposition \ref{p:cGRFtoGRF} that $(\gtot_t, \Htot_t)$ is a solution to generalized Ricci flow. Said differently, the generalized Ricci flow preserves the ansatz for $(\gtot,\Htot)$ above, and reduces to a solution to (\ref{eq:GRFgbAintro}). A precursor of this result for generalized Ricci flat metrics can be found in \cite{BarHek}. Given the results described in the previous subsection, this shows that every string algebroid pluriclosed flow is naturally interpreted as gauge-fixed generalized Ricci flow on a principal bundle.

Directly interpreting string algebroid pluriclosed flow as a solution of pluriclosed flow on a principal bundle is more subtle.  Here we further assume that $K$ is a compact even dimensional Lie group endowed with an integrable left-invariant complex structure $K$.  Using this, the data $(\gw, h)$ defines a natural Hermitian structure $(\bar{\gw}, \bar{J})$ on the total space $P$.  We show that along a string algebroid pluriclosed flow $(\gw_t, \gb_t, h_t)$, the induced complex structure $\bar{J}$ is constant when the structure group is abelian.  Furthermore, any flow line canonically induces a solution to pluriclosed flow on $(P,\Jtot)$. In the case of non-abelian structure group, one still induces pluriclosed flow on $(P, \Jtot)$, although it will no longer be symmetric with respect to the given $K$-action (cf. Proposition \ref{prop:redPCF}).

\subsection{Regularity of uniformly parabolic solutions}

Using the connections between the Hull-Strominger system and pluriclosed flow described in the previous sections, we are able to prove a sharp regularity result for solutions to the Hull-Strominger system, as well as the associated parabolic flows described above.  More precisely, we prove that uniformly parabolic solutions of the string algebroid pluriclosed flow satisfy scale-invariant decay estimates of all higher order derivatives:

\begin{theorem} \label{t:EKthmintro} (cf. Theorem \ref{t:EKthm}) Let $(\gw_t, \gb_t, h_t)$ a solution to \eqref{f:PCFSAbis} defined on $[0, \tau), \tau \leq 1$, such that the initial condition satisfies the Bianchi identity 
\begin{equation}\label{eq:BIintro}
d d^c \gw_0 + \IP{F_{h_0} \wedge F_{h_0}}_{\mathfrak k} = 0.  
\end{equation}
Suppose there are background metrics $\til{\gw}$, $\til{h}$ and $\gl, \gL > 0$ such that for all $(x,t)$,
\begin{align*}
\gl^{-1} \til{\gw} \leq \gw \leq \gL \til{\gw}, \qquad \brs{\gb}_{\til{\gw}} \leq \gL, \qquad \gl^{-1} \til{h} \leq h \leq \gL \til{h}.
\end{align*}
Given $k \in N$ there exists a constant $C = C(n, k, \gl, \gL,\til{\gw},\til{h})$ such that
\begin{align*}
\sup_{M \times [0,\tau)} t \Phi_k (\gw,\til{\gw},h,\til{h}) \leq C.
\end{align*}
\end{theorem}

\noindent The quantity $\Phi_k$ is a scale-invariant measure of the $C^{k+1}$ norm of $\gw$ and $h$ (cf. Definition \ref{d:Phiquantity}).  As a corollary one obtains a priori estimates on all derivatives of uniformly elliptic solutions of the Hull-Strominger system.

To contextualize Theorem \ref{t:EKthmintro}, first note that for instance when there is no auxiliary bundle, and the underlying metric $\gw$ is K\"ahler, the result corresponds to higher regularity of uniformly parabolic solutions of K\"ahler-Ricci flow.  This estimate can be obtained by the method of Evans-Krylov \cite{EvansC2a,Krylov}, or using Yau's $C^3$-estimate \cite{YauCC}.  More generally, with no auxiliary bundle, $\gw$ is pluriclosed and the result corresponds to higher regularity for uniformly parabolic solutions of pluriclosed flow.  Already this is a genuine \emph{system} of equations, with no obvious convexity structure, so the prior methods do not apply, and in general such higher regularity for parabolic systems is known to fail \cite{DeGiorgiCE}.  The regularity in this setting was established in \cite{SBIPCF}, and in \cite{JFS,JordanStreets} the result was sharpened and given a precise geometric formulation.  In particular, in \cite{JFS} the authors gave a formulation of pluriclosed flow in as a HYM-type flow of a generalized Hermitian metric on a holomorphic Courant algebroid, building on fundamental observations of Bismut \cite{Bismut}.   The Chern connections associated to these generalized metrics satisfy a remarkably clean evolution equation (cf. \eqref{eq:CDYHMintro}) which leads directly to the higher regularity in a conceptually similar way to Yau's $C^3$-estimate.

In the more general setting of Theorem \ref{t:EKthmintro}, there are further issues to overcome. Given the relationship to pluriclosed flow discussed above, we may expect the arguments sketched in the preceding paragraph to extend to this setting.  Indeed, in the case the Lie algebra inner product is negative-definite a different approach to higher regularity is possible using a direct maximum principle argument, and this is explained in the next subsection.  However,
in many cases of interest the Lie algebra inner product is of mixed signature, and in these cases the direct maximum principle argument fails.  Here the proof proceeds by first establishing $C^{\ga}$ regularity of $h$.  We show that in local coordinates $h$ is a matrix subsolution of a uniformly parabolic equation, whilst its determinant precisely satisfies a heat equation.  These are the key ingredients of the Evans-Krylov proof of $C^{2,\ga}$ regularity of uniformly parabolic equations, with $h$ roughly playing the role of the Hessian of the solution, and so adapting this method gives the required $C^{\ga}$ estimate.  Using this, the proof concludes using a blowup/contradiction argument, where if the statement was false we produce a nontrivial uniformly parabolic solution $(\gw_{\infty}, h_{\infty})$ of the equation on $\mathbb C^n \times (-\infty, 0]$.  By the a priori regularity of $h$, it follows that the limiting metric $h_{\infty}$ is flat, and these terms drop out of the system.  Then, using that the Bianchi identity is preserved along the flow, it follows that the metrics $\gw_{\infty}$ are pluriclosed, and moreover satisfy pluriclosed flow.  By the prior regularity results for pluriclosed flow discussed above, it follows that $\gw_{\infty}$ is flat, giving the required contradiction.

\subsection{Global existence and convergence results}

The proof of Theorem \ref{t:EKthmintro} hints at the subtle role played by the Lie algebra inner product.  Furthermore, as indicated above, the corresponding geometric structures defined either on the holomorphic string algebroid $\mathcal{Q}$ or the principal bundle $P$ will also have mixed signature depending on that of $\IP{,}_{\mathfrak k}$.  We now turn our focus on the case when this Lie algebra pairing is negative-definite, and correspondingly the relevant geometric structures on $\mathcal{Q}$ or $P$ are positive-definite.  As we will see, in this setting the generalized Ricci flow and pluriclosed flow acquire significantly more structure, leading to a much deeper analytic theory.

The first main result is a different higher regularity theorem related to Theorem \ref{t:EKthmintro}. Recall that any solution $(\omega_0,\beta_0,h_0)$ of the Bianchi identity \eqref{eq:BIintro} determines a holomorphic string algebroid $\mathcal{Q}$, and that any flow line $(\gw_t,\beta_t,h_t)$ of \eqref{f:PCFSAbis} with this initial condition induces a family of Hermitian metrics $\mathbf{G}_t = \mathbf{G}(\omega_t,\beta_t,h_t)$ on $\mathcal{Q}$ (see Proposition \ref{p:PCFhol}). Here the quantity $\Phi_k$ is a scale-invariant measure of the $C^{k+1}$ norm of $\mathbf{G}_t$ (cf. Definition \ref{d:GGPhiquantity}). 

\begin{theorem} \label{t:GGEKintro} (cf. Theorem \ref{t:GGEK})
Assume that the pairing $\IP{,}_{\mathfrak k}$ is negative-definite. Let $(\gw_t,\beta_t,h_t)$ a solution to \eqref{f:PCFSAbis} defined on $[0, \tau), \tau \leq 1$, such that the initial condition satisfies the Bianchi identity \eqref{eq:BIintro}, with associated holomorphic string algebroid $\mathcal{Q}$. Suppose that $\mathcal{Q}$ admits a background metric $\til{\mathbf{G}}$ and constants $\gl,\gL > 0$ so that, for all $(x,t)$, the family of Hermitian metrics $\mathbf{G}_t = \mathbf{G}(\omega_t,\beta_t,h_t)$ satisfies
\begin{align*} 
\gl \til{\mathbf{G}} \leq \mathbf{G} \leq \gL \til{\mathbf{G}}.
\end{align*}
Given $k \in \mathbb N$ there exists $C = C(k,n,\gl,\gL,\til{\mathbf{G}})$ such that
\begin{align*}
    \sup_{M \times [0,\tau)} t \Phi_k(\mathbf{G},\til{\mathbf{G}}) \leq C.
\end{align*}
\end{theorem}

%MGF: I MADE THE STATEMENT LESS TECHNICAL, HERE IS THE OLD ONE:
%\begin{theorem} \label{t:GGEKintro} (cf. Theorem \ref{t:GGEK}) Fix $\mathcal{Q} \to (M, J)$ a holomorphic string algebroid with negative-definite pairing $\IP{,}_{\mathfrak k}$. Given $\mathbf{G}_t$ a solution to pluriclosed flow \eqref{f:PCFSAbis} on $\mathcal{Q}$ defined on $[0,\tau), \tau \leq 1$, suppose there exists a background metric $\til{\mathbf{G}}$ and constants $\gl,\gL > 0$ so that, for all $(x,t)$, the family of Hermitian metrics $\mathbf{G}_t = \mathbf{G}(\omega_t,\beta_t,h_t)$ satisfies
%\begin{align*} 
%\gl \til{\mathbf{G}} \leq \mathbf{G} \leq \gL \til{\mathbf{G}}.
%\end{align*}
%Given $k \in \mathbb N$ there exists $C = C(k,n,\gl,\gL,\til{\mathbf{G}})$ such that
%\begin{align*}
%    \sup_{M \times [0,\tau)} t \Phi_k(\mathbf{G},\til{\mathbf{G}}) \leq C.
%\end{align*}
%\end{theorem}

\noindent In contrast to the proof of Theorem \ref{t:EKthmintro}, the proof of Theorem \ref{t:GGEKintro} is comparatively simpler.  Whereas the method of Theorem \ref{t:EKthmintro} exploits the delicate method of Evans-Krylov, Theorem \ref{t:GGEKintro} is a generalization of Yau's $C^3$ estimate \cite{YauCC} for the complex Monge-Ampere equation.  The geometric idea behind Yau's estimate is to study the norm of the Chern connection, measured using the solution metric, as opposed to a background metric.  This delicately chosen quantity satisfies a particularly clean differential inequality which leads by a direct maximum principle argument to the $C^3$ estimate.  Here the main point, generalizing the main ideas of \cite{JFS,JordanStreets,SBIPCF}, is to study the evolution of the norm of Chern connection of the generalized metric $\mathbf{G}$, measured again using the solution metric $\mathbf{G}$ as opposed to a background metric.  The structure of the pluriclosed flow equation for $\mathbf{G}$ again yields a number of key cancellations which lead to a particularly clean evolution equation.  The maximum principle then yields an a priori estimate for the Chern connection, giving the case $k = 0$ of the theorem.  The higher $k$ estimates are treated by a standard blowup argument using Schauder estimates.

Going further, we are able to establish global existence and convergence of the flow for arbitrary initial data on some string algebroids admitting Chern-flat metrics. Using this Chern-flat background metric, one obtains favorable evolution equations for the evolving metric $\mathbf{G}_t$, giving a priori bounds by the maximum principle.

%ALTERNATIVE FORMULATION
%\begin{theorem} \label{t:flatspaceconvintro} (cf. Theorem \ref{t:flatspaceconv})
%Assume that the pairing $\IP{,}_{\mathfrak k}$ is negative-definite. Let $(\gw_0,h_0)$ be a solution to the Bianchi identity \eqref{eq:BIintro} with associated holomorphic string algebroid $\mathcal{Q}$. Assume that the holomorphic vector bundle underlying $Q$ admits a background Hermitian metric $\mathbf{G}_F$ which is Chern-flat. Given a generalized Hermitian metric $(\omega_0,\beta_0,h_0)$ on $\mathcal{Q}$, the solution to pluriclosed flow \eqref{f:PCFSAbis} with this initial data exists on $[0, \infty)$ and converges to a generalized Hermitian metric $(\omega_\infty,\beta_\infty,h_\infty)$ on $\mathcal{Q}$, such that $\mathbf{G}_{\infty} = \mathbf{G}(\omega_\infty,\beta_\infty,h_\infty)$ is Chern-flat.
%\end{theorem}

\begin{theorem} \label{t:flatspaceconvintro} (cf. Theorem \ref{t:flatspaceconv}) Let $\mathcal{Q} \to (M, J)$ denote a holomorphic string algebroid with negative-definite pairing $\IP{,}_{\mathfrak k}$. Suppose that $\mathcal{Q}$ admits a Hermitian metric $\mathbf{G}_F$ which is Chern-flat.  Then, given a Hermitian metric $\mathbf{G}_0 = \mathbf{G}(\omega_0,\beta_0,h_0)$ on $\mathcal{Q}$, the solution to pluriclosed flow \eqref{f:PCFSAbis} with this initial data exists on $[0, \infty)$ and converges to a Chern-flat limit $\mathbf{G}_{\infty}$.
\end{theorem}

An interesting question is to classify string algebroids $\mathcal{Q}$ admitting Chern-flat Hermitian metrics. Following \cite{SGFLS,GFGM}, the existence of such a metric reduces to find a solution of the explicit coupled system \eqref{eq:CoupledGinstantonExp}.
%\begin{equation*}
%\label{eq:CoupledGinstantonExp}
%\begin{split}
%    R_{\nabla^{-}}-\mathbb{F}^{\dagger}\wedge \mathbb{F}  & = 0,\\
%    \nabla^{+,h}F_h  & = 0,\\
%    [F_h,\cdot] - \mathbb{F}\wedge \mathbb{F}^{\dagger} 
%    & = 0,\\
%	dd^c\omega + \langle F_h \wedge F_h \rangle_{\mathfrak{k}} 
%    & = 0,
%\end{split}
%\end{equation*}
%for suitable endomorphism-valued two-forms $\mathbb{F}^{\dagger}\wedge \mathbb{F}$ and $\mathbb{F}\wedge \mathbb{F}^{\dagger}$, where $R_{\nabla^{-}}$ denotes the curvature tensor of the metric compatible connection $\nabla^+ = \nabla + \tfrac{1}{2}g^{-1}d^c\omega$. 
Given a solution of this system, a choice of generalized Hermitian metric $\mathbf{G}_0$ on $\mathcal{Q}$ is equivalent to a Dolbeault version of the anomaly cancellation equation \eqref{eq:structuralGHMintro} for a pair $(\omega_0,h_0)$, defined in terms of Bott-Chern secondary characteristic classes (see Section \ref{sec:Aeppli} and Remark \ref{rmk:twist}).
%\begin{equation*}\label{eq:structuralGHMintro}
%[\partial(\omega_0 - \omega_F + R(h_0,h_F))] \in H^{2,1}_{\dbar}(M),
%\end{equation*}
%where $R(h_0,h_F)$ denotes a Bott-Chern secondary characteristic class associated to $(h_0,h_F)$ (see Section \ref{sec:Aeppli} and Remark \ref{rmk:twist}).  
A basic example occurs when the base manifold admits a Bismut-flat metric and the bundle is flat and we show in Corollary \ref{c:Bismutflatconv} that in this case one obtains convergence to a geometry of this type.  

\subsection{Conjectures on geometrization}

Pluriclosed flow was originally conceived in \cite{PCF} to construct canonical geometric structures on compact complex surfaces, with the eventual goal of finishing the classification problem for Class $\mbox{VII}_0$ surfaces (cf. \cite{st-geom}).  A fundamental point in this strategy is that every compact complex surface admits pluriclosed metrics \cite{Gauduchon1form}.  In higher dimensions the existence of pluriclosed metrics is obstructed, so a more general approach is required.  Here we replace the pluriclosed condition with the anomaly cancellation equation associated to an auxiliary bundle. As we have seen, this corresponds in a formal sense, and literally in some special settings, to the pluriclosed condition for an associated metric on a higher dimensional manifold fibering over $(M,J)$.  By making an appropriate choice of the bundle, one expects pluriclosed flow on string algebroids to be a more flexible tool for approaching geometrization problems on complex manifolds.  In \S \ref{s:geometrization} we 
record a number of further corollaries, including a scalar curvature monotonicity result and two gradient flow interpretations in the case of definite Lie algebra pairing, related to geometrization. In particular, we show in the presence of a holomorphic volume form that the flow \eqref{f:PCFSA} is the gradient flow of the \emph{dilaton functional}, as introduced in \cite{garciafern2018canonical}. Finally, for a suitable choice of bundle related to the geometry of embedded holomorphic curves and indefinite Lie algebra pairing, we discuss a conjectural picture for string algebroid pluriclosed flow and its relationship to the classification problem for complex threefolds with holomorphically trivial canonical bundle.

\section{Geometry of string algebroids} \label{s:SA}

In this section we recall fundamental aspects of the geometry of string algebroids.  We first recall basic points on the geometry of principal bundles, then use this to describe string algebroids, a special class of Courant algebroids which capture the geometry of the Hull-Strominger system.  Given this we derive an explicit formula for the Ricci tensor and the scalar curvature of a generalized metric and a divergence operator on a string algebroid.

\subsection{Background on principal bundles and connections}\label{s:gpb}

We first review basic differential geometry of principal bundles to fix notation and conventions. The material here is standard and the reader can find the details on any classical reference \emph{e.g.} \cite{KN1}.

Let $K$ be a real Lie group. Let $p \colon P \to M$ be a principal $K$-bundle, with right $K$-action, over a smooth manifold $M$. The surjection $p$ gives rise to an involutive vertical distribution $VP=\mathrm{ker} \; dp$ which fits into the exact sequence:
$$
0\longrightarrow VP \longrightarrow TP \longrightarrow p^*T \longrightarrow 0,
$$
where $T = TM$ is the tangent bundle of $M$. A connection $A$ on $P$ is given by a $\mathfrak{k}$-valued equivariant $1$-form on $P$ which restricts to the identity on $VP$. Given $A$, there is a $K$-equivariant splitting
$$
TP = VP \oplus \mathrm{ker} \; A.
$$
We call $\mathrm{ker} \; A$ the horizontal distribution, and we have associated projection maps, which by abuse of notation we denote by
$$
A : TP \longrightarrow VP, \; \; \; \; \;  A^\perp : TP \longrightarrow \ker A.
$$
The horizontal distribution $\ker A$ maps isomorphically to $T$ under $dp$ and, for any vector field $X\in \Gamma (T)$, we define a ($K$-invariant) lifted horizontal vector field on $P$ by
$$
A^\perp X=dp|_{\ker A}^{-1}(X).
$$
The connection $A$ induces parallel transport and covariant derivatives the bundles associated to $P$, in particular on the adjoint bundle over $M$
$$
\mathrm{ad }\; P = VP/K \cong P\times_K\mathfrak{k}.
$$
The covariant derivative of an $\mathrm{ad} \, P$-valued $k$-form $\alpha\in \Lambda^k(\mathrm{ad} \, P)$ is given by:
\begin{align} \label{f:covderivative}
d_A\alpha = d\alpha (A^\perp \cdot) = d\alpha + [A\wedge \alpha].
\end{align}
The curvature of $A$ %, which measures the non-involutivity of $H$, 
is defined by:
\begin{align}\label{f:curvature}
  F_A = dA(A^\perp \cdot) = dA+\tfrac{1}{2}[A\wedge A],  
\end{align}
and is naturally an element of $\Lambda^2(\mathrm{ad}\ P)$ which satisfies the Bianchi identity $d_A F_A =0$.

The gauge group of $P$ is defined as the subgroup of $K$-equivariant diffeomorphisms of $P$ which project to the identity on the base manifold
$$
\mathcal{G}(P)=\{ \overline{\phi} \in \operatorname{Diff}(P)^K \; | \; p \circ \overline{\phi} = p \}.
$$
Its Lie algebra is given by $K$-invariant vertical vector fields on $P$ and can be identified with the space of sections $\Gamma(\mathrm{ad} \, P)$.
By equivariance, $\langle\cdot, \cdot\rangle_\mathfrak{k}$ induces a pairing on sections of $\mathrm{ad} \, P$, which we shall denote in the same way. The group $\mathcal{G}(P)$ acts on the right on a connection:
$$
\overline{\phi} \cdot A = \overline{\phi}^*A = (d\overline{\phi})^{-1} \circ A \circ d\overline{\phi}.
$$
The infinitesimal right action of a section $s\in \Gamma(\mathrm{ad} \, P)$ on $A$, identified with the $K$-invariant field $X^s\in \Gamma(TP)$, is
$$
\mathcal{L}_{X^s} A  = d_A s.
$$

\subsection{String algebroids}\label{sec:background}

We recall here the necessary background material on Courant algebroids of string type, following \cite{garciafern2018holomorphic}. For the basic background on Courant algebroids we refer to \cite{GRFBook}.

Let $K$ be a real Lie group, endowed with a symmetric, adjoint-invariant, non-degenerate bilinear form on its Lie algebra $\mathfrak{k}$, which we denote
$$
\langle \cdot , \cdot \rangle_{\mathfrak{k}} \colon \mathfrak{k} \otimes \mathfrak{k} \to \RR.
$$
We fix a principal $K$-bundle $p \colon P \to M$ over a smooth manifold $M$ of dimension $m$. The key topological hypothesis is that the first Pontryagin class of $P$ with respect to $\langle \cdot , \cdot \rangle_{\mathfrak{k}} $ vanishes, that is,
$$
p_1(P):= [\langle F_A \wedge F_A \rangle_{\mathfrak{k}} ] = 0 \in H^4(M,\RR),
$$
where $A$ is a choice of connection on $P$. Given this, there exists a transitive Courant algebroid $(E,\left\langle\cdot,\cdot\right\rangle,\left[\cdot,\cdot\right],\pi)$ over $M$ which fits in a double extension of vector bundles
\begin{equation}\label{eq:Coustr}
\begin{split}
0 \longrightarrow T^* \overset{\pi^*}{\longrightarrow} E \overset{\rho}{\longrightarrow} A_P \longrightarrow 0,\\
0 \longrightarrow \ad P \longrightarrow A_P \overset{dp}{\longrightarrow} T \longrightarrow 0,
\end{split}
\end{equation}
where $A_{P} := TP/K$ denotes the Atiyah Lie algebroid of $P$, $\ad P:= \Ker dp \subset A_{P}$ is the adjoint bundle, and $T = TM$ is the tangent bundle of $M$. The map $\rho$ is bracket-preserving and induces an isomorphism of Lie algebroids
$$
\rho \colon A_E := E/(\Ker \pi)^\perp \to A_P
$$ 
such that its restriction to the subbundle $\ad E := \Ker \pi/(\Ker \pi)^\perp \subset A_E$ is an isomorphism of quadratic Lie algebras $\rho_{|\ad E} \colon \ad E \to (\ad P,\langle , \rangle_{\mathfrak{k}} )$. 

These Courant algebroids, associated to principal bundles with vanishing first Pontryagin class and enriched with the data of the morphism $\rho$, are called \emph{string algebroids} \cite{garciafern2018holomorphic}. More explicitly, given a string algebroid of the form \eqref{eq:Coustr}, which we will denote simply by $E$, a choice of isotropic splitting $\sigma \colon T \to E$ determines an isomorphism (see \cite[Section 3]{GF14})
\begin{equation}\label{eq:isomoEE0}
E \cong T \oplus \ad P \oplus T^*,
\end{equation}
a connection $A$ on $P$, and a three-form $H \in \Omega^3(M)$ such that
\begin{equation}\label{eq:bianchitrans}
dH = \langle F_A \wedge F_A \rangle_{\mathfrak{k}}.
\end{equation}
Via the isomorphism \eqref{eq:isomoEE0}, the pairing is given by
\begin{equation}\label{eq:pairing}
\langle X + r + \xi,Y + t + \eta\rangle = \frac{1}{2}(\eta(X) +
\xi(Y)) + \langle r,t \rangle_{\mathfrak{k}},
\end{equation}
the Dorfman bracket is 
\begin{equation}\label{eq:bracket10}
  \begin{split}
    [X+r+\xi,Y+t+\eta]  = {} & [X,Y] + L_{X}\eta - i_{Y}d\xi + i_{Y}i_{X}H\\
    & - [r,t] - F_A(X,Y) + (d_A)_Xt - (d_A)_Y r\\
    & + 2\langle d_A r,t \rangle_{\mathfrak{k}} + 2\langle i_X F_A ,t \rangle_{\mathfrak{k}} - 2\langle i_Y F_A,r \rangle_{\mathfrak{k}},
\end{split}
\end{equation}
the anchor map is the canonical projection $\pi(X + r + \xi) = X$, and the bracket-preserving morphism in \eqref{eq:Coustr} is  $\rho(X + r + \xi) = A^\perp X + X^r$.

The space of isotropic splittings on a string algebroid has a natural transitive action by the group of orthogonal bundle morphisms which restrict to the identity on $T^*$ and cover the identity on $A_P$. In the explicit presentation of $E$ above, such orthogonal transformations are given by \emph{$(b,a)$-transformations}, parameterized by pairs $(b,a)$, with $b \in \Omega^2(M)$ and $a \in \Omega^1(\ad P)$, and acting by
\begin{equation}\label{eq:BA}
e^{(b,a)}(X + r + \xi) = X + r + i_X a + \xi + i_X b - \langle i_X a,a \rangle_{\mathfrak{k}} - 2\langle a,r\rangle_{\mathfrak{k}}.
\end{equation}
Note that a $(b,a)$-transformation may not preserve the Dorfman bracket. In fact, if we denote by $[,]_{H,A}$ the bracket in \eqref{eq:bracket10}, one has
$$
e^{(b,a)}[e^{(-b,-a)}\cdot ,e^{(-b,-a)}\cdot ]_{H,A} = [\cdot,\cdot]_{H',A'},
$$
where
\begin{equation}\label{eq:AHnew}
\begin{split}
A' & = A + a,\\
H' & = H - db + 2 \langle a \wedge F_A \rangle_{\mathfrak{k}} + \langle a \wedge d_A a \rangle_{\mathfrak{k}} + \frac{1}{3}\langle a \wedge [a \wedge a] \rangle_{\mathfrak{k}}.
\end{split}
\end{equation}

\subsection{Generalized metrics and divergence operators}\label{sec:GRiem}

We next review some basic aspects of generalized Riemannian geometry on Courant algebroids of string type that we will need, following \cite{GF19,GRFBook,SeveraValach2}.  We fix a principal $K$-bundle $P$ over a smooth manifold $M$ and a string algebroid $E$ of the form \eqref{eq:Coustr}.

\begin{definition}\label{def:GmetricE}
A \emph{generalized metric} on $E$ is an orthogonal decomposition $E = V_+ \oplus V_-$, so that the restriction of $\langle \cdot,\cdot\rangle$ to $V_+$ is positive definite and that $\pi_{|V_+}:V_{+}\rightarrow T$ is an isomorphism.
\end{definition}

A generalized metric $V_+ \subset E$ is equivalent to a pair $(g,\sigma)$, where $g$ is a Riemann metric $g$ on $M$ and $\sigma \colon T \to E$ is an isotropic splitting (see e.g. \cite{GRFBook}). Alternatively, a generalized metric can be encoded in an orthogonal endomorphism $G \colon E \to E$ such that $G^2 = \Id$. The orthogonal decomposition $E = V_+ \oplus V_-$ is then recovered from the eigenbundles
$$
V_\pm = \Ker (G \mp \Id).
$$
We will use the following notation for the induced orthogonal projections 
\begin{align*}
\pi_\pm :=&\ \frac{1}{2}(G \pm \Id) \colon  E \longrightarrow V_\pm\\
a_{\pm} :=&\ \pi_{\pm} a.
\end{align*}

More explicitly, in the case of our interest, the isotropic splitting $\sigma \colon T \to E$ determined by $G$ induces an isomorphism $E \cong T \oplus \ad P \oplus T^*$ and hence an explicit string algebroid structure as in Section \ref{sec:background}, with bracket \eqref{eq:bracket10} (see \cite[Proposition 3.4]{GF14}), for a uniquely determined three-form $H \in \Omega^3$ and principal connection $A$ on $P$ satisfying \eqref{eq:bianchitrans}. Furthermore, via this identification we have
\begin{equation}\label{eq:Vpm}
V_+ = \{X + gX, X \in T\}, \quad V_- = \{X - gX + r, X \in T, r \in \ad P\},
\end{equation}
and also that
$$
G = \left( \begin{array}{ccc}
0 & 0& g^{-1} \\
0 & - \Id & 0 \\
g & 0 & 0 \end{array}\right),
$$
with orthogonal projections
$$
\pi_+(X + r + \xi) = \frac{1}{2}(X + gX), \qquad \pi_-(X + r + \xi) = \frac{1}{2}(X - gX) + r.
$$
It will be useful to introduce the notation $\sigma_\pm \colon T \to V_\pm \colon X \to \sigma_\pm(X):= X \pm gX.$

In order to introduce natural curvature quantities associated to a generalized metric $G$, the main difficulty is that there is no uniquely determined analogue of the Levi-Civita connection \cite{CSW,GF19}. Instead, there is a weak version of Koszul formula: a generalized metric $G$ on $E$ determines a pair of differential operators
\begin{equation}\label{eq:mixop}
D^+_- \colon \Gamma(V_+) \to \Gamma(V_-^* \otimes V_+), \qquad D^-_+ \colon \Gamma(V_-) \to \Gamma(V_+^* \otimes V_-),
\end{equation}
defined by
$$
D_{a_-}b_+ = [a_-,b_+]_+, \qquad D_{b_+}a_- = [b_+,a_-]_-,
$$
for any pair of sections $a_- \in \Gamma(V_-)$ and $b_+ \in \Gamma(V_+)$. They satisfy natural Leibniz rules with respect to the anchor map
\begin{equation}\label{eq:leibniz}
\begin{split}
D_{a_-}(fb_+) & = \pi(a_-)(f)b_+ + fD_{a_-}b_+,\\
D_{b_+}(fa_-) & = \pi(b_+)(f)a_- + fD_{a_+}a_-,
\end{split}
\end{equation}
for any smooth function $f \in C^\infty(M)$. 

To give an explicit formula for the operators in \eqref{eq:mixop}, we fix a generalized metric $G$ on $E$ and consider the associated isomorphism $E \cong T \oplus \ad P \oplus T^*$ and pair $(H,A)$ satisfying \eqref{eq:bianchitrans}. Define metric connections $\nabla^\pm$ on $(T,g)$ with totally skew-symmetric torsion by
\begin{equation}\label{eq:nablapm}
\nabla^+_XY = \nabla_XY + \frac{1}{2}g^{-1}H(X,Y,\cdot), \qquad \nabla^-_XY = \nabla_XY - \frac{1}{2}g^{-1}H(X,Y,\cdot),
\end{equation}
for $\nabla$ the Levi-Civita connection of $g$.

\begin{lemma}[\cite{GFRT17}]\label{lem:mixopexp} One has

\begin{equation}\label{eq:mixopexp}
  \begin{split}
    D_{b_+} a_- &= \sigma_- (\nabla^-_YX - g^{-1}\langle i_YF_A,r \rangle_{\mathfrak{k}} ) + (d_A)_Y r - F_A(Y,X),\\
    D_{a_-} b_+ &= \sigma_+ (\nabla^+_XY - g^{-1}\langle i_YF_A,r\rangle_{\mathfrak{k}} ),\\
  \end{split}
\end{equation}
where
\begin{equation}\label{eq:ab}
  \begin{split}
    a_- &= \sigma_-(X)+ r = X + r - gX,\\
    b_+ &= \sigma_+(Y) = Y + gY.
  \end{split}
\end{equation}
\end{lemma}

Even though the right notion of curvature tensor for a generalized metric is still unknown (cf. \cite{jurco2016courant}), one can construct a pair of generalized Ricci tensors associated to $G$ (see Section \ref{sec:GRic}). For this, it is required to consider \emph{divergence operators} on the string algebroid $E$ \cite{GF19}, which will keep track of geometrically relevant measures along the generalized Ricci flow.  These operators will be useful to tackle gauge-fixed versions of generalized Ricci flow in Section \ref{sec:GRF}. For simplicity, in the sequel we will assume that the manifold $M$ is oriented.

\begin{definition}\label{def:divE}
A divergence operator on $E$ is a map $\operatorname{div} \colon \Gamma(E) \to C^\infty(M)$ satisfying
$$
\operatorname{div}(fa) = f \operatorname{div}(a) + \pi(a)(f)
$$
for any $f \in C^\infty(M)$.
\end{definition}

Note that a generalized metric $G$ on $E$ has an associated \emph{Riemannian divergence} defined by
\begin{equation}\label{eq:div0}
\operatorname{div}^G(X + r + \xi) = \frac{L_X \mu_g}{\mu_g},
\end{equation}
where $\mu_g$ is the volume element of $g$. Alternatively, using that the Levi-Civita connection $\nabla$ of $g$ is torsion free, one can also write
\begin{equation}\label{eq:divtr}
\operatorname{div}^G(X + r + \xi) = \tr \nabla X := g(v_i,\nabla_{v_i} X),
%Proof of the claim:
%\frac{L_X \mu_g}{\mu_g}= \frac{di_{X} \mu_g}{\mu_g}- = v_i(X^i) = v_i(X^i) + \Gamma^i_{ki}X^k = g(v_i,v_i(X^j)v_j + \Gamma^j_{ik}X^kv_j) = \tr \nabla X
\end{equation}
where $\{v_i\}$ is a local orthonormal frame of $(T,g)$ and we use Einstein's summation convention. 

Following \cite{GRFBook,SeveraValach2}, we introduce next a compatibility condition for pairs $(V_+,\operatorname{div})$. Recall that the space of divergence operators on $E$ is affine, modelled on sections of $E^* \cong E$.

\begin{definition}\label{def:compatibleE}
Let $(G,\operatorname{div})$ be a pair given by a generalized metric $G$ and a divergence operator $\operatorname{div}$ on $E$. Define $\varepsilon \in \Gamma(E)$ by $\langle \varepsilon, \cdot \rangle := \operatorname{div}^G - \operatorname{div}$. We say that $(G,\operatorname{div})$ is a \emph{compatible pair} if $\varepsilon$ is an infinitesimal isometry of $G$, that is,
$$
[\varepsilon,V_+] \subset V_+.
$$ 
Furthermore, we say that $(G,\operatorname{div})$ is \emph{closed} if $\varepsilon \in \Gamma(T^*) \subset \Gamma(E)$.	
\end{definition}

In the following result we provide an explicit characterization of the previous conditions for the case of string algebroids. The proof is analogous to that of \cite[Lemma 2.50]{GRFBook} using the more complicated bracket expression \eqref{eq:bracket10} and is therefore omitted.

\begin{lemma}
A pair $(G,\operatorname{div})$ is compatible if and only
\begin{equation}\label{eq:infisometry}
L_Xg = 0, \qquad d_A z = - \iota_X F_A, \qquad d\varphi = i_XH - 2 \langle F_A,z \rangle,
\end{equation}
where $\varepsilon = X  + z + \varphi$ in the splitting determined by $G$. Furthermore, $\varepsilon$ is closed if and only if $X = z = 0$ and $d \varphi = 0$.
\end{lemma}

\subsection{Generalized Ricci tensor}\label{sec:GRic}

The first general definition of Ricci tensor in generalized geometry was provided in \cite{GF19}, using torsion-free generalized connections associated to a  pair $(G,\operatorname{div})$. Here, we embrace the simpler and more elegant characterization introduced in \cite{SeveraValach2} (see also \cite{jurco2016courant,SCSV}). These a priori different notions have been recently proved to be equivalent in \cite{CaRuPe}. In the sequel, we will identify $V_\pm \cong V_\pm^*$ via the natural isomorphism provided by the pairing on $E$.

\begin{definition}\label{def:GRic}
The Ricci tensors associated to a pair $(G,\operatorname{div})$
$$
\operatorname{Rc}_{G,\operatorname{div}}^+ \in V_- \otimes V_+, \qquad \operatorname{Rc}_{G,\operatorname{div}}^- \in V_+ \otimes V_-
$$
are defined by
\begin{equation}\label{eq:Rcpm}
\begin{split}
\operatorname{Rc}_{G,\operatorname{div}}^+(a_-,b_+) & = \operatorname{div}([a_-,b_+]_+) - \pi(a_-)(\operatorname{div}(b_+)) - \operatorname{tr}_{V_+}[[\cdot,a_-]_-,b_+]_+,\\
\operatorname{Rc}_{G,\operatorname{div}}^-(b_+,a_-) & = \operatorname{div}([b_+,a_-]_-) - \pi(b_+)(\operatorname{div}(a_-)) - \operatorname{tr}_{V_-} [[\cdot,b_+]_+,a_-]_-.
\end{split}
\end{equation}
The total Ricci tensor of a pair $(G,\operatorname{div})$ is defined by
$$
\operatorname{Rc}_{G,\operatorname{div}}:= \operatorname{Rc}_{G,\operatorname{div}}^+ - \operatorname{Rc}_{G,\operatorname{div}}^- \in \Gamma(E \otimes E).
$$
For the natural choice $\operatorname{div} = \operatorname{div}^G$, we will denote $\operatorname{Rc}_{G}^\pm = \operatorname{Rc}_{G,\operatorname{div}^G}^\pm$ (resp. $\operatorname{Rc}_{G} = \operatorname{Rc}_{G,\operatorname{div}^G}$) and call it the $\pm$-Ricci tensor (resp. total Ricci tensor) of the generalized metric $G$.
\end{definition}

Using the Leibniz rules \eqref{eq:leibniz}, combined with the Leibniz rule for the divergence operator, it is not difficult to see that $\operatorname{Rc}_{G,\operatorname{div}}^\pm$ in \eqref{eq:Rcpm} define tensors in $V_\mp \otimes V_\pm$, as claimed. The following result provides a basic structural property of the total Ricci tensor. We note that our proof is valid for arbitrary Courant algebroids.

\begin{lemma}\label{lem:Ricciskew}
Let $(G,\operatorname{div})$ be a pair given by a generalized metric $G$ and a divergence operator $\operatorname{div}$ on a Courant algebroid $E$. Define $\varepsilon \in \Gamma(E)$ by $\la \varepsilon, \cdot \ra:= \operatorname{div}^G - \operatorname{div}$. Then, the total Ricci tensor of $G$ is skew symmetric, that is,
\begin{equation}\label{eq:SkewG}
\operatorname{Rc}_{G}:= \operatorname{Rc}_{G}^+ - \operatorname{Rc}_{G}^- \in \Gamma(\Lambda^2 E).
\end{equation}
Furthermore, $\operatorname{Rc}_{G,\operatorname{div}} \in \Gamma(\Lambda^2 E)$ if and only if $(G,\operatorname{div})$ is a compatible pair.
\end{lemma}

\begin{proof}
Condition $\operatorname{Rc}_{G,\operatorname{div}} \in \Gamma(\Lambda^2 E)$ is equivalent to
$$
\operatorname{Rc}_{G,\operatorname{div}}^+(a_-,b_+) = \operatorname{Rc}_{G,\operatorname{div}}^-(b_+,a_-),
$$
for all for $a_- \in \Gamma(V_-)$ and $b_+ \in \Gamma(V_+)$. We first proof that this equality holds for $\operatorname{Rc}_{G}^\pm$.
%since we have
%\begin{align*}
%\operatorname{Rc}_{G}(a,b) + \operatorname{Rc}_{G}(b,a) & = \operatorname{Rc}_{G}^+(a_-,b_+) - \operatorname{Rc}_{G}^-(a_+,b_-) + \operatorname{Rc}_{G}^+(b_-,a_+) - \operatorname{Rc}_{G}^-(b_+,a_-)
%\end{align*}
By \cite[Proposition 2.3]{SeveraValach2}, we have
\begin{align*}
\operatorname{div}^G([a_-,b_+])  & = \pi(a_-)(\operatorname{div}^G(b_+)) - \pi(b_+)(\operatorname{div}^G(a_-)).
%\operatorname{div}^G([b_-,a_+]) & = \pi(b_-)(\operatorname{div}^G(a_+)) - \pi(a_+)(\operatorname{div}^G(b_-)).
\end{align*}
Let $\{e_i^+\}$ (resp. $\{e_i^-\}$) a local frame of $V_+$ (resp. $V_-$), and let $\{e^i_+\}$ (resp. $\{e^i_-\}$) the dual frame, that is, such that $\langle e_i^+,e^j_+ \rangle = \delta_i^j$. Then, from the formula above, we have
\begin{equation*}
\begin{split}
\operatorname{Rc}_{G}^+&(a_-,b_+) - \operatorname{Rc}_{G}^-(b_+,a_-)\\
=&\ \operatorname{div}^G([a_-,b_+]_+) - \pi(a_-)(\operatorname{div}^G(b_+)) - \langle e_+^i,[[e_i^+,a_-]_-,b_+]\rangle\\
&\ - \operatorname{div}^G([b_+,a_-]_-) + \pi(b_+)(\operatorname{div}^G(a_-)) + \operatorname{tr}_{V_-} [[\cdot,b_+]_+,a_-]_-\\
=&\ \operatorname{div}^G([a_-,b_+]) - \pi(a_-)(\operatorname{div}^G(b_+)) + \pi(b_+)(\operatorname{div}^G(a_-)) - \langle [b_+,e_+^i],[e_i^+,a_-]_-\rangle + \operatorname{tr}_{V_-} [[\cdot,b_+]_+,a_-]_-\\
=&\ - \langle [b_+,e_+^i],[e_i^+,a_-]_-\rangle + \operatorname{tr}_{V_-} [[\cdot,b_+]_+,a_-]_-.
\end{split}
\end{equation*}
The first part of the statement follows now from
\begin{equation*}
\begin{split}
- \langle [b_+,e_+^i],[e_i^+,a_-]_-\rangle & =  - \langle [b_+,e_+^i],e_j^-\rangle \langle e^j_- ,[e_i^+,a_-]\rangle \\
& =  \langle e_+^i,[b_+,e_j^-]\rangle \langle e^j_- ,[e_i^+,a_-]\rangle \\
& = \langle e^j_- ,[[b_+,e_j^-],a_-]\rangle \\
& = - \operatorname{tr}_{V_-} [[\cdot,b_+]_+,a_-]_-.
\end{split}
\end{equation*}
Finally, the proof follows from
\begin{equation*}
\begin{split}
\operatorname{Rc}_{G,\operatorname{div}}^+(a_-,b_+) & = \operatorname{Rc}_{G}^+(a_-,b_+) - \langle\varepsilon_+,[a_-,b_+]\rangle + \pi(a_-) \langle \varepsilon_+,b_+\rangle\\
& = \operatorname{Rc}_{G}^-(b_+,a_-) - \langle[\varepsilon_+,a_-],b_+\rangle,\\
& = \operatorname{Rc}_{G,\operatorname{div}}^-(b_+,a_-) + \langle[\varepsilon_-,b_+],a_-\rangle + \langle  a_-,[\varepsilon_+,b_+]\rangle,\\
& = \operatorname{Rc}_{G,\operatorname{div}}^-(b_+,a_-) + \langle[\varepsilon,b_+],a_-\rangle.
\end{split}
\end{equation*}

\end{proof}

\begin{remark}
Note that Lemma \ref{lem:Ricciskew} was independently recently observed in \cite[Theorem 18]{CaRuPe}, answering a question in \cite{GF19}.
\end{remark}

We provide next an explicit characterization of the generalized Ricci tensors in the case of string algebroids, as originally calculated in \cite{GF14} using Levi-Civita connections. A proof of the following formula using Definition \ref{def:GRic} does not seem to appear in the literature, and here we provide one for completeness. We use the following notations
\begin{equation}\label{f:sqtens}
\begin{split}
H^2(X,Y) & = \langle i_X H, i_Y H \rangle, \qquad F_A^2 = -\langle i_X F_A, i_Y F_A \rangle_{\mathfrak{k}}\\
F_A \ \lrcorner \ H & =  \textstyle \sum_{i_1< i_2} F_A(v_{i_1},v_{i_2})H(v_{i_1},v_{i_2},\cdot)
\end{split}
\end{equation}
where $\{v_i\}$ is a $g$-orthonormal frame and we use the inner product induced by $g$ and $\langle \cdot,\cdot \rangle_{\mathfrak{k}}$ on the corresponding space of tensors.

\begin{proposition}\label{prop:Ricciexp}
Let $(G,\operatorname{div})$ be a pair given by a generalized metric $G$ and a divergence operator $\operatorname{div}$ on a string algebroid $E$.
Define $\varepsilon \in \Gamma(E)$ by $\la \varepsilon, \cdot \ra:= \operatorname{div}^G - \operatorname{div}$. Via the isomorphism $E \cong T \oplus \ad P \oplus T^*$ provided by $G$, we can uniquely write 
$$
\varepsilon = \sigma_+(\varphi_+^\sharp) + z + \sigma_-(\varphi_-^\sharp)
$$ 
for $\varphi_\pm \in \Gamma(T^*)$ and $z \in \Gamma(\ad P)$. Then, one has 
\begin{equation}\label{eq:Ric+exp}
\begin{split}
\operatorname{Rc}_{G,\operatorname{div}}^+(a_-,b_+) = & \ i_Yi_X \left(\operatorname{Rc} - \frac{1}{4}H^2 - F_A^2 + \frac{1}{2} L_{\varphi_+^\sharp}g\right)\\
& + i_Yi_X \left(- \frac{1}{2}d^{\star}H + \frac{1}{2} d \varphi_+ - \frac{1}{2} i_{\varphi_+^\sharp} H \right)\\
& - i_Y \Big{\langle} d_A^{\star}F_A - F_A \lrcorner \ H  + i_{\varphi_+^\sharp} F_A,r\Big{\rangle}_{\mathfrak{k}},
\end{split}
\end{equation}
\begin{equation}\label{eq:Ric-exp}
\begin{split}
\operatorname{Rc}_{G,\operatorname{div}}^-(b_+,a_-) = & \ i_Xi_Y \left(\operatorname{Rc} - \frac{1}{4}H^2 - F_A^2 - \frac{1}{2} L_{\varphi_-^\sharp}g \right)\\
& + i_Xi_Y \left(\frac{1}{2}d^{\star}H - \frac{1}{2} d \varphi_- - \frac{1}{2} i_{\varphi_-^\sharp} H + \langle F_A,z \rangle_{\mathfrak{k}} \right)\\
& - i_Y \Big{\langle} d_A^{\star}F_A - F_A \lrcorner \ H  - d_Az  - i_{\varphi_-^\sharp} F_A,r\Big{\rangle}_{\mathfrak{k}},
\end{split}
\end{equation}
where $a_-,b_+$ are as in \eqref{eq:ab}.
\end{proposition}

\begin{proof}
Firstly, using Lemma \ref{lem:mixopexp} and the proof of Lemma \ref{lem:Ricciskew}, we note that
\begin{equation*}
\begin{split}
\operatorname{Rc}_{G,\operatorname{div}}^+(a_-,b_+) & = \operatorname{Rc}_{G}^+(a_-,b_+) - \langle[\varepsilon_+,a_-],b_+\rangle\\
& = \operatorname{Rc}_{G}^+(a_-,b_+) + \nabla^+\varphi_+(X,Y) - i_Y\langle i_{\varphi_+^\sharp}F_A,r\rangle_{\mathfrak{k}}\\
\operatorname{Rc}_{G,\operatorname{div}}^-(b_+,a_-) & = \operatorname{Rc}_{G}^-(b_+,a_-) - \langle[\varepsilon_-,b_+],a_-\rangle\\
& = \operatorname{Rc}_{G}^-(b_+,a_-) + i_Xi_Y(- \nabla^-\varphi_-^\sharp + \langle F_A,z \rangle_{\mathfrak{k}}) + i_Y \langle (d_A)_Y z + i_{\varphi_-^\sharp}F_A, r\rangle_{\mathfrak{k}}.
\end{split}
\end{equation*}
Formulae \eqref{eq:Ric+exp} and \eqref{eq:Ric-exp} will follow from the general formula
\begin{equation*}
\nabla^\pm \varphi = \frac{1}{2}L_{\varphi^\sharp}g + \frac{1}{2} d\varphi \mp \frac{1}{2}i_{\varphi^\sharp}H,
\end{equation*}
for arbitrary $\varphi \in \Gamma(T^*)$, once we calculate $\operatorname{Rc}_{G}^\pm$. By Lemma \ref{lem:Ricciskew} we have that
$$
\operatorname{Rc}_{G}^+(a_-,b_+) = - \operatorname{Rc}_{G}^-(b_+,a_-),
$$
and hence $\operatorname{Rc}_{G}^-$ is uniquely determined by $\operatorname{Rc}_{G}^+$. Therefore, it suffices to calculate the latter. For this, we choose a local orthonormal frame $\{v_i\}$ for $(T,g)$ and 
observe that the classical Ricci tensor, satisfies
\begin{equation*}
\begin{split}
\operatorname{Rc}(X,Y) = g(v_i,\nabla_{v_i}\nabla_XY) - X(g(v_i,\nabla_{v_i}Y)) -  g\Big{(}v_i,\nabla_{\nabla_{v_i}X}Y \Big{)},
\end{split}
\end{equation*}
Now, we calculate
\begin{equation*}
\begin{split}
\operatorname{tr}_{V_+}[[\cdot,a_-]_-,b_+]_+ =&\ \langle \sigma_+(v_i),[[\sigma_+(v_i),a_-]_-,b_+] \rangle\\
=&\ \langle \sigma_+(v_i),[\sigma_- \left(\nabla^-_{v_i}X - g^{-1}\langle i_{v_i}F_A,r \rangle_{\mathfrak{k}}\right) + (d_A)_{v_i} r - F_A(v_i,X),b_+] \rangle\\
=&\ g\Big{(}v_i,\nabla_{\nabla^-_{v_i}X - g^{-1}\langle i_{v_i}F_A,r \rangle_{\mathfrak{k}}}^+Y \Big{)} - \langle F_A(Y,v_i), (d_A)_{v_i} r - F_A(v_i,X) \rangle_{\mathfrak{k}}\\
=&\ g\Big{(}v_i,\nabla_{\nabla_{v_i}X}Y \Big{)}  + \frac{1}{4}H(v_i,v_j,Y)H(v_i,v_j,X) - \langle F_A(v_i,Y), F_A(v_i,X) \rangle_{\mathfrak{k}} \\
&\ - \frac{1}{2} g(v_i,\nabla_{v_j} Y)H(v_i,X,v_j) + \frac{1}{2}H(\nabla_{v_i}X,Y,v_i) \\
&\ - g(v_i,\nabla_{v_j}^+Y)\langle F_A(v_i,v_j),r \rangle_{\mathfrak{k}} + \langle F_A(v_i,Y), (d_A)_{v_i} r \rangle_{\mathfrak{k}}\\
=&\ g\Big{(}v_i,\nabla_{\nabla_{v_i}X}Y \Big{)} + \frac{1}{4}H^2(X,Y) + F_A^2(X,Y) + \frac{1}{2} H(X,\nabla_{v_j} Y,v_j) + \frac{1}{2}H(\nabla_{v_i}X,Y,v_i)\\
&\ - \langle F_A(\nabla_{v_j}Y,v_j),r \rangle_{\mathfrak{k}} - \frac{1}{2}H(v_j,Y,v_i)\langle F_A(v_i,v_j),r \rangle_{\mathfrak{k}} + \langle F_A(v_i,Y), (d_A)_{v_i} r \rangle_{\mathfrak{k}}.
\end{split}
\end{equation*}
Using the alternative expression $\operatorname{div}^G(a) = \tr \nabla \pi(a)$ for the Riemannian divergence (see \eqref{eq:divtr}), we also have that
\begin{equation*}
\begin{split}
\operatorname{div}^G& ([a_-,b_+]_+) - \pi(a_-)(\operatorname{div}^G(b_+))\\
=&\  g\Big{(}v_i,\nabla_{v_i}(\nabla^+_XY - g^{-1}\langle i_YF_A,r\rangle_{\mathfrak{k}})\Big{)} - X(g(v_i,\nabla_{v_i}Y))\\
=&\  g(v_i,\nabla_{v_i}\nabla_XY) - X(g(v_i,\nabla_{v_i}Y)) + \frac{1}{2}i_{v_i} \nabla_{v_i}(i_Yi_XH) - i_{v_i} \nabla_{v_i}\langle i_YF_A,r\rangle_{\mathfrak{k}}\\
=&\  g(v_i,\nabla_{v_i}\nabla_XY) - X(g(v_i,\nabla_{v_i}Y)) + \frac{1}{2}v_i(H(X,Y,v_i)) - \frac{1}{2}H(X,Y,\nabla_{v_i}v_i)\\
& - \langle (d_A)_{v_i}(F_A(Y,v_i)),r\rangle_{\mathfrak{k}} - \langle F_A(Y,v_i),(d_A)_{v_i}r\rangle_{\mathfrak{k}} + \langle F_A(Y,\nabla_{v_i} v_i),r\rangle_{\mathfrak{k}}.
\end{split}
\end{equation*}
Collecting the previous expressions, the proof now follows:
\begin{equation*}
\begin{split}
\operatorname{Rc}_{G}^+ & (a_-,b_+)\\
=&\ \operatorname{Rc}(X,Y) - \frac{1}{4}H^2 (X,Y) - F_A^2(X,Y)\\
&\ + \frac{1}{2} \Big{(}v_i(H(X,Y,v_i)) - H(X,Y,\nabla_{v_i}v_i) - H(X,\nabla_{v_j} Y,v_j) - H(\nabla_{v_i}X,Y,v_i)\Big{)}\\
&\ - \langle (d_A)_{v_i}(F_A(Y,v_i)),r\rangle_{\mathfrak{k}} - \langle F_A(Y,v_i),(d_A)_{v_i}r\rangle_{\mathfrak{k}} + \langle F_A(Y,\nabla_{v_i} v_i),r\rangle_{\mathfrak{k}}\\
&\ + \langle F_A(\nabla_{v_j}Y,v_j),r \rangle_{\mathfrak{k}}  - \langle F_A(v_i,Y), (d_A)_{v_i} r \rangle_{\mathfrak{k}} + \frac{1}{2}H(v_i,v_j,Y)\langle F_A(v_i,v_j),r \rangle_{\mathfrak{k}}\\
=&\ \operatorname{Rc}(X,Y) - \frac{1}{4}H^2 (X,Y) - F_A^2(X,Y) + \frac{1}{2} i_Yi_X i_{v_i}\nabla_{v_i}H\\
&  + i_Y \langle i_{v_i}(d_A)_{v_i}F_A,r\rangle_{\mathfrak{k}} + i_Y \langle F \lrcorner \ H ,r \rangle_{\mathfrak{k}}\\
=&\ i_Yi_X\Big{(}\operatorname{Rc} - \frac{1}{4}H^2 - F_A^2 - \frac{1}{2}d^{\star}H\Big{)} - i_Y \Big{\langle} d_A^{\star}F_A - F_A \lrcorner \ H ,r\Big{\rangle}_{\mathfrak{k}}.
\end{split}
\end{equation*}
\end{proof}

\subsection{Generalized scalar curvature}\label{sec:GScalar}

We recall briefly the definition of the generalized scalar curvature, which we will use in Section \ref{sec:Scalarmon}. This quantity was originally introduced in \cite{CSW,SeveraValach2} and more recently studied in \cite{SGFLS,GRFBook,GMolina,SCSV} via different methods. Since we are only concerned with the case of string algebroids, we shall follow closely the recent work \cite[Section 2.4]{SGFLS}. As an application, we will prove that generalized Ricci flat metrics have constant generalized scalar curvature.

To fix conventions, note that we use the tensorial norm on $k$-forms, defined by
$$
\alpha \wedge \star \alpha = \frac{1}{k!}|\alpha|^2 dV_g.
$$
More explicitly, given a $g$-orthonormal coframe $\{e^i\}$, we can write
$$
\alpha = \sum_{1 \leq i_1 < \ldots < i_k \leq n} \alpha_{i_1 \ldots i_k} e^{i_1} \wedge \ldots \wedge e^{i_k}, \quad 
|\alpha|^2 := \sum_{i_1,\ldots i_k = 1}^n \alpha_{i_1 \ldots i_k}^2.
$$
Furthermore, for the curvature of a connection $F_A$ the quantity $\brs{F_A}^2$ is the tensor norm defined using $g$ and $\IP{,}_{\mathfrak k}$, with a sign i.e.
$$
|F_A|^2 := -\sum_{i_1,i_2 = 1}^n \IP{ (F_A)_{i_1 i_2}, (F_A)_{i_1 i_2}}_{\mathfrak k}.
$$

\begin{defn}\label{def:GScalar}
Let $(G,\operatorname{div})$ be a pair given by a generalized metric and a divergence operator on a string algebroid $E$. The \emph{generalized scalar curvature} of  $(G,\operatorname{div})$ is defined by
\begin{equation}\label{eq:Genscalar}
\mathcal{S}^+ = R_g-\frac{1}{12}|H|^2 - \frac{1}{2}|F_A|^2 - 2 d^\star \varphi_+ - \left|\varphi_+\right|^2,
\end{equation}
where $R_g$ is the scalar curvature of $g$ and we use the unique decomposition of $\varepsilon \in \Gamma(E)$, for $\la \varepsilon, \cdot \ra:= \operatorname{div}^G - \operatorname{div}$, in terms of $\varphi_\pm \in \Gamma(T^*)$ and $z \in \Gamma(\ad P)$, as in Proposition \ref{prop:Ricciexp}. 
\end{defn}

For generalized Ricci flat metrics we have natural identities for the generalized scalar curvature, proved originally by the second author in his PhD disertation \cite{GMolina}. To establish them, we first need some technical computations:

\begin{lemma}\label{lem:tech_SUGRA}
        Let $P\rightarrow M$ be a principal $K$-bundle over a smooth manifold $M$, where $K$ has quadratic Lie algebra $(\mathfrak{k},\langle \cdot,\cdot\rangle_{\mathfrak{k}})$. Let $(g,H,A)$ be a triple consisting of a Riemannian metric, a $3$-form and a connection on $P$. Let $X_p\in T_pM$ at a point $p\in M$. Then:
    \begin{enumerate}
        \item There exists a smooth vector field $X$ extending $X_p$ such that $(\nabla X)_p = 0$, where $\nabla$ stands for the Levi-Civita connection of $g$.
        \item The following formulas hold at $p$: 
        \begin{align}
            \label{eq:tech_SUGRA_1}
            X(\mathrm{div} \; Y) = & \; \mathrm{div}(\nabla_X Y)-\mathrm{Rc}(X,Y),\\ 
        \label{eq:tech_SUGRA_2}d^\star(i_X\mathrm{Rc})  = & \; -\tfrac{1}{2}X(R_g),\\
            \label{eq:tech_SUGRA_3}
            d^\star(i_X H^2)  = & \; -\tfrac{1}{6}X(|H|^2)+\tfrac{1}{3}\textstyle \sum_{i,j,k} dH(X,v_i,v_j,v_k)H(v_i,v_j,v_k) \nonumber\\
            &\ + \textstyle \sum_{j,k} d^\star H(v_j,v_k)H(X,v_j,v_k),\\
            \label{eq:tech_SUGRA_4}
            d^\star(i_X F_A^2)  = & \;  -\tfrac{1}{4}X(|F_A|^2)-\textstyle \sum_i \langle d_A^\star F_A(v_i),F_A(X,v_i)\rangle_{\mathfrak{k}},
            \end{align}
 where $X$ is as in (1), $Y$ is any vector field, and $\{ v_i \}$ is a $g$-orthonormal basis.
\end{enumerate}
\end{lemma}

\begin{proof}
The existence of $X$ follows taking normal coordinates around the point $p$. Equation \eqref{eq:tech_SUGRA_2} is then a straightforward consequence of the Bianchi identity for $\Rm$. Formula \eqref{eq:tech_SUGRA_3} follows from the same computation as in \cite[Lemma 3.19]{GRFBook}. Note however the extra term in \eqref{eq:tech_SUGRA_3} since we are not assuming $H$ is closed. The last formula \eqref{eq:tech_SUGRA_4} follows again from analogous computations, using here that $d_A F_A = 0$. 
\end{proof}

We establish next the relationship between generalized Ricci flat metrics and the generalized scalar curvature.

\begin{proposition}\label{prop:RicciflatScalar}
Let $(G,\operatorname{div})$ be a pair given by a generalized metric and a divergence operator  such that $\operatorname{Rc}_{G,\operatorname{div}}^+ = 0$. % over an $m$-dimensional manifold $M$. 
Then, one has 
\begin{equation}
\label{eq:GonzalezMolinaformula}
    d \mathcal{S}^+ + d\varphi \ \lrcorner \ H = 0,
\end{equation}
where $\operatorname{div}^G - \operatorname{div} = \la \varepsilon, \cdot \ra $ and we set $\varphi := \varphi_+ = g(\pi \varepsilon_+, \cdot) \in \Lambda^1$. In particular, if $d\varphi = 0$ then $\mathcal{S}^+$ is locally constant and furthermore
\begin{equation}
\label{eq:GonzalezMolinaformulabis}
    d \left(  \frac{1}{6}|H|^2 + \frac{1}{2}|F_A|^2 - d^\star \varphi - |\varphi|^2 \right) = 0.
\end{equation}
\end{proposition}

\begin{proof}
First, by \eqref{eq:Ric+exp}, the condition of $\mathrm{Rc}^+_{G,\mathrm{div}}=0$ is equivalent to the system:
\begin{equation}\label{eq:GRflat}
\operatorname{Rc} - \frac{1}{4}H^2 - F_A^2 + \frac{1}{2} L_{\varphi^\sharp}g = 0, \quad d^\star H - d \varphi +  i_{\varphi^\sharp} H = 0, \quad d_A^\star F_A - F_A \lrcorner \ H + i_{\varphi^\sharp} F_A = 0.
\end{equation}
Then, taking the trace of the first equation in \eqref{eq:GRflat}, we get:   \begin{align}\label{eq:tr_Einstein_SUGRA}
        R_g-\tfrac{1}{4}|H|^2-|F_A|^2-d^\star \varphi = 0.
    \end{align}
   In the sequel, let $X$ be a vector field such that $(\nabla X)_p=0$ for some $p\in M$. Then, using \eqref{eq:tr_Einstein_SUGRA}, at the point $p$ we have:
    \begin{align*}
        X(R_g) & = X(\tfrac{1}{4}|H|^2+|F_A|^2)-X(\mathrm{div}(\varphi^\sharp))\\
        % & = X(\tfrac{1}{4}|H|^2+|F_A|^2)+\mathrm{Rc}(X,\varphi^\sharp)-\mathrm{div}(\nabla_X \varphi^\sharp)\\
        & = X(\tfrac{1}{4}|H|^2+|F_A|^2)+\mathrm{Rc}(X,\varphi^\sharp)-d^\star\left(i_X(\mathrm{Rc}-\tfrac{1}{4}H^2-F_A^2)\right),
    \end{align*}
    where in the last step we have used again the first equation in \eqref{eq:GRflat}. Then, using the formulae in Lemma \ref{lem:tech_SUGRA}, we continue the above computation to get:
    
    \begin{align*}
        X(R_g)  = & \; X\left(\tfrac{1}{4}|H|^2+|F_A|^2\right)+\mathrm{Rc}(X,\varphi^\sharp)+\tfrac{1}{2}X(R_g)+d^\star\left(i_X(\tfrac{1}{4}H^2+F_A^2)\right)\\
      %  = & \; X\left(\tfrac{1}{2}|H|^2+2|F_A|^2\right)+2\mathrm{Rc}(X,\varphi^\sharp)+d^\star\left(i_X(\tfrac{1}{2}H^2+2F_A^2)\right)\\
        = & \; X\left(\tfrac{1}{2}|H|^2+2|F_A|^2\right)+2\mathrm{Rc}(X,\varphi^\sharp)-\tfrac{1}{12}X(|H|^2)\\
        & + \tfrac{1}{6}\textstyle \sum_{i,j,k} dH(X,v_i,v_j,v_k)H(v_i,v_j,v_k)+ \tfrac{1}{2}\textstyle \sum_{j,k} d^\star H(v_j,v_k)H(X,v_j,v_k)\\
        & -\tfrac{1}{2}X(|F_A|^2)-2\textstyle \sum_i \langle d_A^\star F_A(v_i),F_A(X,v_i)\rangle_{\mathfrak{k}}\\
        = & \; X\left(\tfrac{5}{12}|H|^2+\tfrac{3}{2}|F_A|^2 \right)+2\mathrm{Rc}(X,\varphi^\sharp)+\tfrac{1}{2}\textstyle \sum_{j,k}d\varphi(v_j,v_k)H(X,v_j,v_k)\\
        &-\tfrac{1}{2}\textstyle\sum_{j,k} H(\varphi^\sharp,v_j,v_k)H(X,v_j,v_k) +\tfrac{1}{6}\textstyle \sum_{i,j,k}\langle F_A\wedge F_A\rangle_{\mathfrak{k}}(X,v_i,v_j,v_k)H(v_i,v_j,v_k)\\
        & -2\textstyle \sum_i \langle d_A^\star F_A(v_i),F_A(X,v_i)\rangle_{\mathfrak{k}}\\
        = & \; X\left(\tfrac{5}{12}|H|^2+\tfrac{3}{2}|F_A|^2 \right) +2\left(\mathrm{Rc}-\tfrac{1}{4}H^2 \right)(X,\varphi^\sharp)+i_X (d\varphi \ \lrcorner \ H) \\
        & -2\langle (d_A^\star F_A - F_A \lrcorner \ H)(v_i),F_A(X,v_i)\rangle_{\mathfrak{k}}\\
        = & \; X\left(\tfrac{5}{12}|H|^2+\tfrac{3}{2}|F_A|^2 \right) + 2(\mathrm{Rc}-\tfrac{1}{4}H^2-F_A^2)(X,\varphi^\sharp)+i_X (d\varphi \ \lrcorner \ H)\\
       % = & \; X\left(\tfrac{5}{12}|H|^2+\tfrac{3}{2}|F_A|^2 \right) - 2(\nabla_X \varphi)(\varphi^\sharp)+i_X (d\varphi \ \lrcorner \ H)\\
        = & \; X\left(\tfrac{5}{12}|H|^2+\tfrac{3}{2}|F_A|^2 \right)-X(|\varphi|^2)+i_X (d\varphi \ \lrcorner \ H),
    \end{align*}
    where, in the second line we have collected the terms in $X(R_g)$, and from the fifth line onwards, we use the equations in \eqref{eq:GRflat} and the Bianchi identity. Since the above computation is tensorial in $X$, it follows that:
\begin{align}\label{eq:pre_d_dilaton}
    X \left(R_g-\tfrac{5}{12}|H|^2-\tfrac{3}{2}|F_A|^2+|\varphi|^2\right)= i_X (d\varphi \ \lrcorner \ H).
\end{align}
The first part of the statement follows substracting \eqref{eq:pre_d_dilaton} from twice the expression obtained by taking exterior derivative in \eqref{eq:tr_Einstein_SUGRA}, and comparing with \eqref{eq:GonzalezMolinaformula}. In case $d\varphi = 0$, it follows that $d\mathcal{S}^+=0$, as claimed. The second part of the statement follows now substracting \eqref{eq:pre_d_dilaton} from the exterior derivative applied to \eqref{eq:tr_Einstein_SUGRA}.
\end{proof}

\section{Generalized Ricci flow on string algebroids} \label{s:GRFSA}

The goal of this section is to study the generalized Ricci flow on string algebroids and its relationship to the case of exact Courant algebroids, following \cite{GF19,GRFBook}.  We first recall the basic definitions in terms of generalized geometry and then derive the explicit flow against a fixed background for the associated metric, principal connection, and $b$-field.  We also obtain explicit formulas for the flow put into a general gauge by the action of a time-dependent family of Courant automorphisms, which is needed in what follows.

\subsection{Generalized Ricci flow}\label{sec:GRF}

\begin{definition}\label{def:GRF}
Given $E$ a string algebroid over a smooth manifold $M$, a one-parameter family $G_t$ of generalized metrics satisfies \emph{generalized Ricci flow} if
\begin{align} \label{eq:CourantGRF}
G^{-1} \dt G =&\ -2 \operatorname{Rc}_G.
\end{align}
Elementary computations using the definition of $G$ and Lemma \ref{lem:Ricciskew} show that this is a well-defined equation of sections of $(\Hom(V_-,V_+) \oplus \Hom(V_+,V_-)) \cap \Lambda^2 E$.
\end{definition}

Our next goal is to give an explicit formula for the generalized Ricci flow on string algebroids. For this, we fix a reference generalized metric $G_0$, and induced presentation $E \cong T \oplus \ad P \oplus T^*$ and pair $(H_0,A_0)$ satisfying the Bianchi identity (see Section \ref{sec:background})
\begin{equation}\label{eq:bianchitrans0}
dH_0 = \langle F_{A_0} \wedge F_{A_0} \rangle_{\mathfrak{k}}.
\end{equation}
Let $G_t$ be a one-parameter family of generalized metrics on $E$. As mentioned before, $G_t$ is equivalent to a Riemannian metric $g_t$ on $M$ and a choice of isotropic splitting $\sigma_t \colon T \to E$, which can be identified with a uniquely determined $(b,a)$-transformation $(-b_t,a_t)$ (see \eqref{eq:BA}), such that $\sigma(T) = e^{(b_t,-a_t)}(T)$. Therefore, in our setup 
\begin{equation}\label{eq:Ggbaexp}
G_t = e^{(b_t,-a_t)}\left( \begin{array}{ccc}
0 & 0& g_t^{-1} \\
0 & - \Id & 0 \\
g_t & 0 & 0 \end{array}\right)e^{(-b_t,a_t)}.
\end{equation}
with eigenbundles
\begin{equation}\label{eq:Vpmba}
\begin{split}
V_{+}^t =&\ e^{(b_t,-a_t)}\{X + g_t(X), X \in T\},\\
V_{-}^t =&\ e^{(b_t,-a_t)} \{X - g_t(X) + r, X \in T, r \in \ad P\}.
\end{split}
\end{equation}

\begin{proposition}\label{prop:GRFexp}
A one-parameter family of generalized metrics $G_t$ on the string algebroid $E$ satisfies generalized Ricci flow if and only if $(g_t,b_t,A_t)$ satisfies
\begin{equation}\label{eq:GRFgbA}
\begin{split}
\dt g & = -2\operatorname{Rc} + \frac{1}{2} H^2 + 2 F_A^2,\\
\dt b & = - d^{\star}H - \langle a \wedge ( d_A^{\star}F_A - F_A \lrcorner H) \rangle_{\mathfrak{k}},\\
\dt A & = - d_A^{\star}F_A + F_A \lrcorner \ H,
\end{split}
\end{equation}
where $A_t = A_0 + a_t$ and
\begin{equation}\label{eq:Htdef}
H_t := H_0 + db + 2 \langle a_t \wedge F_{A_0} \rangle_{\mathfrak{k}} + \langle a_t  \wedge d_{A_0}a_t \rangle_{\mathfrak{k}} + \frac{1}{3}\langle a_t \wedge [a_t \wedge a_t] \rangle_{\mathfrak{k}}.
\end{equation}
\end{proposition}

\begin{proof}
By Lemma \ref{lem:Ricciskew}, for the first part of the statement it suffices to prove that
$$
\dot G_t \circ \pi_-^t = - 2 \operatorname{Rc}_{G_t}^+.
$$
where $\dot G_t := \dt G_t$. To see this, first we note that
\begin{align*}
C_t & := e^{(-b_t,a_t)} \left(\dt e^{(b_t,-a_t)}\right) = - \left(\dt e^{(-b_t,a_t)} \right) e^{(b_t,-a_t)} = \left( \begin{array}{ccc}
0 & 0& 0 \\
- \dot a_t & 0 & 0 \\
\dot b_t - \langle a_t\wedge \dot a_t \rangle_{\mathfrak{k}} &  2 \langle \dot a_t, \rangle_{\mathfrak{k}} & 0 \end{array}\right),
\end{align*}
where we set $\dot b_t = \dt b_t$ and $\dot a_t = \dt a_t$ for simplicity. We also have
\begin{align*}
\dot G_t  
& = [e^{(b_t,-a_t)} C_t e^{(-b_t,a_t)}, G_t] + e^{(b_t,-a_t)}\left( \begin{array}{ccc}
0 & 0& - g_t^{-1} \dot g_t g_t^{-1} \\
0 & 0 & 0 \\
\dot g_t & 0 & 0 \end{array}\right)e^{(-b_t,a_t)}.
\end{align*}
where $\dot g_t = \dt g_t$. Then, letting
\begin{align*}
  \begin{split}
    a_- &= e^{(b_t,-a_t)}(X + r - g_tX), \qquad b_+ = e^{(b_t,-a_t)}(Y + g_tY),
  \end{split}
\end{align*}
we conclude
\begin{align*}
\langle \dot G_t a_-, b_+ \rangle & = \langle \dot g_t X +  g_t^{-1} \dot g_t  X,Y + g_tY \rangle - 2 \langle C_t (X + r - g_t X),Y + g_tY \rangle  \\
& = \dot g_t(X,Y) + (-\dot b_t + \langle a_t \wedge \dot a_t \rangle_{\mathfrak{k}})(X,Y) - 2 i_Y\langle \dot a_t, r \rangle_{\mathfrak{k}}.
\end{align*}
Finally, applying Proposition \ref{prop:Ricciexp} and using that the pair $(H_t,A_t)$ determined by $G_t$ is given by \eqref{eq:AHnew} (cf. \eqref{eq:Ggbaexp}), we have 
\begin{equation*}
\begin{split}
-2 \operatorname{Rc}_{G_t}^+(a_-,b_+) =&\ i_Yi_X\left(-2 \operatorname{Rc} + \frac{1}{2}H_t^2 + 2 F_{A_t}^2 + d^{\star}H_t\right) + 2 i_Y \Big{\langle} d_{A_t}^{\star}F_{A_t} - F_{A_t} \lrcorner \ H_t,r\Big{\rangle}_{\mathfrak{k}}.
\end{split}
\end{equation*}
Comparing components of the previous two equations gives the result.
\end{proof}

In the next result we calculate the evolution equation for the three-form $H$ along the generalized Ricci flow \eqref{eq:GRFgbA}.  From this, we will deduce that the evolution equation \eqref{eq:GRFgbA} for triples $(g,b,A)$ preserves the Bianchi identity \eqref{eq:bianchitrans}. Therefore, generalized Ricci flow on string algebroids can be regarded as an evolution equation for classical tensors, which furthermore is compatible with generalized geometry. More precisely, we have the following:

\begin{lemma}\label{lem:GRFH}
Let $P$ be a principal $K$-bundle over a smooth manifold $M$.  Assume that there exists a three-form $H_0$ on $M$ and a connection $A_0$ on $P$ satisfying the Bianchi identity \ref{eq:bianchitrans0}. 
Suppose $(g_t,b_t,A_t)$ is a solution to \eqref{eq:GRFgbA}, where $H_t$ is defined by \eqref{eq:Htdef} and $a_t = A_t - A_0$. Then $H_t$ satisfies
\begin{equation}\label{eq:GRFH}
\dt H = - d d^{\star}H - 2 \langle (d_A^{\star}F_A - F_A \lrcorner \ H)\wedge F_A \rangle_{\mathfrak{k}}.
\end{equation}
Conversely, if $(g_t,A_t)$ solve their respective equations in \eqref{eq:GRFgbA} where $H_t$ is a one-parameter family of three-forms satisfying \eqref{eq:GRFH}, and $(H_0,A_0)$ satisfies the Bianchi identity \eqref{eq:bianchitrans0}, then $(H_t,A_t)$ satisfies the Bianchi identity for all $t$.
\end{lemma}  

\begin{proof}
To calculate the evolution of the three-form $H_t$, we note that the proof of \cite[Lemma 3.23]{garciafern2018canonical} implies that
\begin{align*}
\dot H_t & = d \dot b_t + 2 \langle \dot a_t\wedge F_{A_t} \rangle_{\mathfrak{k}} - d  \langle a_t \wedge \dot  a_t \rangle_{\mathfrak{k}} =  - d d^{\star}H_t + 2 \langle \dot a_t \wedge F_{A_t} \rangle_{\mathfrak{k}},
\end{align*}
where $\dot H_t := \dt H_t$, and the first part of the statement follows. For the last part, it suffices to check that the time derivative of the Bianchi identity vanishes, but this follows from \eqref{eq:GRFH} since
\begin{align*}
\dt \left( d H_t - \langle F_{A_t} \wedge F_{A_t} \rangle_{\mathfrak{k}} \right) & = d \dot H_t - 2 \langle d_{A_t}\dot a_t \wedge F_{A_t} \rangle_{\mathfrak{k}}\\
& =  2 d \langle \dot a_t \wedge F_{A_t} \rangle_{\mathfrak{k}} - 2 d \langle \dot a_t \wedge F_{A_t}\rangle_{\mathfrak{k}} = 0,
\end{align*}
where we have used the Bianchi identity $d_{A_t}F_{A_t} = 0$ for the connections $A_t$.
\end{proof}

\begin{remark} The system (\ref{eq:GRFgbA}) is thus a natural coupling of Ricci flow to the heat flow for a three-form and the Yang-Mills flow for a principal connection.  In the case $M$ is a Riemann surface, the three-form $H$ drops out, and the system reduces to the Ricci-Yang-Mills flow \cite{StreetsRYMsurfaces,SRYM2,StreetsThesis,YoungSRYM,YoungThesis}. Note that here the flow is accompanied by the nontrivial flow of the two-form $b$, which is determined by the other data but necessary for the generalized geometric description.
\end{remark}

\subsection{Gauge-fixing}\label{sec:gaugefix}

We introduce next a gauge-fixing mechanism for generalized Ricci flow, which plays an important role in the applications to complex geometry. Recall that an automorphism of a string algebroid $E$, defined by a double extension \eqref{eq:Coustr}, is given by a commutative diagram of the form
\begin{equation}\label{eq:defstringiso}
	\xymatrix{
		0 \ar[r] & T^* \ar[r] \ar[d]^{d\phi^{-1}} & E \ar[r]^\rho \ar[d]^{\Phi} & A_P \ar[r] \ar[d]^{d\bar{\phi}} & 0,\\
		0 \ar[r] & T^* \ar[r] & E \ar[r]^{\rho} & A_{P} \ar[r] & 0,
	}
\end{equation}
where $\Phi \colon E \to E$ is a Courant algebroid automorphism (that is, a vector bundle isometry preserving the bracket and the anchor map) and $\bar{\phi} \in \operatorname{Diff}(P)$ is a $K$-equivariant diffeomorphism of $P$ which covers a diffeomorphism $\phi \in \operatorname{Diff}(M)$ on the base manifold $M$. 

Following \cite{GRFBook}, we introduce next a gauge-fixed version of generalized Ricci flow.

\begin{definition} Given $E$ a string algebroid, a one-parameter family of generalized metrics $G_t$ satisfies the \emph{gauge-fixed generalized Ricci flow} if there exists a one-parameter family of string algebroid automorphisms $\Phi_t \in \Aut(E)$ such that
$$
\til G_t = (\Phi_t)_*G_t = \Phi_t G_t \Phi_t^{-1}
$$ 
is a solution of the generalized Ricci flow. 
\end{definition}

Arguing as in the proof of \cite[Proposition 4.22]{GRFBook}, $G = G_t$ is a solution of the gauge-fixed generalized Ricci flow if and only if there exists a one-parameter family of infinitesimal automorphisms $\zeta_t \in \operatorname{Lie} \Aut(E)$ such that
\begin{equation}\label{eq:RFgaugefixedabs}
G^{-1}\dt G = - 2 \operatorname{Rc}_G - G^{-1}\zeta_t\cdot G,
\end{equation}
where $\zeta_t\cdot G$ denotes the infinitesimal action of $\zeta_t$ on $G_t$. In our next result we give an explicit formula for \eqref{eq:RFgaugefixedabs}. As in Proposition \ref{prop:GRFexp}, we fix a reference generalized metric $G_0$, induced presentation $E \cong T \oplus \ad P \oplus T^*$, and pair $(H_0,A_0)$ satisfying the Bianchi identity \eqref{eq:bianchitrans0}. Then, a one-parameter family of generalized metrics $G_t$ on $E$ is equivalent to a one-parameter family of triples $(g_t,b_t,A_t)$ as in \eqref{eq:Ggbaexp}.

\begin{proposition}\label{prop:GRFgaugeexp}
The one-parameter family of generalized metrics $G_t$ on the string algebroid $E$ satisfies gauge-fixed generalized Ricci flow \eqref{eq:RFgaugefixedabs} if and only if there exists a family of vector fields $X_t$, sections $s_t \in \Gamma(\ad P)$, and two-forms $B_t$ on $M$, satisfying
\begin{equation}\label{eq:LieAutEeq}
d \left(B_t + i_{X_t} H_0 - 2 \langle s_t , F_{A_0} \rangle_\mathfrak{k} \right) = 0,
\end{equation}
such that $(g_t,b_t,A_t)$ satisfies
\begin{equation}\label{eq:GRFgbAgauge}
\begin{split}
\dt g & = -2\operatorname{Rc} + \frac{1}{2} H^2 + 2 F_A^2 + L_{X}g,\\
\dt b & = - d^{\star}H - \langle a \wedge ( d_{A}^{\star}F_{A} - F_A \lrcorner \ H) \rangle_{\mathfrak{k}} + L_{X}b - B - \langle a \wedge( d_{A_0}s + i_{X}F_{A_0})  \rangle_{\mathfrak{k}}\\
\dt A & = - d_A^{\star}F_A + F_A \lrcorner \ H + d_{A}(s+ i_Xa) + i_{X}F_{A_t},
\end{split}
\end{equation}
where $A_t = A_0 + a_t$ and $H_t$ is defined as in \eqref{eq:Htdef}. Furthermore, assuming \eqref{eq:GRFgbAgauge}, the evolution of $H_t$ is given by
\begin{equation}\label{eq:GRFHgauge}
\dt H = - d d^{\star}H - 2 \langle (d_A^{\star}F_A - F_A \lrcorner \ H) \wedge F_A \rangle_{\mathfrak{k}} + L_XH.
\end{equation}
\end{proposition}

\begin{proof}
In the fixed splitting of $G_0$, a string algebroid automorphism $\Phi \in \Aut(E)$ is explicitly given by (see \cite[Corollary 4.2]{GFRT17})
$$
\Phi = f_{\bar{\phi}} \circ e^{(B,a^{\bar{\phi}})},
$$ 
where
$$
f_{\bar{\phi}}  = \left( \begin{array}{ccc}
d \phi & 0& 0 \\
0 & d \bar{\phi} & 0 \\
0 & 0 & d \phi^{-1} \end{array}\right).
$$
Here, $\bar{\phi}$ and $\phi$ are as in \eqref{eq:defstringiso} and $a^{\bar{\phi}} = \bar{\phi}^*A_0 - A_0 \in \Gamma(T^* \otimes \ad P)$. Furthermore, this data satisfies
$$
\phi^*H_0 = H_0 - dB + 2 \langle a^{\bar{\phi}}\wedge F_{A_0} \rangle_{\mathfrak{k}} + \langle a^{\bar{\phi}}\wedge d_{A_0}a^{\bar{\phi}} \rangle_{\mathfrak{k}} + \frac{1}{3}\langle a^{\bar{\phi}}\wedge [a^{\bar{\phi}} \wedge a^{\bar{\phi}}] \rangle_{\mathfrak{k}}.
$$
Then, its action on a generalized metric $G = G(g,b,A)$ is given by (cf. also \cite[Proposition 4.3]{GFRT17})
\begin{align*}
\Phi \cdot G & = \Phi G \Phi^{-1}\\
& = f_{\bar{\phi}} e^{(b',-a')} \left( \begin{array}{ccc}
0 & 0& g^{-1} \\
0 & - \Id & 0 \\
g & 0 & 0 \end{array}\right) e^{(-b',a')} f_{\bar{\phi}}^{-1}\\
& = f_{\bar{\phi}} e^{(b',-a')} f_{\bar{\phi}}^{-1}\left( \begin{array}{ccc}
0 & 0& (\phi_*g)^{-1} \\
0 & - \Id & 0 \\
\phi_*g & 0 & 0 \end{array}\right) f_{\bar{\phi}} e^{(-b',a')} f_{\bar{\phi}}^{-1}\\
&  = e^{(\phi_*b',-\bar{\phi}_* a')} \left( \begin{array}{ccc}
0 & 0& (\phi_*g)^{-1} \\
0 & - \Id & 0 \\
\phi_*g & 0 & 0 \end{array}\right) e^{(-\phi_*b',\bar{\phi}_* a')} = G(\phi_* g,\phi_*b',\bar{\phi}_* A),
\end{align*}
where
\begin{equation}\label{eq:gauge_pars}
    b' = b + B  -\langle a^{\bar{\phi}} \wedge a\rangle_\mathfrak{k}, \qquad a' = a - a^{\bar{\phi}},
\end{equation}
and we have used that
$$
\bar{\phi}_* a' = \bar{\phi}_* A - \bar{\phi}_* A_0 - \bar{\phi}_*(\bar{\phi}^*A_0 - A_0) = \bar{\phi}_* A - A_0.
$$
Given now a one-parameter family of automorphisms $\Phi_t \equiv (\phi_t,\bar{\phi}_t,\til B_t) \in \Aut(E)$ it determines uniquely a triple $\zeta_t \equiv (X_t,s_t,B_t)$ as in the statement (see \cite[Corollary 4.2]{GFRT17}). Note here that, more canonically, $(X_t,s_t)$ can be identified with the following $K$-invariant vector field on the total space of $P$
$$
V_t = s_t + A_0^\perp X_t \in \Gamma(TP)^K,
$$
where $A_0^\perp X_t$ denotes the horizontal lift of $X_t$ with respect to the connection $A_0$. By Lemma \ref{lem:Ricciskew}, the calculation of \eqref{eq:RFgaugefixedabs} reduces to compute
\begin{equation}\label{eq:gaugefixproof}
(\dot G_t) \circ \pi_-^t  = - 2 \operatorname{Rc}_{G_t}^+ - (\zeta_t\cdot G_t)\circ \pi_-^t.
\end{equation}
Note that infinitesimal action of $V_t$ on a connection $A = A_ 0 + a$ is given by
$$
V_t \cdot A = - d_A(s_t + i_{X_t}a) - i_{X_t} F_A.
$$
Then, using the same notation as in the proof of Proposition \ref{prop:GRFexp} we have
\begin{align*}
\langle (\zeta_t & \cdot G_t) a_-, b_+ \rangle_{\mathfrak{k}}\\
=&\ - L_{X_t}g_t(X,Y) + (L_{X_t}b - B_t - \langle V_t \cdot A_0 \wedge a_t \rangle_{\mathfrak{k}} + \langle a_t \wedge V_t \cdot A_t \rangle_{\mathfrak{k}})(X,Y) - 2 \langle V_t \cdot A_t, r \rangle_{\mathfrak{k}}\\
=&\ - L_{X_t}g_t(X,Y) - (-L_{X_t}b_t + B_t + \langle a_t \wedge \left( d_{A_t}(s_t + i_{X_t}a_t) + d_{A_0}s_t + i_{X_t}(F_{A_t} + F_{A_0})\right) \rangle_{\mathfrak{k}})(X,Y)\\
&\ + 2 \langle d_{A_t}(s_t + i_{X_t}a_t) + i_{X_t}F_{A_t}, r \rangle_{\mathfrak{k}},
\end{align*}
and therefore \eqref{eq:gaugefixproof} is equivalent to
\begin{equation*}
\begin{split}
\dt g_t & = -2\operatorname{Rc}_t + \frac{1}{2} H_t^2 - 2 \langle i_{v_i}F_{A_t},i_{v_i}F_{A_t} \rangle_{\mathfrak{k}} + L_{X_t}g_t,\\
\dt b_t - \langle a_t\wedge \dot a_t \rangle_\mathfrak{k} & = -d^{\star}H_t + L_{X_t}b_t - B_t - \langle a_t \wedge \left( d_{A_t}(s_t + i_{X_t}a_t) + d_{A_0}s_t + i_{X_t}(F_{A_t} + F_{A_0}) \right)\rangle_{\mathfrak{k}},\\
\dt A_t & = - d_{A_t}^{\star}F_{A_t} + F_{A_t} \lrcorner \ H_t + d_{A_t}(s_t + i_{X_t} a_t) + i_{X_t}F_{A_t}.
\end{split}
\end{equation*}
From this, we get
\begin{equation*}
\begin{split}
\dt b_t & = - d^{\star}H_t + L_{X_t}b_t - B_t - \langle a_t \wedge \left( d_{A_t}(s_t + i_{X_t}a_t) + d_{A_0}s_t + i_{X_t}(F_{A_t} + F_{A_0}) - \dot a_t \right)\rangle_{\mathfrak{k}},\\
& = - d^{\star}H_t + L_{X_t}b_t - B_t - \langle a_t \wedge \left( d_{A_t}^{\star}F_{A_t} - F_{A_t} \lrcorner \ H_t \right) \rangle_{\mathfrak{k}} + \langle a_t \wedge \left( d_{A_0}s_t + i_{X_t}F_{A_0}\right) \rangle_{\mathfrak{k}},
\end{split}
\end{equation*}
which proves formula \eqref{eq:GRFgbAgauge} assuming \eqref{eq:RFgaugefixedabs}. For the converse, note that any one-parameter family of triples $\zeta_t \equiv (X_t,s_t,B_t)$ as in the statement integrates to a one-parameter family of automorphisms $\Phi_t \in \Aut(E)$, similarly as in \cite[Lemma 3.9]{garciafern2020gauge}), and hence the first part of the statement follows from the previous calculations. 

Finally, to prove \eqref{eq:GRFHgauge} we note that the three-form $\til H_t$ corresponding to $\til G_t = (\Phi_t)_*G_t$ is given by
$$
\til H_t = (\phi_t)_* H_t.
$$
Hence, the claimed evolution for $H_t$ follows from Lemma \ref{lem:GRFH}.
\end{proof}

To finish this section, we interpret the gauge-fixed generalized Ricci flow for a special class of infinitesimal automorphisms in terms of the generalized Ricci tensor $\operatorname{Rc}_{G}^+$ associated to a particular choice of divergence operator (see Proposition \ref{prop:Ricciexp}). We need the following technical lemma. The proof follows from a tedious but straightforward calculation, and is omitted.

\begin{lemma}\label{lem:LieAut}
Let $G = G(g,b,A)$ be a generalized metric on $E$, with associated three-form $H$ and connection $A = A_0 + a$. Given $X \in \Gamma(TM)$ and $s \in \Gamma(\ad P)$ arbitrary sections and $\kappa \in \Lambda^2$ a closed two-form, we set
$$
B = \kappa -i_{X} H_0 + 2 \langle s,F_{A_0}\rangle_{\mathfrak{k}}.
$$  
Then, if we define
\begin{equation}\label{eq:Btilde}
\til B := B - L_{X}b + \langle a \wedge \left( d_{A}(s + i_{X}a) + d_{A_0}s + i_{X}(F_{A} + F_{A_0})\right) \rangle_{\mathfrak{k}},
\end{equation}
the following equality holds:
\begin{align*}
\til B = \kappa + d(-i_Xb - 2 \langle a , s \rangle_{\mathfrak{k}}  - \langle a , i_X a \rangle_{\mathfrak{k}}) - i_{X} H + 2 \langle s + i_Xa,F_{A}\rangle_{\mathfrak{k}}.
\end{align*}

\end{lemma}

\begin{remark}\label{rem:Lieexchange}
More invariantly, Lemma \ref{lem:LieAut} can be interpreted as the existence of a natural isomorphism between the Lie algebras of infinitesimal automorphisms of a pair of isomorphic Courant algebroids. For instance, for any infinitesimal automorphisms $\zeta \equiv (V,B)$ of $E$ (with bracket determined by $(H_0,A_0)$), where 
$$
V = s + A_0^\perp X = s + i_{X}a + A^\perp X,
$$
the two-form $\til B$ defined by \eqref{eq:Btilde} satisfies
$$
d(\til B + i_{X}H - 2\langle s + i_{X}a,F_{A} \rangle_{\mathfrak{k}}) = 0.
$$
The previous identity means that $(V,\til B)$ is an infinitesimal symmetry of the Dorfman bracket on $T \oplus \ad P \oplus T^*$ determined by $(H,A)$ (see \eqref{eq:bracket10}).
\end{remark}

We are ready to prove the main result of this section.

\begin{proposition}\label{p:GFGRFinGG+}Let $G_t$ be a one-parameter family of generalized metrics on $E$ and $X_t$ a one-parameter family of vector fields on $M$.  Defining
\begin{equation}\label{f:Bstrgauge}
B_t = d (i_{X_t} g_t + i_{X_t}b_t - \langle a_t , i_{X_t} a_t \rangle_{\mathfrak{k}}) -i_{X_t} H_0 - 2 \langle i_{X_t}a_t,F_{A_0}\rangle_{\mathfrak{k}},
\end{equation}
the data $\zeta_t \equiv (X_t,- i_{X_t} a_t,B_t)$ defines a one-parameter family of infinitesimal automorphisms of $E$.  Furthermore, $G_t$ satisfies the $\zeta_t$-gauge-fixed generalized Ricci flow \eqref{eq:RFgaugefixedabs} if and only if
\begin{equation}\label{eq:GRFdiv+}
\dt G \circ \pi_-  = - 2 \operatorname{Rc}_{G}^+(G,\operatorname{div}),
\end{equation}
where $\operatorname{div} = \operatorname{div}^{G} + \langle 2i_{X_t} g_t, \rangle$. More explicitly, \eqref{eq:GRFdiv+} is given by
\begin{equation}\label{eq:GRFdiv+exp}
\begin{split}
\dt g & = -2\operatorname{Rc} + \frac{1}{2} H^2 + 2 F_A^2 + L_{X}g,\\
\dt b & = - d^{\star}H - \langle a \wedge \left( d_A^{\star}F_A - F_A \lrcorner \ H - i_{X} F_A \right) \rangle_{\mathfrak{k}} - d (i_X g) +  i_{X} H,\\
\dt A & = - d_A^{\star}F_A + F_A \lrcorner \ H + i_{X} F_A,
\end{split}
\end{equation}

\end{proposition}

\begin{proof}
Given $\zeta_t \equiv (X_t,-i_{X_t}a_t,B_t)$ as in the statement, this defines a one-parameter family of infinitesimal automorphisms of $E$ since equation \eqref{eq:LieAutEeq} is trivially satisfied for all $t$.  Then, the $\zeta_t$-gauge-fixed generalized Ricci flow \eqref{eq:RFgaugefixedabs} is equivalent to
\begin{equation*}\label{eq:GRFgbAgaugeauxbis}
\begin{split}
\dt g_t & = -2\operatorname{Rc}_t + \frac{1}{2} H_t^2 + 2 F_{A_t}^2 + L_{X_t}g_t,\\
\dt b_t - \langle a_t \wedge \dot a_t \rangle_\mathfrak{k} & = - d^{\star}H_t + L_{X_t}b_t - B_t - \langle a_t \wedge \left( d_{A_0}s_t + i_{X_t}(F_{A_t} + F_{A_0}) \right) \rangle_{\mathfrak{k}},\\
\dt A_t & = - d_{A_t}^{\star}F_{A_t} + F_{A_t} \lrcorner \ H_t + i_{X_t}F_{A_t}.
\end{split}
\end{equation*}
From Lemma \ref{lem:LieAut}, it follows now that
\begin{align*}
\dt b_t & = - d^{\star}H_t - d (i_{X_t} g_t + i_{X_t}b_t - \langle a_t , i_{X_t} a_t \rangle_{\mathfrak{k}}) - d(-i_{X_t}b_t + \langle a_t , i_{X_t} a_t \rangle_{\mathfrak{k}}) + i_{X_t} H_t + \langle a_t \wedge\dot a_t \rangle_\mathfrak{k}\\
& = - d^{\star}H_t  - \langle a_t \wedge \left( d_{A_t}^{\star}F_{A_t} - F_{A_t} \lrcorner \ H_t - i_{X_t}F_{A_t} \right) \rangle_{\mathfrak{k}} -  d(i_{X_t}g_t) + i_{X_t} H_t.
\end{align*}
Thus, we conclude that, with the given choices, the $\zeta_t$-gauge-fixed generalized Ricci flow \eqref{eq:RFgaugefixedabs} is equivalent to \eqref{eq:GRFdiv+exp}. It remains to prove the equivalence between \eqref{eq:GRFdiv+exp} and \eqref{eq:GRFdiv+}. For this, writting $\operatorname{div} = \operatorname{div}^{G} - \langle \varepsilon, \rangle$, we have
$$
\varepsilon = - 2i_{X_t} g_t = 
- \sigma_+(X_t) + \sigma_-(X_t), %= - X_t - g(X_t) + X_t - g(X_t)
$$
and therefore the desired equivalence follows from Proposition \ref{prop:Ricciexp} and the proof of Proposition \ref{prop:GRFexp}.
\end{proof}

\section{Pluriclosed flow on string algebroids}\label{s:PCFSA}

In this section we introduce the pluriclosed flow on string algebroids and study geometric aspects of this evolution equation. We take a gauge theoretic approach and introduce the flow as a higher version of HYM flow for connections on a bundle, and eventually show it is gauge-equivalent to \eqref{eq:GRFgbA}. We will then prove that solutions to \emph{string algebroid pluriclosed flow} are equivalent to solutions of a coupled HYM flow of Hermitian metrics on an associated
holomorphic string algebroid and study its interplay with the Aeppli cohomology of the complex manifold. 

\subsection{Pluriclosed flow in the unitary gauge}\label{sec:UPCF}

In this section we introduce pluriclosed flow on a smooth string algebroid $E$ over a smooth manifold $M$, with even real dimension $\dim_\RR M = 2n$, working in a `unitary gauge'.  Let $E^c = E \otimes_\RR \CC$ be the complexification of $E$, endowed with the natural induced structure of a smooth complex string algebroid over $M$ (see \cite[Section 2.2]{garciafern2020gauge})
\begin{equation}\label{eq:Coustrcx}
\begin{split}
0 \longrightarrow T^* \otimes \CC \longrightarrow E^c \overset{\rho^c}{\longrightarrow} A_{P^c} \longrightarrow 0,\\
0 \longrightarrow \ad P \otimes \CC \longrightarrow A_{P^c} \overset{dp}{\longrightarrow} T \otimes \CC \longrightarrow 0,
\end{split}
\end{equation}
with anchor map $\pi^c \colon E^c \to T \otimes \CC$, where $P^c = P \times_K K^c$ is the associated principal $K^c$-bundle, for $K^c$ the complexification of $K$.

\begin{definition}\label{def:lifting}
A $(0,1)$-lifting of $T \otimes \CC$ to $E$ is a rank-$n$ isotropic subbundle
$$
\overline{\ell} \subset E^c
$$
such that $\pi^c(\overline{\ell}) \subset T \otimes \CC$ is the $(0,1)$-eigenbundle of an almost complex structure $J$ on $M$, that is,
$$
\pi^c(\overline{\ell}) = T^{0,1}_J \subset T \otimes \CC.
$$
We will say that $\overline{\ell}$ is \emph{integrable} if $\overline{\ell}$ is involutive with respect to the Dorfman bracket on $E^c$.
\end{definition}

Notice that $J$ in the previous definition is uniquely determined by $\overline{\ell}$, as it follows from $T^{1,0}_J = \overline{T^{0,1}_J}$. Consider the rank-$n$ subbbundle 
\begin{equation}\label{eq:Whor}
W = \{a \in \ell \oplus \overline{\ell} \; | \; \overline{a} = a \} \subset E,
\end{equation}
where $\ell$ is the conjugate bundle to $\overline \ell$ with respect to the natural real structure on $E^c$. Then, it is not difficult to see that the anchor map 
$$
\pi_{|W} \colon W \to T 
$$
induces an isomorphism, that is, $W$ determines a \emph{horizontal lift} of $T$, in the sense of \cite[Definition 4.10]{garciafern2020gauge}. Recall that such $W$ is equivalent to a real symmetric two-tensor $g$ and an isotropic splitting $\sigma \colon T \to E$, such that
$$
W = \{\sigma(X) + g(X) \; | \; X \in T\}.
$$
When $W$ is determined by a $(0,1)$-lifting of $\overline{\ell}  \subset E^c$, the isotropic condition on $\overline{\ell}$ implies furthermore that $g$ is of type $(1,1)$ with respect to the induced almost complex structure $J$ on $M$. Using this, we can provide the following characterization of $(0,1)$-liftings of $T \otimes \CC$ to $E$. The proof is straightforward from \cite[Lemma 3.12]{AAG2}.

\begin{lemma}[\cite{AAG2}]\label{lem:01lift}

A $(0,1)$-lifting $\overline{\ell}  \subset E^c$ is equivalent to a triple $(J,\omega,\sigma)$, where $J$ is an almost complex structure on $M$, $\omega \in \Omega_J^{1,1}$ is a real $(1,1)$-form with respect to $J$, and $\sigma \colon T \to E$ is an isotropic splitting. From this data, $\overline{\ell}$ is recovered via
$$
\overline{\ell} = e^{\i \omega} \sigma(T^{0,1}_J) \subset E^c.
$$
Furthermore, $\overline{\ell}$ is integrable if and only if 
$$
N_J = 0, \qquad H = -d^c_J\omega, \qquad F_A^{0,2_J} = 0,
$$
where $H$ and $A$ are the real three-form on $M$ and the connection on $P$ detetermined by $\sigma$, respectively.
\end{lemma}

We introduce next a natural notion of positivity for $(0,1)$-liftings based on the previous characterization. 

\begin{definition}

A $(0,1)$-lifting $\overline{\ell}  \subset E^c$ is said to be \emph{positive} if $g = \omega(,J)$ is a positive-definite Hermitian metric on $(M,J)$.
\end{definition}

Equivalently, one can state the previous definition in terms of the horizontal subbundle $W \subset E$ determined by $\overline{\ell}$ by saying that $W = V_+$, for $E = V_+ \oplus V_-$ a generalized metric in the sense of Definition \ref{def:GmetricE}. Notice that, under the positivity assumption, the generalized metric is uniquely determined by $\overline{\ell}$.

Fix an integrable positive $(0,1)$-lifting $\overline{\ell} \subset E^c$. Our next goal is to show that $\overline{\ell}$ has an associated infinitesimal symmetry 
$$
\zeta^{\overline{\ell}} \in \Lie \Aut E^c
$$
of $E^c$. We follow closely  \cite[Section 3.4]{AAG2}. We start by describing a local potential for $\zeta^{\overline{\ell}}$, that is, a section $\varepsilon^{\overline{\ell}} \in \Gamma(E^c)$ such that
$$
\zeta^{\overline{\ell}}  = [\varepsilon^{\overline{\ell}},],
$$
where $[,]$ denotes the Dorfman bracket on $E^c$. Given a section $e \in \Gamma(E^c)$, we will denote by
$$
a = a_\ell + a_{\overline \ell} + a_-
$$
the different pieces in the decomposition
$$
E^c = \ell \oplus \overline \ell \oplus V_-^c,
$$
where $V_+ \cong W$ is the generalized metric induced by $\overline \ell$ and $V_-^c:= V_- \otimes \CC$. Let $(z_1, \ldots, z_n)$ be holomorphic local coordinates for the integrable complex structure $J$ induced by $\overline{\ell}$. We consider the canonical holomorphic volume form
$$
\Omega_0 = dz_1 \wedge \ldots \wedge dz_n.
$$
Using the isomorphism $E \cong T \oplus \ad P \oplus T^*$ induced by $V_+$, we construct local isotropic frames
\begin{equation}\label{eq:localisoframe}
\begin{split}
\epsilon_j & = e^{-\i\omega}g^{-1}d\overline{z}_j  = g^{-1}d\overline{z}_j + d\overline{z}_j \in \ell,\\
\overline \epsilon_j & = e^{\i\omega}\dbar_j = \dbar_j + g \dbar_j  \in \overline \ell,
\end{split}
\end{equation}
where we use the compact notation $\dbar_j : = \frac{\partial}{\partial\overline{z}_j}$. Notice that the frames above satisfy, for all $j,k = 1, \ldots, n$, the following abstract properties
$$
[\overline \epsilon_j, \overline \epsilon_k ] = 0, \qquad \langle \epsilon_j, \overline \epsilon_k \rangle  = \delta_{jk}, \qquad \langle \epsilon_j, \epsilon_k \rangle  = \langle \overline \epsilon_j, \overline \epsilon_k \rangle = 0,
$$
which are key to the following result.

\begin{lemma}[\cite{AAG2}]\label{lem:braketsum0}
There exists a uniquely determined smooth local section $\varepsilon^{\overline{\ell}}$ of $\Ker \pi^c \subset E^c$ satisfying the \emph{D-term equation}
\begin{equation}
\label{eq:globalsumlocal}
\frac{1}{2} \sum_{j=1}^n\left[\overline{\epsilon}_j,\epsilon_j\right] = \varepsilon^{\overline{\ell}}_\ell - \varepsilon^{\overline{\ell}}.
\end{equation}
More explicitly, 
\begin{equation}\label{eq:varepsilonDterm}
\varepsilon^{\overline{\ell}} = d \log \|\Omega\|_\omega + \i \left(d^{\star}\omega - d^c\log \|\Omega\|_\omega \right) - \frac{\i}{2} \Lambda_\omega F_A.
\end{equation}
\end{lemma}

We are ready to prove the first main technical result of this section, building on \cite{AAG2}. Recall from the proof of Proposition \ref{prop:GRFgaugeexp} that an element $\zeta \in \Lie \Aut E^c$ which induces the zero vector field on $M$ can be identified, in the splitting of $E^c$ determined by $\overline{\ell}$, with a section $\zeta = s + B \in \Gamma((\ad P \oplus \Lambda^2) \otimes \CC)$ satisfying  
$$
d(B - 2 \langle s, F_A \rangle_{\mathfrak{k}}) = 0.
$$
Its action on sections $e = Y + t + \eta \in \Gamma(E)$ is explicitly given by
\begin{equation}\label{eq:actLieE}
\zeta \cdot e =    i_{Y}B + [s,t] + (d_A)_Y s - 2\langle d_A s,t \rangle_{\mathfrak{k}}.
\end{equation}
%Building on \cite{AAG2}, we now show that any such lifting induces a canonical infinitesimal symmetry of $E^c$ via the Dorfman bracket.

\begin{proposition}\label{prop:zetaell}
Associated to an integrable positive $(0,1)$-lifting $\overline{\ell} \subset E^c$ there is a well-defined infinitesimal symmetry 
$$
\zeta^{\overline{\ell}} \in \Lie \Aut E^c,
$$
given locally by $\zeta^{\overline{\ell}}  = [\varepsilon^{\overline{\ell}},]$. More explictly,
\begin{equation}\label{eq:zetaell}
\zeta^{\overline{\ell}} = \i \rho_B(\omega) + \i \langle F_A,\Lambda_\omega F_A \rangle_{\mathfrak{k}} + \frac{\i}{2} \Lambda_\omega F_A
\end{equation}
in the splitting of $E^c$ determined by $\overline{\ell}$.
\end{proposition}

\begin{proof}
The proof follows easily combining Lemma \ref{lem:braketsum0} with the explicit formula for the Dorfman bracket \eqref{eq:bracket10}, which gives (cf. \eqref{eq:actLieE})
\begin{equation*}
  \begin{split}
    [\varepsilon^{\overline{\ell}},Y+t+\eta]  = {} & - i_{Y}d(\i\left(d^{\star}\omega - d^c\log \|\Omega\|_\omega \right)) + \frac{\i}{2}[\Lambda_\omega F_A,t] + \frac{\i}{2}(d_A)_Y (\Lambda_\omega F_A)\\
    & - \langle d_A (\i\Lambda_\omega F_A),t \rangle_{\mathfrak{k}} + \langle i_Y F_A,\i\Lambda_\omega F_A \rangle_{\mathfrak{k}},\\
    = {} & i_{Y}(\i \rho_B(\omega) + \i\langle F_A,\Lambda_\omega F_A \rangle_{\mathfrak{k}}) \\
    & + \frac{\i}{2}[\Lambda_\omega F_A,t] + \frac{\i}{2}(d_A)_Y (\Lambda_\omega F_A) - \i\langle d_A (\Lambda_\omega F_A),t \rangle_{\mathfrak{k}}.
\end{split}
\end{equation*}
For the last equality we have used the following formula for the Bismut Ricci form of $\omega$ in local holomorphic coordinates:
$$
\rho_B(\omega) = - d(d^{\star}\omega - d^c \log \|\Omega_0\|_\omega).
$$
\end{proof}

Notice that, when $\overline{\ell}$ is non-integrable, it still makes sense to define a symmetry $\zeta^{\overline{\ell}}$ associated to $\overline{\ell}$ by formula \eqref{eq:zetaell}. This follows from the existence of the Bismut connection, and hence of the Bismut Ricci form, for almost Hermitian structures, as studied by Gauduchon in \cite{Gauduchonconn}. The main difference in the non-integrable case is that $\zeta^{\overline{\ell}}$ may not admit a local potential with respect to the Dorfman bracket. 

We next introduce the pluriclosed flow through a universal gauge theoretic construction.  Even though it is stated on a string algebroid $E$, it can be easily generalized to a much larger class of Courant algebroids.  Recall that a variation $\dot{\overline{\ell}}$ of a positive $(0,1)$-lifting $\overline{\ell}$ can be regarded as a map
\begin{equation}\label{eq:inftaction}
\dot{\overline{\ell}} \in \operatorname{Hom}(\overline{\ell}, \ell \oplus V_-^c).
\end{equation}
Furthermore, given $\zeta \in \Lie \Aut E^c$, we denote by $\zeta \cdot \overline{\ell}$ the infinitesimal action of $\zeta$ on the lifting $\overline{\ell}$, regarded as an element in \eqref{eq:inftaction}. The flow in the next definition shall be understood in the previous sense.

\begin{definition}\label{def:UPCF}
Given $E$ a string algebroid over an even dimensional smooth manifold $M$, a one-parameter family $\overline{\ell}_t \subset E^c$ of positive $(0,1)$-liftings of $T \otimes \CC$ to $E$ satisfies \emph{pluriclosed flow} if
\begin{align} \label{eq:UPCF}
\dt \overline{\ell}_t =&\  - \zeta^{\overline{\ell}} \cdot \overline{\ell}_t.
\end{align}
\end{definition}

We start our study of the flow with an interesting structural property, namely, that it preserves both integrability of the lifting and the induced complex structure on $M$ via Lemma \ref{lem:01lift}.

\begin{lemma}\label{lem:UPCF}
Let $\overline{\ell}_t$ be a flow line for \eqref{eq:UPCF}, defined on some interval $[0,t_0]$. Let $(J_t,\omega_t,\sigma_t)$ be the associated one-parameter family of triples as in Lemma \ref{lem:01lift}. Then, the following hold

\begin{enumerate}

\item if $\overline{\ell}_0$ is integrable, then so is $\overline{\ell}_t$ for all $t$,

\item $J_t = J_0$ for all $t$.
 
\end{enumerate}
 
\end{lemma}

\begin{proof}
Let $\Phi_t$ be the one-parameter family of automorphisms of $E^c$ generated by $\zeta^{\overline{\ell}_t}$, with $\Phi_0 = \Id_{E^c}$. Then, one has
$$
\dt \Phi_t(\overline{\ell}_t) = \Phi_t \Bigg{(}\zeta^{\overline{\ell}_t} \cdot \overline{\ell}_t + \dt \overline{\ell}_t\Bigg{)} = 0,
$$
by definition of the flow in $\eqref{eq:UPCF}$. Hence, it follows that
$$
\overline{\ell}_t = \Phi_t^{-1}(\overline{\ell}_0)
$$
for all $t$ and, since $\Phi_t$ preserves the Dorfman bracket and $\overline{\ell}_0$ is integrable by hypothesis, the first part of the statement follows. As for the second part, we note that, since $\pi^c(\zeta^{\overline{\ell}_t}) = 0$ (see Lemma \ref{lem:braketsum0}), $\Phi_t$ fits into a diagram of the form
\begin{equation*}
	\xymatrix{
		0 \ar[r] & T^* \ar[r] \ar[d]^{\Id} & E \ar[r]^\rho \ar[d]^{\Phi_t} & A_P \ar[r] \ar[d]^{d\bar{\phi}_t} & 0,\\
		0 \ar[r] & T^* \ar[r] & E \ar[r]^{\rho} & A_{P'} \ar[r] & 0,
	}
\end{equation*}
where $\bar{\phi}_t$ is in the gauge group of $P$ for all $t$ and hence, in particular, it covers the identity diffeomorphism on the base manifold $M$. From this, it follows that $\pi \Phi_t = \pi$ and hence
$$
T^{0,1}_{J_t} := \pi \overline{\ell}_t = \pi \Phi_t^{-1}(\overline{\ell}_0) = \pi(\overline{\ell}_0) = T^{0,1}_{J_0}
$$
for all $t \in [0,t_0]$.
\end{proof}

Our next goal is to give an explicit formula for pluriclosed flow \eqref{eq:UPCF}. For this, we fix a reference positive $(0,1)$-lifting $\bar \ell_0$ of $T \otimes \CC$ to $E$. The associated horizontal subspace $W_0$ via \eqref{eq:Whor} determines a generalized metric $G_0$, an induced presentation $E \cong T \oplus \ad P \oplus T^*$ and pair $(H_0,A_0)$ satisfying the Bianchi identity \eqref{eq:bianchitrans0} (see Section \ref{sec:background}). Let $\overline{\ell}_t$ be a one-parameter family of positive $(0,1)$-liftings. Let $(J_t,\omega_t,\sigma_t)$ be the associated one-parameter family of triples as in Lemma \ref{lem:01lift}. In the fixed isotropic splitting determined by $\bar \ell_0$, the isotropic splitting $\sigma_t$ can be identified with a uniquely determined $(b,a)$-transformation $(-b_t,a_t)$, such that $\sigma(T) = e^{(b_t,-a_t)}(T)$. Therefore, in our setup 
\begin{equation}\label{eq:elltba}
\bar{\ell}^t =\ e^{(\i\omega_t+b_t,-a_t)}(T^{0,1}_{J_t}), \qquad \ell^t = \ e^{(-\i\omega_t+b_t,-a_t)} (T^{1,0}_{J_t}).
\end{equation}
while the corresponding generalized metric is given by \eqref{eq:Vpmba}. Notice that, by Lemma \ref{lem:UPCF}, it follows that $J_t$ is constant in $t$.

\begin{proposition}\label{prop:UPCFexp}
The one-parameter family of positive $(0,1)$-liftings $\overline{\ell}_t \subset E^c$ satisfies pluriclosed flow \eqref{eq:UPCF} if and only if $(J_t,\omega_t,b_t,A_t)$ satisfies the following evolution equations
\begin{equation}\label{eq:CPFunitary}
\begin{split}
\dt \omega & = -\rho_B(\omega)^{1,1} - \langle F_A^{1,1}, \Lambda_\omega F_A \rangle_{\mathfrak{k}},\\
\dt b^{1,1} & = \frac{1}{2} \langle a \wedge  Jd_A(\Lambda_\omega F_A) \rangle_{\mathfrak{k}}^{1,1},\\
\dt b^{0,2} & = - \i \rho_B(\omega)^{0,2} - \i \langle F_A^{0,2}, \Lambda_\omega F_A \rangle_{\mathfrak{k}} + \frac{\i}{2}\langle a^{0,1} \wedge  \dbar^A(\Lambda_\omega F_A) \rangle_{\mathfrak{k}},\\
\dt A & = \frac{1}{2}Jd_A(\Lambda_\omega F_A),\\
\dt J & = 0,
\end{split}
\end{equation}
where $A_t = A_0 + a_t$. Furthermore, assuming that the initial condition is integrable, one has that $F_{A_t}^{0,2} = 0$ and $-d^c\omega = H_t$ for all $t$, where $H_t$ is given by 
\eqref{eq:Htdef}.
\end{proposition}

\begin{proof}
The last equation in \eqref{eq:CPFunitary} follows from Lemma \ref{lem:UPCF}. To prove the remaining equations, given $t_0$ in the interval of definition of the flow line, we will denote $\ell = \ell_{t_0}$, for simplicity, and similarly for the corresponding classical tensors. Consider the map
\begin{equation}\label{eq:projectionellt}
P_t \colon \bar \ell \rightarrow \bar \ell_t: a_{\bar \ell} \rightarrow (a_{\bar \ell})_{\bar \ell_t}.
\end{equation}
More explicitly, regarding $P_t$ as a map $P_t \colon \bar \ell \to E^c$ and denoting by $\mathcal{J}_t$ the almost complex structure on the horizontal subspace $W_t$ of $\ell_t$ defined by \eqref{eq:Whor},  we can write
\begin{align*}
P_t(a_{\bar \ell}) & = \frac{1}{4} \Bigg{(} \Id_{W_t^c} + \i \mathcal{J}_t\Bigg{)}(G_t + \Id_{E^c}) a_{\bar \ell}\\
& = \frac{1}{4} \Bigg{(} \Id_{E^c} + \i e^{(b_t,-a_t)} (\Id + g_t) J \pi^c\Bigg{)} (G_t + \Id_{E^c}) a_{\bar \ell}
\end{align*}
where we have used Lemma \ref{lem:UPCF}, which implies $J := J_{t_0} = J_t$ for all $t$. From this, we can identify 
$$
\dot{\overline{\ell}} := \dt_{|t = t_0} \overline{\ell}_t \in \operatorname{Hom}(\overline{\ell}, \ell \oplus V_-^c)
$$
with
$$
\dot{\overline{\ell}} = \pi_\ell \oplus \pi_- \circ \dt_{|t = t_0} P_t,
$$
where $\pi_\ell \colon E^c \to \ell$ and $\pi_- \colon E^c \to V_-^c$ denote the orthogonal projections. Then, letting
\begin{align*}
    a_{\bar{\ell}} = e^{(\i\omega+b,-a)}X^{0,1},\quad b_{\bar{\ell}} = e^{(\i\omega+b,-a)}Y^{0,1},
\end{align*}
for $X^{0,1},Y^{0,1} \in T^{0,1}_J$, using the notation in Proposition \ref{prop:GRFexp}, we have
\begin{align*}
\langle \dot{\overline{\ell}} a_{\bar{\ell}}, b_{\bar{\ell}} \rangle =&\ \frac{1}{2} \langle C_{t_0} (X^{0,1} + g X^{0,1}),Y^{0,1} + gY^{0,1} \rangle + \frac{1}{4}\dot g(X^{0,1},Y^{0,1})\\
&\ + \frac{1}{4} \langle \dot{G} a_{\bar{\ell}}, b_{\bar{\ell}} \rangle + \frac{\i}{4}\langle (\Id + g) J\pi^c \dot G a_{\bar{\ell}},(\Id + g)Y^{0,1}\rangle\\
=&\ \frac{1}{2} \langle C_{t_0} (X^{0,1} + g X^{0,1}),Y^{0,1} + gY^{0,1} \rangle + \frac{1}{4}\dot g(X^{0,1},Y^{0,1}) + \frac{\i}{4} g(J\pi^c \dot G a_{\bar{\ell}},Y^{0,1}),
\end{align*}
where we have used that $\dot G a_{\bar{\ell}} \in V_-^c$. Now, similarly as in the proof of Proposition \ref{prop:GRFexp}, we have
\begin{align*}
\dot G a_{\bar{\ell}} & =  e^{(b,-a)} \left( g^{-1} i_{X^{0,1}}(-\dot b + \langle a \wedge \dot a \rangle_{\mathfrak{k}}) - i_{X^{0,1}}(-\dot b + \langle a \wedge \dot a \rangle_{\mathfrak{k}}-2\dot{a}(X^{0,1}) \right)\\
&\ + e^{(b,-a)}(- g^{-1} \dot g(X^{0,1}) + \dot g(X^{0,1}) )
\end{align*}
and therefore 
$$
\pi^c \dot G a_{\bar{\ell}} = g^{-1} i_{X^{0,1}}(-\dot b + \langle a \wedge \dot a \rangle_{\mathfrak{k}})  - g^{-1} \dot g(X^{0,1}).
$$
Finally, this implies
\begin{align*}
\langle \dot{\overline{\ell}} a_{\bar{\ell}}, b_{\bar{\ell}} \rangle &  = \frac{1}{2}\dot g(X^{0,1},Y^{0,1}) - \frac{1}{2}(-\dot b + \langle a \wedge \dot a \rangle_{\mathfrak{k}})(X^{0,1},Y^{0,1}) = \frac{1}{2}(\dot b - \langle a \wedge \dot a \rangle_{\mathfrak{k}})(X^{0,1},Y^{0,1}),
\end{align*}
where we have used that $J_t$ is constant and hence $\dot g^{0,2} = 0$ for all $t$. Similarly, denoting
$$
b_- = e^{(b,-a)}(Y + r - g(Y))
$$
we have
\begin{align*}
\langle \dot{\overline{\ell}} a_{\bar{\ell}}, b_- \rangle & = \frac{1}{2} \langle C_{t_0} (X^{0,1} + g X^{0,1}),Y + r - gY \rangle + \frac{1}{2} \langle \dot g(X^{0,1}),Y + r - g(Y) \rangle + \frac{1}{4} \langle \dot{G} a_{\bar{\ell}}, b_- \rangle\\
& = - \langle \dot a(X^{0,1}),r \rangle_{\mathfrak{k}} - \frac{1}{2}(-\dot b + \langle a \wedge \dot a \rangle_{\mathfrak{k}} )(X^{0,1},Y) + \frac{1}{2} \dot g (X^{0,1},Y).
\end{align*}
Next, we calculate the infinitesimal action of 
$\zeta^{\overline{\ell}}$ on $\overline{\ell}$. By Lemma \ref{lem:LieAut} and Remark \ref{rem:Lieexchange}, the formula for $\zeta^{\overline{\ell}}$ in the fixed isotropic splitting of $\overline \ell_0$ is given by (cf. Proposition \ref{prop:zetaell})
$$
\zeta^{\overline{\ell}} = B - \langle a \wedge (d_{A}  + d_{A_0}) s \rangle_{\mathfrak{k}} + s
$$
where 
$$
B:= \i \rho_B(\omega) + \i \langle F_A,\Lambda_\omega F_A \rangle_{\mathfrak{k}}, \qquad s := \frac{\i}{2} \Lambda_\omega F_A.
$$
Then (see \eqref{eq:actLieE}),
\begin{align*}
\zeta^{\overline{\ell}} \cdot a_{\bar{\ell}}
& = i_{X^{0,1}}(B - \langle a \wedge (d_{A}  + d_{A_0}) s \rangle_{\mathfrak{k}}) + 2  \langle d_{A_0}s, a(X^{0,1}) \rangle_{\mathfrak{k}} + i_{X^{0,1}}d_{A}s 
\end{align*}
which implies
\begin{align*}
\langle \zeta^{\overline{\ell}} \cdot a_{\bar{\ell}}, b_{\overline\ell} \rangle =&\ \frac{1}{2}(B - \langle a \wedge (d_{A}  + d_{A_0}) s \rangle_\mathfrak{k}) (X^{0,1},Y^{0,1})
 + \langle i_{Y^{0,1}}d_{A_0}s, a(X^{0,1}) \rangle_{\mathfrak{k}} - \langle i_{X^{0,1}}d_{A}s, a(Y^{0,1})\rangle_\mathfrak{k}\\
=&\ \frac{1}{2}B(X^{0,1},Y^{0,1}) - \frac{1}{2}\langle a(X^{0,1}), i_{Y^{0,1}}d_{A}s\rangle_\mathfrak{k} + \frac{1}{2}\langle a(X^{0,1}), i_{Y^{0,1}}d_{A_0}s\rangle_\mathfrak{k}\\
&\ + \frac{1}{2}\langle a(Y^{0,1}), i_{X^{0,1}}d_{A_0}s\rangle_\mathfrak{k} - \frac{1}{2}\langle a(Y^{0,1}), i_{X^{0,1}}d_{A}s\rangle_\mathfrak{k}\\
=&\ \frac{1}{2}B(X^{0,1},Y^{0,1}) - \frac{1}{2}\langle a(X^{0,1}), [a(Y^{0,1}),s]\rangle_\mathfrak{k} - \frac{1}{2}\langle a(X^{0,1}), [s,a(Y^{0,1})]\rangle_\mathfrak{k}\\
=&\ \frac{\i}{2}(\rho_B(\omega) + \langle F_A,\Lambda_\omega F_A \rangle_{\mathfrak{k}})(X^{0,1},Y^{0,1}),\\
\langle \zeta^{\overline{\ell}} \cdot a_{\bar{\ell}}, b_- \rangle =&\ \frac{1}{2}(B - \langle a \wedge (d_{A}  + d_{A_0}) s \rangle_\mathfrak{k}) (X^{0,1},Y)
 + \langle i_{Y}d_{A_0}s, a(X^{0,1}) \rangle_{\mathfrak{k}} - \langle i_{X^{0,1}}d_{A}s, a(Y)\rangle_\mathfrak{k} + \langle i_{X^{0,1}}d_{A}s, r \rangle_\mathfrak{k} \\
=&\ \frac{\i}{2}(\rho_B(\omega) + \langle F_A,\Lambda_\omega F_A \rangle_{\mathfrak{k}})(X^{0,1},Y) + \langle i_{X^{0,1}}\dbar_{A}s, r \rangle_\mathfrak{k}.
\end{align*}
Collecting the previous calculations, it follows that the flow \eqref{eq:UPCF} is equivalent to
\begin{equation*}
\begin{split}
(-\dot b + \langle a \wedge \dot a \rangle_\mathfrak{k})^{0,2}  & = \i(\rho_B + \langle F_A, \Lambda_\omega F_A \rangle_{\mathfrak{k}})^{0,2},\\
(\i\dot\omega + \dot b - \langle a \wedge \dot a \rangle_\mathfrak{k})^{1,1 + 0,2}  & = -\i (\rho_B + \langle F_A, \Lambda_\omega F_A \rangle_{\mathfrak{k}})^{1,1 + 0,2},\\
\dot a^{0,1} & =  \frac{\i}{2}\dbar_A (\Lambda_\omega F_A).
\end{split}
\end{equation*}
The first part of the statement follows now by adding and substracting the first two evolution equations, taking into account that $\dot b + \langle a \wedge \dot a \rangle_\mathfrak{k}$ and $\dot \omega$ are real two-forms, and
$$
\dot a = \dot a^{0,1} + (\dot a^{0,1})^* = \frac{\i}{2}(\dbar_A - \partial_A) (\Lambda_\omega F_A) = \frac{1}{2}J d_A (\Lambda_\omega F_A),
$$
where $\mathfrak{k}^c \to \mathfrak{k}^c \colon s \mapsto s^*$ denotes the $\CC$-antilinear Cartan involution, fixing $\mathfrak{k} \subset \mathfrak{k}^c := \mathfrak{k} \otimes_\RR \CC$. Assuming now that the initial lifting is integrable, the second part of the statement follows from Lemma \ref{lem:UPCF} and Lemma \ref{lem:01lift}.
\end{proof}

To finish this section, we establish the relation between the pluriclosed flow \eqref{eq:UPCF} and the string algebroid generalized Ricci flow \eqref{eq:CourantGRF} (see also \eqref{eq:GRFgbA}).

\begin{proposition}\label{prop:UPCFRiemannian}
Let $\overline{\ell}_t \subset E^c$ be a one-parameter family of positive, integrable, $(0,1)$-liftings satisfying pluriclosed flow \eqref{eq:UPCF}. Then, $(g_t,b_t,A_t)$ satisfies 
\begin{equation}\label{eq:UPCFGRF}
\begin{split}
\dt g & = -\operatorname{Rc} + \frac{1}{4} H^2 + F_A^2 - \frac{1}{2}L_{\theta_\omega^\sharp}g,\\
\dt b & = - \frac{1}{2}\(d^{\star}H + \IP{a \wedge \left( d_A^{\star}F_A - F_A \lrcorner \ H + i_{\theta_\omega^\sharp} F_A \right) }_{\mathfrak{k}} - d \theta_\omega +  i_{\theta_\omega^\sharp} H\),\\
\dt A & = - \frac{1}{2}(d_A^{\star}F_A - F_A \lrcorner \ H + i_{\theta_\omega^\sharp} F_A),
\end{split}
\end{equation}
where $g_t = \omega_t(,J)$, $\theta_{\omega_t}$ denotes the associated Lee form, and $H_t = -d^c\omega$ satisfies \eqref{eq:Htdef} for all $t$. Consequently, the one-parameter family of generalized metrics $G_t$ on $E$ associated to $\overline{\ell}_t$ satisfies $\zeta_t$-gauge-fixed generalized Ricci flow \eqref{eq:RFgaugefixedabs},  for the one-parameter family of infinitesimal automorphisms of $E$ given by 
$$
\zeta_t \equiv (- \theta_{\omega_t}^\sharp,i_{\theta_{\omega_t}^\sharp}a_t,B_t),
$$
where
$$
B_t = d (- \theta_{\omega_t} - i_{\theta_{\omega_t}^\sharp}b_t + \langle a_t , i_{\theta_{\omega_t}^\sharp} a_t \rangle_{\mathfrak{k}}) + i_{\theta_{\omega_t}^\sharp} H_0 + 2 \langle i_{\theta_{\omega_t}^\sharp}a_t,F_{A_0}\rangle_{\mathfrak{k}}.
$$
\end{proposition}

\begin{proof}
Note that for any Hermitian metric $\omega$ on $(M,J)$ and any connection $A$ on $P$ solving the Bianchi identity \eqref{eq:bianchitrans}, with $H = - d^c\omega$ and $F_A^{0,2} = 0$, one has the following
formulae (see the proof of \cite[Lemma 5.3]{GFGM} and \cite[Proposition 5.6]{GFGM}) 
\begin{equation}\label{eq:rhoBdecomp}
\begin{split}
    \rho_B^{1,1}(\cdot,J\cdot) & = \operatorname{Rc} - \frac{1}{4}H^2 -F_A^2+\IP{\Lambda_\omega F_A,F_A(J,)} + \frac{1}{2}L_{\theta_\omega^\sharp} g,\\
    \rho_B^{2,0+0,2}(\cdot,J\cdot) & = -\frac{1}{2}( d^{\star} H - d \theta_\omega +  i_{\theta_\omega^\sharp}H),\\
    J d_A(\Lambda_\omega F_A) & = - d_A^{\star}F_A + F_A \lrcorner \ H - i_{\theta_\omega^\sharp} F_A.
\end{split}
\end{equation}
Consequently, the first part of the statement follows from Proposition \ref{prop:UPCFexp}. The second part of the statement is a straightforward consequence of Proposition \ref{p:GFGRFinGG+}.
\end{proof}

\subsection{Generalized Hermitian metrics on holomorphic string algebroids}\label{sec:GHstring}

Our next goal is to give a description of the pluriclosed flow \eqref{def:UPCF} in the holomorphic gauge. For this we will use that an integrable $(0,1)$-lifting (see Definition \ref{def:lifting}) determines uniquely a string algebroid in the holomorphic category \cite{garciafern2020gauge}. In this way, $(0,1)$-liftings relate to the complex Courant algebroid $E^c = E \otimes \mathbb{C}$ as Dolbeault operators relate
to a smooth complex vector bundle. Given this, we will introduce the basic unknowns of the pluriclosed flow in the holomorphic gauge: generalized Hermitian metrics.

We fix a complex manifold $(M,J)$ with integrable complex structure $J$ and denote its anti-holomorphic tangent bundle by $T^{0,1} = T^{0,1}_J$. In the sequel, we consider integrable $(0,1)$-liftings $\overline{\ell}$ of $T \otimes \CC$, such that
$$
\pi^c(\overline{\ell}) = T^{0,1}.
$$
Given such a lifting $\overline{\ell}$, following \cite{GualtieriGKG} we consider the reduction of $E \otimes \mathbb{C}$ by $\overline{\ell}$ given by the orthogonal bundle 
$$
\mathcal{Q}_{\overline{\ell}} := \overline{\ell}^\perp/\overline{\ell},
$$
where $\overline{\ell}^\perp$ is the orthogonal complement of $\overline{\ell}$ with respect to the symmetric pairing on $E \otimes \mathbb{C}$,
which fits in a vector bundle complex of the form
\begin{equation}\label{eq:holCouseqaux}
T^*_{1,0} \overset{\pi^*}{\longrightarrow} \mathcal{Q} \overset{\pi}{\longrightarrow} T^{1,0}.
\end{equation}
The bundle $\mathcal{Q}_{\overline{\ell}}$ carries a natural Dolbeault operator such that the previous sequence is holomorphic, defined as follows: given $s$ a smooth section of $\mathcal{Q}_{\overline{\ell}}$, we define
$$
\overline{\partial}^{\overline{\ell}}_V s = [\til V,\til s] \quad \textrm{mod}\ \overline{\ell}
$$
where $V \in \Gamma(T^{0,1})$, $\til V$ is the unique lift of $V$ to $\overline{\ell}$, and $\til s$ is any lift of $s$ to a section of $\overline{\ell}^\perp$. The Jacobi identity for the Dorfman bracket on $E \otimes \mathbb{C}$ implies that
$$
\overline{\partial}^{\overline{\ell}} \circ \overline{\partial}^{\overline{\ell}} = 0,
$$
that is, $\overline{\partial}^{\overline{\ell}}$ is integrable, and that it induces a Dorfman bracket on the holomorphic sections of $\mathcal{Q}_{\overline{\ell}}$. It is not difficult to prove that $\mathcal{Q}_{\overline{\ell}}$ satisfies axioms analogue to those of a Courant algebroid in the holomorphic category \cite{garciafern2018holomorphic}. Furthermore, it was proved in \cite[Proposition 2.18]{garciafern2020gauge}
that this \emph{holomorphic Courant algebroid} is of string type, that is, it fits in an exact sequence of holomorphic vector bundles
\begin{equation*}
\begin{split}
0 \longrightarrow T^*_{1,0} \overset{\pi^*}{\longrightarrow} \mathcal{Q} \overset{\rho}{\longrightarrow} A_{P^c}^{1,0}, \longrightarrow 0,
\end{split}
\end{equation*}
where $A_{P^c}^{1,0}$ denotes the holomorphic Atiyah Lie algebroid of a holomorphic principal bundle $(P^c,\overline{J})$ over $(M,J)$, uniquely determined by $\bar \ell$. The map $\rho$ is bracket preserving and considered as part of the data of the string algebroid.

The holomorphic string algebroids of our interest can be always put in a special normal form, as follows: let $K^c$ be the complex reductive Lie group given by the complexification of the compact group $K$. We endow the Lie algebra $\mathfrak{k}^c = \mathfrak{k} \otimes \CC$ with the $\CC$-linear extension of $\langle \cdot,\cdot \rangle_{\mathfrak{k}}$, for which we use the same notation
$$
\langle \cdot , \cdot \rangle_{\mathfrak{k}} \colon \mathfrak{k}^c \otimes \mathfrak{k}^c \to \CC.
$$
Let $(P^c,\overline{J})$ be a holomorphic principal $K^c$-bundle over $(M,J)$. Let $\tau \in \Lambda^{3,0+2,1}$ be a complex three-form on $M$ of type $(3,0)+(2,1)$ and let $A^c$ be a smooth principal connection on $P^c$ inducing the given holomorphic structure, that is, 
$$
\overline{J} = (A^c)^\perp J dp + \sqrt{-1} A^c,
$$
for $dp \colon TP^c \to T$ the canonical projection. In particular, since $\overline{J}$ is integrable, it follows that
$$
F_{A^c}^{0,2} = 0.
$$
Assume further that the following Bianchi identity is satisfied
\begin{equation}\label{eq:bianchitranshol}
d\tau = \langle F_{A^c} \wedge F_{A^c} \rangle_{\mathfrak{k}}.
\end{equation}
Consider the bundle
\begin{equation}\label{eq:Q0}
\mathcal{Q}_{\tau,A^c} = T^{1,0} \oplus \ad P^c \oplus T^*_{1,0},
\end{equation}
with Dolbeault operator
\begin{equation*}\label{eq:DolQ}
\dbar_{\tau,A^c} (V + r + \xi)  = \dbar V + i_V F_{A^c}^{1,1} + \dbar^{A^c} r + \dbar \xi - i_{V}\tau^{2,1} + 2\langle F_{A^c}^{1,1}, r\rangle_{\mathfrak{k}},
\end{equation*}
non-degenerate symmetric bilinear form
$$
\IP{V + r + \xi , V + r + \xi}  = \xi(V) + \IP{r,r}_{\mathfrak{k}},
$$
bracket,
\begin{equation*}
	\begin{split}
	[V+ r + \xi,W + t + \eta]_0   = {} & [V,W] - F^{2,0}_{A^c}(V,W) + \partial^{A^c}_V t - \partial^{A^c}_W r - [r,t]\\
	& {} + i_V \partial \eta + \partial (\eta(V)) - i_W\partial \xi + i_Wi_V \tau^{3,0},\\
	& {} + 2\langle\partial^{A^c} r, t\rangle_{\mathfrak{k}} + 2\langle i_V F_{A^c}^{2,0}, t\rangle_{\mathfrak{k}} - 2\langle i_W F_{A^c}^{2,0}, r\rangle_{\mathfrak{k}},	
	\end{split}
\end{equation*}
and anchor map $\pi_0(V+ r + \xi) = V$. It is an exercise to show that $\mathcal{Q}_{\tau,A^c}$ defines a holomorphic Courant algebroid in the sense of \cite[Definition 2.1]{garciafern2018holomorphic}. Furthermore, the map (see Section \ref{s:gpb})
$$
\rho_0 \colon \mathcal{Q}_{\tau,A^c} \to A_{P^c}^{1,0} \colon V + r + \xi \to (A^c)^\perp V + X^r
$$
is holomorphic and it endows $\mathcal{Q}_{\tau,A^c}$ with a structure of holomorphic string algebroid \cite[Proposition 2.4]{garciafern2020gauge}.

In the next result we establish the relation between the abstract construction of holomorphic string algebroids via reduction and the explicit normal form above. We fix a smooth string algebroid $E$ over a complex manifold $(M,J)$ and a reference integrable $(0,1)$-lifting $\overline{\ell}_0 \subset E^c = E \otimes \mathbb{C}$ such that $\pi^c(\overline{\ell}_0) = T^{0,1}$. Then, the horizontal subspace $W_0$ in \eqref{eq:Whor} determines an isotropic splitting of $E$ and hence a presentation 
$$
E^c \cong T \otimes \CC \oplus \ad P^c \oplus T^* \otimes \CC
$$ 
and a pair $(H_0,A_0)$ satisfying the Bianchi identity \eqref{eq:bianchitrans0} (see Section \ref{sec:background}). Let $\overline{\ell} = \overline{\ell}(\omega,\sigma)$ be another $(0,1)$-lifting of $T^{0,1}$ (see Lemma \ref{lem:01lift}). In the fixed isotropic splitting determined by $\bar \ell_0$, we can write (cf. \eqref{eq:elltba})
\begin{equation*}
\bar{\ell} =\ e^{(\i\omega + b,-a)}(T^{0,1}),
\end{equation*}
for a uniquely determined $(b,a)$-transformation. The next result follows from \cite[Proposition 2.16]{garciafern2020gauge} and gives an explicit formula for the holomorphic Courant algebroid associated to $\bar \ell(\omega,b,a)$. For this, we use the identity
$$
\bar \ell(\omega,b,a)^\perp = \bar \ell(\omega,b,a) \oplus e^{(\i\omega + b,-a)}(T^{1,0}  \oplus \ad P^c  \oplus T^*_{1,0}).
$$

\begin{lemma}[\cite{garciafern2020gauge}]\label{l:liftingQ} 
There is a canonical isomorphism
$$
\mathcal{Q}_{\bar \ell(\omega,b,a)} \cong \mathcal{Q}_{2 \i\partial \omega,A},
$$
given by
$$
\Big{[}e^{(\i\omega+b,-a)}(X+r+\xi^{1,0})\Big{]}\mapsto X^{1,0}+r+\xi^{1,0},
$$
where $\mathcal{Q}_{2 i\partial \omega,A}$ is a holomorphic string algebroid with holomorphic principal $K^c$-bundle $(P^c,\overline{J})$, where $\overline{J}$ is the holomorphic principal $K^c$-bundle structure induced by $A = A_0 + a$.
\end{lemma}

We are ready to introduce the notion of generalized Hermitian metric on holomorphic string algebroids  (cf. \cite{garciafern2018canonical,garciafern2020gauge}). For our applications, we need to fix a holomorphic string algebroid $\mathcal{Q}$ with underlying holomorphic principal $K^c$-bundle $(P^c,\overline{J})$, and vary a smooth string algebroid $E$ with underlying principal $K$-bundle
$$
P_h = h(M)\cdot K \subset P^c, 
$$
where $h \in \Gamma(P^c/K)$ denotes a reduction to the maximal compact subgroup $K \subset K^c$.

\begin{defn}\label{def:metricQ}
Let $(M,J)$ be a complex manifold endowed with a holomorphic string algebroid $\mathcal{Q}$. A \emph{generalized Hermitian metric} on $\mathcal{Q}$ is given by a tuple $\mathbf{h} = (h,E,\bar \ell,\varphi)$, where
\begin{enumerate}

\item $h \in \Gamma(P^c/K)$ is a reduction of $P^c$ to $K \subset K^c$,

\item $\rho \colon E \to A_{P_h}$ is a smooth real string algebroid over $M$ with underlying principal $K$-bundle $P_h$,

\item $\bar \ell \subset E^c = E \otimes \CC$ is a positive, integrable, $(0,1)$-lifting such that $\pi^c(\overline{\ell}) = T^{0,1}$,

\item $\varphi \colon \mathcal{Q}_{\bar \ell} \to \mathcal{Q}$ is an isomorphism of holomorphic Courant algebroids in a commutative diagram of the form
\begin{equation}\label{eq:defGHM}
	\xymatrix{
		0 \ar[r] & T^*_{1,0} \ar[r] \ar[d]^{\Id} & \mathcal{Q}_{\bar \ell} \ar[r]^\rho \ar[d]^{\varphi} & A_{P^c}^{1,0} \ar[r] \ar[d]^{\Id} & 0,\\
		0 \ar[r] & T^*_{1,0} \ar[r] & \mathcal{Q} \ar[r]^{\rho_0} & A_{P^c}^{1,0} \ar[r] & 0.
	}
\end{equation}
\end{enumerate}
\end{defn}

In the next result we choose a model $\mathcal{Q} = \mathcal{Q}_{\tau,A_
0}$ and give an explicit description of the degrees of freedom of a generalized Hermitian metric, following \cite[Proposition 4.13]{garciafern2020gauge}. 
By the last item in the previous definition, the connection $A = A_0 + \alpha$ determined by $\bar \ell$ induces the \emph{same} holomorphic structure on $P^c$ as $A_0$, and hence in particular
$$
\alpha \in \Gamma(T^*_{1,0} \otimes \ad P^c).
$$

\begin{proposition}\label{propo:Chernclassic}
Let $(M,J)$ be a complex manifold endowed with a holomorphic string algebroid $\mathcal{Q} = \mathcal{Q}_{\tau,A_0}$. A generalized Hermitian metric $\mathbf{h}$ on $\mathcal{Q}$ is equivalent to a triple $(\omega,\beta,h)$, where $\omega \in \Gamma(\Lambda^{1,1}_\mathbb{R})$ and $\beta \in \Gamma(\Lambda^{2,0})$ are, respectively, a real, positive, $(1,1)$ form and a $(2,0)$ form on $(M,J)$, and $h \in \Gamma(P^c/K)$ is a reduction of $P^c$ to $K \subset K^c$, satisfying
\begin{equation}\label{eq:structuralGHM}
2\i\partial \omega = \tau - 2\i d\beta + 2 \langle \alpha \wedge F_{A_0} \rangle_{\mathfrak{k}} + \langle \alpha \wedge d_{A_0}\alpha \rangle_{\mathfrak{k}} + \frac{1}{3}\langle \alpha \wedge [\alpha \wedge \alpha] \rangle_{\mathfrak{k}},
\end{equation}
for $A^h = A_0 + \alpha$ the Chern connection of $h$ on $(P^c, \bar J)$.
\end{proposition}

\begin{proof}
Given a generalized Hermitian metric $\mathbf{h} = (h,E,\bar \ell,\varphi)$, the tuple $(h,E,\bar \ell)$ is equivalent to a tuple $(P_h,\omega,A)$ satisfying
$$
dd^c \omega + \la F_A \wedge F_{A} \ra_\mathfrak{k} = 0, \qquad F_A^{0,2} = 0,
$$
by Lemma \ref{lem:01lift}, where $\omega \in \Gamma(\Lambda^{1,1}_\mathbb{R})$ and $A$ is a connection on $P_h$. The string algebroid $E$ is recovered as in Section \ref{sec:background} taking $H = -d^c\omega$. By \cite[Lemma 2.7]{garciafern2018holomorphic}, the isomorphism $\varphi$ is equivalent to a pair $(\beta,\bar \phi)$, where $\beta \in \Gamma(\Lambda^{2,0})$ and $\bar \phi \colon (P^c,\bar J_A) \to (P^c,\bar J)$ is a holomorphic gauge transformation for
$$
\overline{J}_A = A^\perp J dp + \sqrt{-1} A,
$$
satisfying \eqref{eq:structuralGHM} with $\alpha = g A - A_0$. The last item in Definition \ref{def:metricQ} forces $g = \Id$ and hence the connection $A = A_0 + \alpha$ induces the \emph{same} holomorphic structure as $A_0$, and therefore $A = A^h$, the Chern connection of $h$ on $(P^c,\bar J)$. The converse follows trivially from the previous discussion.    
\end{proof}

\begin{remark}\label{rem:BIdentity}
In the sequel, we will often identify generalized Hermitian metrics $\mathbf{h}$ on $\mathcal{Q}$ with triples $(\omega,\beta,h)$ as in the statement of Proposition \ref{propo:Chernclassic}. By the structural equation \eqref{eq:structuralGHM}, any such triple satisfies the \emph{Bianchi identity}
\begin{equation}\label{eq:BIGHM}
dd^c \omega + \la F_{h} \wedge F_{h} \ra_\mathfrak{k} = 0,
\end{equation}
for $F_h$ the curvature of the Chern connection of $h$.
\end{remark}

Given a generalized Hermitian metric $\mathbf{h} = (h,E,\bar \ell,\varphi)$, one can naturally associate a pseudo-Hermitian metric on the holomorphic vector bundle underlying $\mathcal{Q}$, which motivates the previous definition. To see this, note that $\bar \ell$ defines a generalized metric $G$ on $E$ by
$$
W \oplus W^\perp = V_+ \oplus V_-, 
$$
where $W$ is the horizontal lift of $T$ to $E$ defined by \eqref{eq:Whor}. This implies that
$$
\bar \ell^\perp = (V_- \otimes \mathbb{C}) \oplus \bar \ell.
$$
Therefore, as a smooth orthogonal bundle $\mathcal{Q}_\ell$ is canonically isomorphic to 
$$
\mathcal{Q}_{\bar \ell}  \cong V_- \otimes \mathbb{C}.
$$

\begin{defn}\label{d:generalizedmetriHerm}
Given a generalized Hermitian metric $\mathbf{h} =(h,E,\bar \ell,\varphi)$ on $\mathcal{Q}$, the induced Hermitian metric is
$$
\mathbf{G} = \mathbf{G}_{\mathbf{h}}:= \varphi_* \mathbf{G}_{\bar \ell}
$$
where $\mathbf{G}_{\bar \ell}$ is the Hermitian metric on $\mathcal{Q}_{\bar \ell}$ defined by
$$
\mathbf{G}_{\bar \ell}([s_1],[s_2]) = - \IP{\pi_- s_1, \overline{\pi_- s_2}}
$$
for $[s_j] \in \bar \ell^\perp/\bar \ell$ and $\pi_- \colon \bar \ell^\perp \to V_- \otimes \mathbb{C}$ the orthogonal projection.
\end{defn}

In our next result we calculate an explicit formula for the generalized Hermitian metric $\mathbf{G}$ in terms of the isomorphism $\mathcal{Q}_{\bar \ell} \cong \mathcal{Q}_{2\i\partial \omega,A^h}$ in Lemma \ref{l:liftingQ}.

\begin{lemma}\label{t:Ggeneralized1} 
Let $\mathcal{Q}$ be a holomorphic string algebroid with a generalized Hermitian metric $\mathbf{h} =(h,E,\bar \ell,\varphi)\equiv (\omega,\beta,h)$. The complex isometry $\psi \colon \mathcal{Q}_{2\i\partial \omega,A^h} \to V_-\otimes \mathbb{C}$ induced by Lemma \ref{l:liftingQ} is given by
$$
\psi(q) = e^{\i\omega}V + r - \tfrac{1}{2} e^{-\i\omega} g^{-1}\xi.
$$
for $q = V + r + \xi$. Consequently, the induced Hermitian metric $\mathbf{G} = \mathbf{G}_{\mathbf{h}}$ on $\mathcal{Q}$ is
\begin{equation*}
\begin{split}
\mathbf{G}(q,q) & = g(V,\overline{V}) - \IP{r + i_V \alpha,\overline{r}^h + i_{\overline{V}} \overline{\alpha}^h}_\mathfrak{k}\\
& + \tfrac{1}{4}\IP{\xi + 2\i i_V \beta - \langle i_V \alpha,\alpha \rangle_{\mathfrak{k}} - 2\langle \alpha,r\rangle_{\mathfrak{k}},\overline{\xi} - 2 \i i_{\overline{V}} \overline{\beta} - \overline{\langle i_{V} \alpha,\alpha \rangle}_{\mathfrak{k}} - 2\overline{\langle \alpha,r\rangle}_{\mathfrak{k}}}_g,
\end{split}
\end{equation*}
where $r \to \overline{r}^h$ is the $\CC$-antilinear involution determined by $\ad P^c \cong \ad P_h \oplus \i \ad P_h$.

\begin{proof}
The formula for $\psi$ follows as in \cite[Lemma 4.5]{GFGM}, which readily implies
$$
\mathbf{G}_{\bar \ell}(q,q) = g(V,\overline{V}) + \tfrac{1}{4}\IP{\xi,\overline{\xi}}_g -  \IP{r,\overline{r}^h}_\mathfrak{k}.
$$
Arguing now as in the proof of \cite[Lemma 2.7]{garciafern2018holomorphic}
\begin{align*}
\varphi^{-1}(V + r + \xi) & = e^{(2\i\beta,\alpha)}(V + r + \xi)\\
& = V + r + i_V \alpha + \xi + 2\i i_V \beta - \langle i_V\alpha,\alpha \rangle_{\mathfrak{k}} - 2\langle \alpha,r\rangle_{\mathfrak{k}}
\end{align*}
and the proof follows by a direct calculation.
\end{proof}
\end{lemma}

\begin{remark}\label{rem:signature}
By the previous lemma, the signature of $\mathbf{G}$ is $(2n + l_2,l_1)$, where $(l_1,l_2)$ is the signature of $\IP{,}_\mathfrak{k}$ and $2n = \dim_\RR M$.
\end{remark}

\subsection{Pluriclosed flow on holomorphic string algebroids}\label{sec:PCFstring}

Building on Proposition \ref{propo:Chernclassic} we next introduce a version of pluriclosed flow in the holomorphic gauge in terms of classical data. Let $(M,J)$ be a complex manifold endowed with a holomorphic principal $K^c$-bundle $(P^c,\bar J)$. We will consider triples
$$
(\omega,\beta,h) \in \Gamma(\Lambda^{1,1}_\mathbb{R} \oplus \Lambda^{2,0}) \times \Gamma(P^c/K)
$$
where $\omega$ and $\beta$ are, respectively, a real, positive, $(1,1)$ form and a $(2,0)$ form on $(M,J)$, and $h \in \Gamma(P^c/K)$ is a reduction of $P^c$ to $K \subset G$. We define the second Ricci curvature of $h$ with respect to $\omega$ by
$$
S^h_g := \i \Lambda_{\omega} F_h,
$$
where $F_h$ denotes the curvature of the Chern connection of $h$ on the holomorphic principal bundle $(P^c,\bar J)$. In the sequel, we will use `vector bundle notation' in the holomorphic gauge, namely, we choose a faifhful representation for $K^c$ and regard a reduction $h \in \Gamma(P^c/K)$ as a Hermitian metric in the associated holomorphic vector bundle. This convention is convenient for obtaning more homogeneous formulae in the evolution for $h$ and the generalized Hermitian metrics in place.

\begin{definition}\label{def:PCFQ}
We say that a one-parameter family of triples $(\omega_t,\beta_t,h_t)$ is a solution of \emph{pluriclosed flow} if
\begin{equation}\label{eq:CPFhol}
\begin{split}
\dt \omega & = -\rho_B(\omega)^{1,1} + \i\langle S^h_g, F_h \rangle_{\mathfrak{k}},\\
\dt \beta & = - \rho_B(\omega)^{2,0} - \frac{\i}{2} \langle \alpha \wedge \partial^h S^h_g \rangle_{\mathfrak{k}},\\
h^{-1} \dt h  & = - S^h_g,
\end{split}
\end{equation}
where $\alpha_t = A^{h_t} - A^{h_0}$.
\end{definition}

\begin{remark}\label{rem:omegahflow} Observe that  a solution of pluriclosed flow \eqref{eq:CPFhol} is actually determined by the pair $(\omega_t,h_t)$.  For this reason we will at times refer to a family of pairs $(\omega_t,h_t)$ alone as a solution to pluriclosed flow. Nonetheless, the equation
for $\beta$ plays an important role in deriving estimates and determining the relationship
to generalized geometry below.
\end{remark}

Let $\mathcal{Q} = \mathcal{Q}_{\tau,A^c}$ be a holomorphic string algebroid over $(M,J)$ with underlying principal bundle $(P^c,\bar J)$. Given a generalized Hermitian metric $\mathbf{h} \equiv (\omega,\beta,h)$ on $\mathcal{Q}$, we define its associated second Ricci curvature by
$$
S^{\mathbf{G}}_g := \i \Lambda_{\omega} F_{\mathbf{G}},
$$
where $F_{\mathbf{G}}$ denotes the curvature of the Chern connection of $\mathbf{G} = \mathbf{G}(\omega,\beta,h)$ on the holomorphic vector bundle $\mathcal{Q}$. In the next result we show that the flow \ref{eq:CPFhol} is compatible with our notion of generalized Hermitian metric via Proposition \ref{propo:Chernclassic}, furnishing a coupled version of Donaldson HYM flow for metrics on a bundle \cite{DonaldsonHYM}. 

\begin{prop}\label{p:PCFhol}
Let $(\omega_t,\beta_t,h_t)$ be a solution of the pluriclosed flow \ref{eq:CPFhol}. If the initial condition $(\omega_0,\beta_0,h_0)$ satisfies equation \eqref{eq:structuralGHM} then so does $(\omega_t,\beta_t,h_t)$ for all $t$. Furthermore, if this is the case, then the associated family of Hermitian metrics $\mathbf{G}_t = \mathbf{G}(\omega_t,\beta_t,h_t)$ on $\mathcal{Q}$ satisfies
\begin{equation}\label{eq:DFlow}
\mathbf{G}^{-1} \dt \mathbf{G}  =\ - S_g^{\mathbf{G}}.
\end{equation}
Conversely, provided that $K$ is semisimple, given a family of generalized Hermitian metrics $\mathbf{h}_t \equiv (\omega_t,\beta_t,h_t)$ on $\mathcal{Q}$ satisfying \eqref{eq:DFlow}, the triple $(\omega_t,\beta_t,h_t)$ is a solution of the pluriclosed flow \eqref{eq:CPFhol}.

\begin{proof}
Setting
$$
C_t = 2 \langle \alpha \wedge F_{A_0} \rangle_{\mathfrak{k}} + \langle \alpha \wedge d_{A_0}\alpha \rangle_{\mathfrak{k}} + \frac{1}{3}\langle \alpha \wedge [\alpha \wedge \alpha] \rangle_{\mathfrak{k}},
$$
by the proof of \cite[Lemma 3.23]{garciafern2018canonical} it follows that
$$
\dot C_t = 2 \langle \dot\alpha_t \wedge F_{h_t} \rangle_{\mathfrak{k}} + d \langle \dot \alpha_t \wedge \alpha_t \rangle_{\mathfrak{k}},
$$
and therefore
\begin{align*}
2\i \partial \dot \omega_t - \dot C_t & = -2\i \partial \rho_B(\omega_t)^{1,1} - 2 \partial\langle S_t,F_{h_t}  \rangle_{\mathfrak{k}} - 2 \langle \dot\alpha_t \wedge F_{h_t} \rangle_{\mathfrak{k}} + d \langle \alpha_t \wedge \dot \alpha_t \rangle_{\mathfrak{k}}\\
& = 2\i d \rho_B(\omega_t)^{2,0} - 2 \langle \partial^{h_t}S_t \wedge F_{h_t}\rangle_{\mathfrak{k}} -  2 \langle \partial^{h_t}(h_t^{-1} \dot h_t)  \wedge F_{h_t} \rangle_{\mathfrak{k}} + d \langle \alpha_t \wedge \partial^{h_t}(h_t^{-1}\dot h_t ) \rangle_{\mathfrak{k}}\\
& = 2\i d \rho_B(\omega_t)^{2,0} - d \langle \alpha_t \wedge \partial^{h_t}S_t \rangle_{\mathfrak{k}}\\
& = - 2\i d \dot \beta_t,
\end{align*}
where $S_t := S^{h_t}_{g_t}$ and we have used that $d\rho_B(\omega_t) = 0$ and
$$
\dot \alpha_t = \partial^{h_t}(h_t^{-1}\dot h_t ).
$$
This proves the first part of the statement.

Fix $t>0$. After possibly modifying $\mathcal{Q}$ by an automorphism (see \cite[Lemma 2.7]{garciafern2018holomorphic}) we may assume without loss of generality that $\gb_t =0 = \alpha_t$. Notice that the Hermitian metric $\mathbf{G}_t = \mathbf{G}(\omega_t,\beta_t,h_t)$ as well as its Chern connection are compatible with the $\CC$-antilinear involution determined by the isomorphism 
$$
\psi_t \varphi^{-1}_t \colon \mathcal{Q} \to V^t_- \otimes \CC
$$
(see Lemma \ref{t:Ggeneralized1} and \cite[Proposition 4.4]{GFGM}). Hence,  using Lemma \ref{t:Ggeneralized1}, it suffices to calculate the variations
\begin{align*}
\dt \mathbf{G}(V,V) & = \dot g_t (V,\overline{V}), &  \dt \mathbf{G}(V,r) &= - \IP{i_V \dot \alpha_t ,\overline{r}^{h_t}}_\mathfrak{k},\\
\dt \mathbf{G}(V,\xi) & = \tfrac{\sqrt{-1}}{2}\IP{i_V \dot \beta_t,\overline{\xi}}_g, &  \dt \mathbf{G}(r,r) & = - \IP{r,[\overline{r}^{h_t},h_t^{-1}\dot h_t ]}  = - \IP{[h_t^{-1}\dot h_t,r],\overline{r}^{h_t}},
\end{align*}
for $q = V + r + \xi$ a smooth section of $\mathcal{Q}$, since the others will contain redundant information.
On the other hand, applying \cite[Proposition 4.9]{GFGM} we have
$$
S_{\mathbf{G}_t} = \psi_t^{-1}\left(\begin{array}{cc}
    - g^{-1}(\i\rho_B(\omega_t) + \IP{S_t,F_{h_t}}) & - (d^{h_t}S_t)^{\dagger}\\
    d^{h_t}S_t & [S_t,]
\end{array}\right) \psi_t
$$
where
$$
\psi_t \colon \mathcal{Q} \cong T^{1,0} \oplus \ad P^c \oplus T^*_{1,0} \to  V_-\otimes \mathbb{C}  \cong TM \otimes \CC \oplus \ad P^c
$$
is defined by
$$
\psi(V + r + \xi) = V - \tfrac{1}{2} g^{-1}\xi + r.
$$
Note that
$$
(\psi_*\mathbf{G}_t)(W + r,W + r)
= g(W,\overline{W}) -  \IP{r,\overline{r}^h}_\mathfrak{k}
$$
and therefore
\begin{equation*}
\begin{split}
\mathbf{G}_t(S_{\mathbf{G}_t} V,V) & = - (\psi_*\mathbf{G}_t)(g^{-1}(\i\rho_B(\omega_t) + \IP{S_t,F_{h_t}}))V,V)\\
& =  -(\i\rho_B(\omega_t) + \IP{S_t,F_{h_t}}))^{1,1}(V,\overline{V}),\\
\mathbf{G}_t(S_{\mathbf{G}_t} V,r) & =(\psi_*\mathbf{G}_t)(d^{h_t}S_t(V),r) = - \IP{\partial^{h_t}S_t(V),\overline{r}^{h_t}}_\mathfrak{k}\\
\mathbf{G}_t(S_{\mathbf{G}_t} V,\xi) & = \tfrac{1}{2}(\psi_*\mathbf{G}_t)(g^{-1}(\i\rho_B(\omega_t) + \IP{S_t,F_{h_t}}))V,g^{-1}\xi) = \tfrac{\i}{2}\rho_B(\omega_t)^{2,0}(V,g^{-1}\overline{\xi}),\\
\mathbf{G}_t(S_{\mathbf{G}_t} r,r) & = (\psi_*\mathbf{G}_t)([S_t,r],\overline{r}^{h_t}) = - \IP{[S_t,r],\overline{r}^{h_t}}_\mathfrak{k}.
\end{split}
\end{equation*}
The fact that $\mathbf{G}_t$ satisfies the evolution equation \ref{eq:DFlow} follows by simple inspection of the previous formulae.

Conversely, given $\mathbf{G}_t = \mathbf{G}(\omega_t,\beta_t,h_t)$ on $\mathcal{Q}$ satisfying \eqref{eq:DFlow}, by the previous calculation (in arbitrary gauge), the triple $(\omega_t,\beta_t,h_t)$ is a solution of
\begin{equation}\label{eq:CPFholaux}
\begin{split}
\dt \omega & = -\rho_B(\omega)^{1,1} + \i\langle S^h_g, F_h \rangle_{\mathfrak{k}},\\
\dt \beta & = - \rho_B(\omega)^{2,0} - \frac{\i}{2} \langle \alpha \wedge  \partial^h(S^h_g) \rangle_{\mathfrak{k}},\\
\dt A_h & = - \partial^h S^h_g\\
[h^{-1} \dt h,]  & = - [S^h_g,].
\end{split}
\end{equation}
Assuming that $K$ is semisimple, the last equation implies that $h^{-1} \dt h  = - S^h_g$ and the third equation is therefore redundant, and the proof follows.
\end{proof}
\end{prop}

The previous result establishes an interesting formal parallel between the pluriclosed flow \ref{eq:CPFhol} and Donaldson HYM flow for metrics on a bundle \cite{DonaldsonHYM}. In the following result we prove that this analogy actually carries over to the present setup, establishing a correspondence between flow lines for  \eqref{eq:CPFhol}, in the holomorphic gauge, and flow lines for \eqref{eq:UPCF} (cf. Proposition \ref{prop:UPCFexp}), in the unitary gauge. To fix ideas, assume that $h_t$ is a family of Hermitian metrics on a holomorphic vector bundle over $(M,J)$ satisfying the HYM flow equation
$$
h^{-1} \dt h = - S^h,
$$
where we have fixed a background Hermitian metric $g$. Note that $h_t = \overline \phi_t^*h_0$, where $\overline \phi_t$ is the one-parameter family of complex gauge transformations induced by $- \tfrac{1}{2}S^{h_t}$.
Denote by $A_0$ the Chern connection of $h_0$. Then, via the Chern correspondence we obtain a family of $h_0$-unitary connections $A_t = \overline \phi_t \cdot A_0$, which satisfies
\begin{equation}\label{eq:HYMflow}
\dt A_t = \frac{1}{2}\overline \phi_t(\dbar_{A _0}S^{h_t} - \partial_{A_0} S^{h_t})\overline \phi_t^{-1} = \frac{1}{2} J d_{A_t} (\Lambda_\omega F_{A_t}),
\end{equation}
where we have used that
\begin{equation}\label{eq:naturalS}
\overline \phi_t S^{h_t}\overline \phi_t^{-1} =  \i \Lambda_\omega F_{A_t}.
\end{equation}
Conversely, given a solution to \eqref{eq:HYMflow}, it determines a solution $h_t = \overline \phi_t^*h_0$ of the HYM flow equation, where $\overline \phi_t$ is the family of complex gauge transformations integrating $-\tfrac{\i }{2}\Lambda_\omega F_{A_t}$.

To prove the analogue of the previous correspondence in our setup, we fix a holomorphic string algebroid over $(M,J)$, which we assume to be of the form $\mathcal{Q} = \mathcal{Q}_{2\i\partial \omega_0,A_0}$, for $A_0 = A_{h_0}$ the Chern connection of $h_0$. In particular, we have
$$
dd^c\omega_0 + \IP{F_{A_0} \wedge F_{A_0}}_\mathfrak{k} = 0.
$$
Let $E_0$ denote the smooth string algebroid determined by $(P_{h_0},-d^c\omega_0,A_0)$ and $E^c_0$ its complexification. Given a generalized Hermitian metric $\mathbf{h} = (h,E,\bar \ell,\varphi)$ on $\mathcal{Q}$, by direct application of \cite[Lemma 3.6]{garciafern2020gauge}, it determines uniquely an isomorphism of complex string algebroids $\Phi_{\mathbf{h}} \colon E^c \to E^c_0$ given by a commutative diagram
\begin{equation*}\label{eq:gluingbricks2OLD}
\xymatrix{
\ar@{-->}[dr]  E^c \ar[rrrr]^{\Phi_{\mathbf{h}}} & &  & & E_0^c.  \ar@{-->}[dl]\\\
 & \mathcal{Q}_{\bar \ell} \ar[r]^{\varphi}  \ar@/^2pc/[rr]^{\varphi_0^{-1}\circ \varphi} & \mathcal{Q}  & \ar[l]_{\varphi_0} \mathcal{Q}_{\bar \ell_0}  &
  }
\end{equation*}
where $\varphi_0 = e^{-\i \omega_0}$ (see Lemma \ref{l:liftingQ}), and such that
$$
\Phi_{\mathbf{h}}(\bar{\ell}) = \bar \ell_0:= e^{\i \omega_0} T^{0,1} \subset E^c_0.
$$
Via the Chern correspondence for string algebroids \cite[Lemma 4.11]{garciafern2020gauge}, the \emph{complex gauge group} $\cG(E_0^c) \subset \Aut E^c_0$, given by automorphisms of $E^c_0$ covering the identity on $M$ (cf. Section \ref{sec:gaugefix}), acts on the left on the space of generalized Hermitian metrics on $\mathcal{Q}$ (see \cite[Appendix A.1]{garciafern2020gauge}). This action can be implicitly defined via the identity
$$
\Phi_{\mathbf{h}'} = \Phi \circ \Phi_{\mathbf{h}}, \qquad \mathbf{h}' := \Phi\cdot \mathbf{h}.
$$
More explicitly, we can identify $\Phi$ with a pair $(\bar{\phi},B)$ where $\bar{\phi} \colon P^c \to P^c$ is a complex gauge transformation and $B \in \Gamma(\Lambda^2_\CC)$ is a complex two form on $M$, satisfying (see \cite[Corollary 4.2]{GFRT17})
$$
dB = 2 \langle a^{\bar{\phi}}\wedge F_{A_0} \rangle_{\mathfrak{k}} + \langle a^{\bar{\phi}} \wedge d_{A_0}a^{\bar{\phi}} \rangle_{\mathfrak{k}} + \frac{1}{3}\langle a^{\bar{\phi}}\wedge [a^{\bar{\phi}} \wedge a^{\bar{\phi}}] \rangle_{\mathfrak{k}}.
$$
Then, in terms of the parameters $\mathbf{h} \equiv (\omega,\beta,h)$ (see Proposition \ref{propo:Chernclassic}), we have
\begin{equation*}\label{eq:PicQactexp}
\Phi \cdot \mathbf{h} \equiv (\omega',\beta',\bar{\phi}h),
\end{equation*}
where (for $a^{\bar{\phi}} = \bar{\phi}^*A_0 - A_0$) (see \cite[Lemma A.2]{garciafern2020gauge})
\begin{equation}\label{eq:PicQactexp2}
\begin{split}
\omega' & = \omega - \operatorname{Im} \(B + \IP{ a^{\bar{\phi}} \wedge (A_0 - A^h) }_\mathfrak{k} + \IP{ (\bar{\phi} A^h - A_0 )\wedge (A^{\bar{\phi}h} - \bar{\phi}A^h)}_\mathfrak{k} \)^{1,1},\\
2\i\beta' & = 2\i \beta - \(B + \IP{ a^{\bar{\phi}} \wedge (A_0 - A^h) }_\mathfrak{k} + \IP{ (\bar{\phi} A^h - A_0) \wedge (A^{\bar{\phi}h} - \bar{\phi}A^h) }_\mathfrak{k}\)^{2,0} \\
&  \phantom{ {} = } + \overline{\(B + \IP{(\bar{\phi} A^h - A_0) \wedge (A^{\bar{\phi}h} - \bar{\phi}A^h)}_\mathfrak{k}\)^{0,2}}.
\end{split}
\end{equation}

To establish the desired equivalence between the pluriclosed flow in the holomorphic gauge \ref{eq:CPFhol} and the pluriclosed flow in the unitary gauge \eqref{eq:UPCF}, we first prove an abstract characterization of the former (cf. Definition \ref{def:UPCF}).

\begin{lemma}\label{lem:PCFholabs}
Let $\mathbf{h}_t \equiv (h_t,E_t,\bar \ell_t,\varphi_t)$ be a one-parameter of generalized Hermitian metrics on $\mathcal{Q}$. Consider the associated family of infinitesimal complex gauge transformations
$$
\widetilde{\zeta}_t =(\Phi_{\mathbf{h}_t})_*\zeta_t \in \Lie \cG(E^c_0),
$$
where $\zeta_t := \zeta^{\bar \ell_t} \in \Lie \cG(E^c_t)$ is defined by (cf. Proposition \ref{prop:zetaell})
$$
\zeta_t = \i \rho_B(\omega_t) + \langle F_{h_t},S^{h_t}_{g_t} \rangle_{\mathfrak{k}} + \frac{1}{2} S^{h_t}_{g_t}.
$$
Then, $\mathbf{h}_t \equiv (\omega_t,\beta_t,h_t)$ solves the pluriclosed flow \ref{eq:CPFhol} if and only if
\begin{equation}\label{eq:CPCFholabs}
\dt \mathbf{h}_t = \widetilde{\zeta}_t\cdot \mathbf{h}_t.
\end{equation}
\end{lemma}

\begin{proof}
Denoting $\zeta_t = s_t + B_t$, by Remark \ref{rem:Lieexchange} we have
$$
\widetilde{\zeta}_t = s_t + B_t -\IP{\alpha_t \wedge (d_{A_0} s 
 + d_{A^{h_t}} s)}_{\mathfrak{k}} \in \Lie \cG(E^c_0),
$$
where $\alpha_t = A^{h_t} - A_0$. The proof follows now by direct application of \cite[Lemma A.4]{garciafern2020gauge}, which implies that the infinitesimal action of $\widetilde{\zeta}_t = \widetilde{s}_t + \widetilde B_t$ on $\mathbf{G}_t \equiv \mathbf{G}(\omega_t,\beta_t,h_t)$ is given by
$$
\widetilde{\zeta}_t \cdot (\omega_t,\beta_t,h_t) = (\dot \omega,\dot \beta,\dot h) 
$$
where
\begin{align*}
\dot \omega & =  - \operatorname{Im}\(\widetilde B^{1,1}_t + 2 \IP{\alpha_t \wedge \dbar \widetilde s_t}\) =  - \operatorname{Im} B^{1,1}_t =  - \rho_B(\omega_t)^{1,1} + \i \langle S^{h_t}_{g_t}, F_{h_t} \rangle_{\mathfrak{k}}, \\
2\i \dot \beta & = - \widetilde B^{2,0}_t + \overline{\widetilde B^{0,2}_t} - \IP{\alpha_t \wedge (\partial^{h_0}\widetilde s_t - \partial^{h_t} \widetilde s_t^{*_{h_t}})}\\
& = - B^{2,0}_t + \overline{ B^{0,2}_t} + \IP{\alpha_t \wedge (\partial^{h_t}s_t - \partial^{h_t} s_t^{*_{h_t}})} = - 2 \i \rho_B(\omega_t)^{2,0} + \IP{\alpha_t \wedge \partial^{h_t}S^{h_t}_{g_t}},\\
\dot h & = - h_t(s_t -  s_t^{*_{h_t}}) = - h_t S^{h_t}_{g_t}.
\end{align*}
\end{proof}

Having established the abstract formula \eqref{eq:CPCFholabs}, the equivalence between \eqref{eq:CPFhol} and \eqref{eq:UPCF} follows now by a purely formal argument, using the natural dependence of the infinitesimal gauge symmetry $\zeta^{\bar \ell}$, constructed in Proposition \ref{prop:zetaell}), on the parameters.

\begin{prop}\label{p:PCFholU}
Let $\mathbf{h}_t$ be a one-parameter family of generalized Hermitian metrics on $\mathcal{Q}$ solving \eqref{eq:CPFhol}. Let $\widetilde{\Phi}_t  \in \cG(E^c_0)$ be the family of complex gauge transformations integrating (see Lemma \ref{lem:PCFholabs})
$$
\widetilde{\zeta}_t = (\Phi_{\mathbf{h}_t})_*\zeta_t \in \Lie \cG(E^c_0).
$$
Then, 
$$
\widetilde {\bar \ell}_t := \widetilde \Phi_t^{-1} (\bar \ell_0)
$$
is a solution of \eqref{eq:UPCF}.    Conversely, given ${\bar \ell}_t \subset E_0^c$ a solution of \eqref{eq:UPCF}, denote $\mathbf{h}_0 \equiv (\omega_0,0,h_0)$ and let $\widetilde{\Phi}_t  \in \cG(E^c_0)$ be the one-parameter family of complex gauge transformation integrating
$\zeta^{\bar \ell_t} \in \Lie \cG(E^c_0)$. Then, 
$\mathbf{h}_t = \widetilde{\Phi}_t \mathbf{h}_0$ solves \eqref{eq:CPFhol}.
\end{prop}

\begin{proof}
Note that the element $\zeta^{\bar\ell} \in \Lie \Aut E^c$ defined in Proposition \ref{prop:zetaell} can be regarded as a function on two parameters
$$
\zeta^{\bar\ell} = \zeta(\bar \ell,E),
$$
where $E \subset E^c$ is regarded as a compact form of $E^c$ (see \cite[Definition 4.4]{garciafern2020gauge}). Here, the reduction $h$ is thought of as part of the data determined by the real string algebroid $E$. Recall that $\zeta(\bar \ell,E)$ is constructed by taking suitable local frames of $\bar \ell$ and its conjugate on $E^c = E \otimes \CC$, and solving the D-term equation \eqref{eq:globalsumlocal}, which involves the Dorfman bracket and orthogonal projection on $E^c$. Consequently, this quantity is natural, in the sense that given an isomorphism of complex string algebroids $\Phi \colon E^c \to \tilde E^c$ one has
$$
\Phi_*\zeta(\bar \ell,E) = \zeta(\Phi(\bar \ell),\Phi(E)).
$$
For the first part of the statement, without loss of generality we assume that $\mathbf{h}_0 \equiv (\omega_0,0,h_0)$. By Lemma \ref{lem:PCFholabs}, it follows that \eqref{eq:CPFhol} implies
$$
\dt \widetilde{\Phi}_t^{-1} \mathbf{h}_t = \widetilde{\Phi}_t^{-1}\(- \widetilde{\zeta}_t \cdot  \mathbf{h}_t + \dt \mathbf{h}_t\) = 0,
$$
and therefore $\mathbf{h}_t  = \widetilde{\Phi}_t \mathbf{h}_0$ for all $t$. Consequently, it follows that
$$
\widetilde \Phi_t^{-1}  \Phi_{\mathbf{h}_t}E_t = E_0
$$
for all $t$. We also have 
$$
\widetilde{\zeta}_t = \zeta(\Phi_{\mathbf{h}_t} \bar \ell_t,\Phi_{\mathbf{h}_t} E_t) = \zeta(\bar \ell_0,\Phi_{\mathbf{h}_t} (E_t)), \qquad \zeta^{\widetilde {\bar \ell}_t} := \zeta(\widetilde {\bar \ell}_t,E_0), 
$$
and therefore we conclude 
$$
\dt \widetilde {\bar \ell}_t = - \widetilde \Phi_t^{-1} (\widetilde{\zeta}_t \cdot \bar \ell_0) = - (\widetilde \Phi_t^{-1} \widetilde{\zeta}_t) \cdot \widetilde {\bar \ell}_t = -  \zeta(\widetilde {\bar \ell}_t,\widetilde \Phi_t^{-1}  \Phi_{\mathbf{h}_t}E_t) \cdot \widetilde {\bar \ell}_t = -  \zeta^{\widetilde {\bar \ell}_t} \cdot \widetilde {\bar \ell}_t.
$$
Conversely, given ${\bar \ell}_t \subset E_0^c$ a solution of  \eqref{eq:UPCF}, we consider the family $\mathbf{h}_t = \widetilde{\Phi}_t \mathbf{h}_0$, where $\widetilde{\Phi}_t  \in \Aut E^c_0$ is as in the statement. Then, $\mathbf{h}_t$ satisfies
$$
\dt \mathbf{h}_t = \widetilde{\Phi}_t (\zeta(\bar \ell_t,E_0)\mathbf{h}_0) = \zeta(\widetilde{\Phi}_t (\bar \ell_t),\widetilde{\Phi}_t (E_0))\mathbf{h}_t = \zeta(\bar \ell_0,E_t)\cdot \mathbf{h}_t = \widetilde{\zeta}_t \cdot \mathbf{h}_t,
$$
where we have used that $\widetilde{\Phi}_t^{-1} (\bar \ell_0) = \bar \ell_t$ by definition of the flow \ref{eq:UPCF}. The proof follows by application of Lemma \ref{lem:PCFholabs}.
\end{proof}

\begin{remark}
Regarding \ref{eq:CPFhol} as an evolution equation for $(\omega_t,h_t)$ (see Remark \ref{rem:omegahflow}), it is not difficult to see, using the identity \eqref{eq:naturalS}, that $(\omega_t,A_t)$ is a solution of the first and fourth equations in \eqref{eq:CPFunitary}, where $A_t = \overline \phi_t \cdot A_0$ for $\overline \phi_t$ the one-parameter family of complex gauge transformations induced by $- \tfrac{1}{2}S^{h_t}_{g_t}$. Conversely, given a solution $(\omega_t,A_t)$ of the first and fourth equations in \eqref{eq:CPFunitary}, it follows that $(\omega_t,\overline \phi_t^*h_0)$ is a solution of the first and third equations in \ref{eq:CPFhol}, where $\overline \phi_t$ is the family of complex gauge transformations integrating $-\tfrac{\i }{2}\Lambda_\omega F_{A_t}$. The non-trivial observation made in Proposition \ref{p:PCFholU} is that this simple correspondence can be lifted to the natural flows for the $b$-fields, and that it is furthermore realized via the complicated action \eqref{eq:PicQactexp2} of the complex gauge group of the string algebroid $E_0^c$ on the space of generalized Hermitian metrics.
\end{remark}

\subsection{Aeppli classes and one-form reduction}\label{sec:Aeppli}

Let $\mathcal{Q}$ be a holomorphic string algebroid over $(M,J)$ with underlying principal bundle $(P^c,\bar J)$. We will assume that the space of generalized Hermitian metrics on $\mathcal{Q}$, which we denote $B_{\mathcal{Q}}^+$, is non-empty. By Proposition \ref{propo:Chernclassic}, we can therefore assume that $\mathcal{Q} = \mathcal{Q}_{2\i\partial \omega_0,A^{h_0}}$, for a pair $(\omega_0,h_0)$ satisfying the Bianchi identity
$$
dd^c \omega_0 + \la F_{h_0} \wedge F_{h_0} \ra_\mathfrak{k} = 0.
$$
It was observed in \cite[Section 3.3]{garciafern2018canonical} that any generalized Hermitian metric has an associated Aeppli class, that is, there is a map
\begin{equation}\label{eq:Aclass}
\mathfrak{a} \colon B_{\mathcal{Q}}^+ \to H^{1,1}_A(M,\RR),
\end{equation}
were $H^{1,1}_A(M,\RR)$ denotes the real points of the degree $(1,1)$ Aeppli cohomology group, defined by
\begin{equation}\label{eq:A-cohomology}
\begin{split}
H^{p,q}_{A}(M) & = \frac{\Ker(dd^c\colon \Lambda^{p,q} \to \Lambda^{p+1,q+1})}{\operatorname{Im}(\partial \oplus \dbar \colon \Lambda^{p-1,q}\oplus \Lambda^{p,q-1} \to \Lambda^{p,q})}.
\end{split}
\end{equation}
To recall the construction of $\mathfrak{a}$ we need first to consider secondary characteristic classes introduced by Bott and Chern \cite{BottChern} (see also \cite{BGS,DonaldsonHYM}).

\begin{proposition}[\cite{BGS,DonaldsonHYM}]\label{prop:Donaldson}
For any pair of reductions $h_0,h_1 \in \Gamma(P^c/K)$ there is a secondary characteristic class
\begin{equation}\label{eq:BCinvariant}
R(h_1,h_0) \in \Lambda^{1,1}_\RR/\operatorname{Im}(\partial \oplus \dbar)
\end{equation}
with the following properties:
\begin{enumerate}

\item $R(h_0,h_0) = 0$, and, for any third reduction $h_2$,
$$
R(h_2,h_0) = R(h_2,h_1) + R(h_1,h_0),
$$

\item if $h$ varies in a one-parameter family $h_t$, then
\begin{equation}\label{eqref:BCinvariantder}
\frac{d}{dt}R(h_t,h_0) = \i \IP{h_t^{-1} \dot h_t ,F_{h_t}}_\mathfrak{k}
\end{equation}

\item the following identity holds
$$
dd^c R(h_1,h_0) = \la F_{h_1}\wedge F_{h_1}\ra_\mathfrak{k} - \la F_{h_0}\wedge F_{h_0}\ra_\mathfrak{k}.
$$
\end{enumerate}
\end{proposition}

The \emph{Bott-Chern class} \eqref{eq:BCinvariant} can be defined by integration of \eqref{eqref:BCinvariantder} along a path in the space of reductions of $P^c$. More precisely, given $h_0$ and $h_1$, one defines
\begin{equation}\label{eq:Rtilde}
\tilde R(h_1,h_0) = \i \int_0^1 \IP{h_t^{-1}\dot h_t ,F_{h_t}}_{\mathfrak{k}} dt \in \Lambda^{1,1}_\RR,
\end{equation}
for a choice of path $h_t$ joining $h_0$ and $h_1$. For a different choice of path, $\tilde R(h_1,h_0)$ differs by an element in $\operatorname{Im}(\partial \oplus \dbar)$, and hence one obtains a well-defined class in \eqref{eq:BCinvariant}. 
Following Remark \ref{rem:BIdentity}, any generalized Hermitian metric $(\omega,\beta,h)$ satisfies the Bianchi identity \ref{eq:BIGHM} and therefore it follows from Proposition \ref{prop:Donaldson} that the following map is well-defined.

\begin{definition}[\cite{garciafern2018canonical}]\label{def:Aeppli}
The Aeppli class of a generalized Hermitian metric $\mathbf{h} \equiv (\omega,\beta,h) \in B_{\mathcal{Q}}^+$ is defined by
$$
\mathfrak{a}(\omega,\beta,h) = [\omega - \omega_0 + R(h,h_0)] \in H^{1,1}_A(M,\RR).
$$
\end{definition}

\begin{remark}
More invariantly, Aeppli classes of elements in $B_{\mathcal{Q}}^+$ can be defined without a choice of reference point $(\omega_0,0,h_0)$, as in \cite[Definition 3.20]{garciafern2018canonical}, by means of a natural equivalence relation induced by property (3) in Proposition \ref{prop:Donaldson}. An alternative definition of this quantity has been proposed in \cite{picard2024strominger}.
\end{remark}

To have some control on the Aeppli classes of generalized Hermitian metrics it is convenient to consider an extension of the map \eqref{eq:Aclass}. For this, we set
$$
B_{\mathcal{Q}} := \{(\gamma,\beta,h) \; | \; \textrm{ satisfying } \eqref{eq:structuralGHM}\} \subset \Lambda^{1,1}_\mathbb{R} \oplus \Lambda^{2,0} \oplus \Gamma(P^c/K).
$$
Note that $B_{\mathcal{Q}}^+ \subset B_{\mathcal{Q}}$ corresponds to the open set where the $(1,1)$-form $\gamma$ is positive. Then, we have a well-defined map
\begin{equation}\label{eq:Aclassext}
\mathfrak{a} \colon B_{\mathcal{Q}} \to H^{1,1}_A(M,\RR),
\end{equation}
defined by the same formula as in Definition \ref{def:Aeppli}, replacing $\omega$ by $\gamma$. Consider the first \v Cech cohomology $H^1(\Lambda^{2,0}_{cl})$ of the sheaf of closed $(2,0)$-forms on $(M,J)$, which can be identified with
\begin{equation*}\
H^1(\Lambda^{2,0}_{cl}) \cong \frac{\Ker \; d \colon \Lambda^{3,0} \oplus \Lambda^{2,1} \to \Lambda^{4,0} \oplus \Lambda^{3,1} \oplus \Lambda^{2,2}}{ \operatorname{Im} \; d \colon \Lambda^{2,0} \to\Lambda^{3,0} \oplus \Lambda^{2,1}}.
\end{equation*}
Using \eqref{eq:A-cohomology} we define a map 
\begin{equation}\label{eqannex:partialmap}
\partial \colon H^{1,1}_A(M) \to H^1(\Lambda^{2,0}_{cl})
\end{equation}
induced by the $\partial$ operator on forms acting on representatives. 

\begin{lemma}[\cite{garciafern2018canonical}]
The image of \eqref{eq:Aclassext} is given by
$$
\operatorname{Im} \mathfrak{a} = \Ker \partial \cap H^{1,1}_A(M,\RR).
$$
\end{lemma}

We prove next the desired evolution of the Aeppli class along the pluriclosed flow \ref{eq:CPFhol}.

\begin{proposition}\label{lem:PCFAeppli}
Let $\mathbf{h}_t \equiv (\omega_t,\beta_t,h_t) \in B_{\mathcal{Q}}^+$ be a solution of pluriclosed flow \ref{eq:CPFhol}. Denoting by $\mathfrak{a}_t = \mathfrak{a}(\omega_t,\beta_t,h_t)$ the associated one-parameter family of Aeppli classes, we have
$$
\dt \mathfrak{a}_t = - c_1(M,J) \in H^{1,1}_A(M,\RR),
$$
where $c_1(M,J)$ denotes Aeppli class induced by the first Chern class of $(M,J)$.
\end{proposition}

\begin{proof}
The proof is straightforward from the defining formula of the flow \ref{eq:CPFhol}, combined with property (2) in Proposition \ref{prop:Donaldson}, which implies 
\begin{align*}
\dt \mathfrak{a}_t & = \Bigg{[}-\rho_B(\omega_t)^{1,1} + \i\langle S^{h_t}_{g_t}, F_{h_t} \rangle_{\mathfrak{k}} + \i \IP{h^{-1}_t\dot h_t , F_{h_t}} \Bigg{]}\\
& = - [\rho_B(\omega_t)^{1,1}] = -c_1(M,J).
\end{align*}
\end{proof}

Using the previous result, we define next a `one-form reduction' of the flow \ref{eq:CPFhol}, which will be key for the proof of well-posedness in Proposition \ref{p:STE}. Let $\mathbf{h}_t \equiv (\omega_t,\beta_t,h_t) \in B_{\mathcal{Q}}^+$ be a solution of pluriclosed flow \ref{eq:CPFhol} and choose a background Hermitian metric $\til{\gw}$. By Proposition \ref{lem:PCFAeppli}, there exists a one-parameter family $\xi_t \in \Lambda^1_\RR$ such that
$$
\omega_t =  \omega_0 + (d\xi_t)^{1,1} -  \tilde R(h_t,h_0) - t \rho_C(\til{\gw}),
$$
where $\rho_C(\til{\gw}) \in \Lambda^{1,1}_\RR$ is the representative of $c_1(M,J) \in H^{1,1}_A(M,\RR)$ induced by the volume $\til{\gw}^n/n!$. Arguing similarly as in the proof of \cite[Lemma 5.7]{garciafern2018canonical}, it can be proved that $\xi_t$ can be taken to be a smooth function of $t$, and therefore 
$$
\dt(d\xi_t)^{1,1} = - \rho_B(\omega)^{1,1} + \rho_C(\til{\gw}) =  (dd^{\star_t}\omega_t)^{1,1} + \frac{1}{2}dd^c\log \frac{\omega_t^n}{ \til{\gw}^n}.
$$
This motivates our next result.

\begin{lemma}\label{lem:oneformred} 
Let $(\xi_t,h_t)$ be a one-parameter family of pairs, where  $\xi_t \in \Lambda^{1,0}$ is a complex one-form and $h_t \in \Gamma(P^c/K)$, solving the evolution equation
\begin{equation*}
\begin{split}
\dt \xi =&\ \delb^{\star}_{\gw} \gw - \frac{\i}{2} \del \log \frac{\gw^n}{\til{\gw}^n},\\
h^{-1}\dt h =&\ - S^h_g,\\
\dt \hat{\omega} =&\ - \rho_C(\til{\gw}) + \i \IP{S^h_g, F_h}_{\mathfrak k}, \qquad \gw = \hat{\gw} + \delb \xi + \del \overline{\xi}.
\end{split}
\end{equation*}
Then, $(\omega_t,h_t)$ solves the pluriclosed flow \ref{eq:CPFhol}.

\begin{proof} 
It suffices to prove that $\omega_t$ satisfies the first equation in \ref{eq:CPFhol}, but this follows trivially from
\begin{align*}
\dt \omega_t & = \delb \delb^{\star}_{\gw} \gw + \del \del^{\star}_{\gw} \gw + \frac{1}{2} dd^c \log \frac{\gw^n}{\til{\gw}^n} - \rho_C(\til{\gw}) + \i \IP{S^h_g, F_h}_{\mathfrak k}\\
& = (dd^{\star_t}\omega_t)^{1,1} - \rho_C(\omega_t)  + \i \IP{S^h_g, F_h}_{\mathfrak k}\\
& = - \rho_B(\omega_t)^{1,1}  + \i \IP{S^h_g, F_h}_{\mathfrak k}.
\end{align*}
\end{proof}
\end{lemma}

\begin{remark}
When $(M,J)$ admits a global holomorphic volume form $\Omega$, the evolution for the (real) one-form potential $\xi_t \in \Lambda^1_\RR$ can be more elegantly written as 
$$
\dt \xi - \frac{\i}{2} h^{-1}\dt h  = \operatorname{Im} \varepsilon
$$
where $\xi \in \Lambda^1_\RR$ is the  (real) one-form potential and $\varepsilon$ is the solution of the D-term equation \eqref{eq:globalsumlocal}, given by (cf. \eqref{eq:varepsilonDterm}) 
\begin{equation*}
\varepsilon = d \log \|\Omega\|_{\omega} + \i \left(d^{\star}\omega - d^c\log \|\Omega\|_{\omega} \right) - \frac{1}{2} S^{h}_{g}.
\end{equation*}
\end{remark}

\begin{remark}\label{rmk:twist}
For our applications, we will need to fix an initial pair $(\omega_0,h_0)$ 
and a background pair $(\omega',h')$, both satisfying the Bianchi identity \eqref{eq:BIGHM}. We will not require that the associated holomorphic string algebroids $ \mathcal{Q}_{2\i\partial \omega_0,A^{h_0}}$ and $ \mathcal{Q}_{2\i\partial \omega',A^{h'}}$ are isomorphic, but the weaker condition of being isomorphic as holomorphic vector bundles. In practice, this boils down to the cohomological condition  
\begin{equation}\label{eq:GHermitianexpweak}
[\omega' - \omega_0 + R(h',h_0)] \in \Ker \partial^D \subset H^{1,1}_A(M),
\end{equation}
where 
$$
\partial^D \colon H^{1,1}_A(M) \to H^{2,1}_{\overline{\partial}}(M) 
$$
is the natural map to Dolbeault cohomology induced by the $\partial$ operator on forms acting on representatives (see \cite[Remark 3.13]{GFGM} and \cite[Lemma 3.23]{garciafern2018canonical}). 
\end{remark}

\section{Dimensional reduction} \label{s:dimredpcf}

In this section we show that solutions to generalized Ricci flow on string algebroids are equivalent to solutions of generalized Ricci flow satisfying a certain symmetry ansatz on an exact Courant algebroid. This implies that all of the analytic theory developed for generalized Ricci flow on exact Courant algebroids applies immediately to the string algebroid case, and we detail this in \S \ref{s:geometrization}. We then study the dimensional reduction of pluriclosed flow. 

\subsection{Dimensional reduction of the Bismut Ricci tensor}\label{sec:dimensionalred}

Let $M$ be a smooth manifold and $p \colon P \to M$ a smooth principal $K$-bundle, with quadratic Lie algebra $(\mathfrak{k},\langle\cdot,\cdot\rangle_{\mathfrak{k}})$. We compute the dimensional reduction, from the total space of $P$ to $M$, of several geometric quantities that will be useful in the sequel. We use throughout the notations introduced in Section \ref{s:gpb}. For the dimensional reduction calculations, we will make an extensive use of standard reference \cite[Chapter 9]{Besse}.

We start by studying the geometry on the fibres of the principal bundle. By compactness, we can extend $-\langle\cdot,\cdot\rangle_{\mathfrak{k}}$ to a (possibly indefinite) bi-invariant metric $g_K$ on $K$. Given left-invariant fields $U,V,W$ on $K$, the Cartan $3$-form $H_K\in \Lambda^3_K$ is defined by
$$
H_K(U,V,W) = g_K([U,V],W),
$$
and the Levi-Civita connection of $g_K$ satisfies
\begin{equation}\label{eq:LeviCK}
\nabla^{g_K}_U V = \tfrac{1}{2}[U,V].
\end{equation}
As a consequence, we have the following (see e.g. \cite[Proposition 3.53]{GRFBook}):

\begin{lemma}
    Let $\nabla^+=\nabla^{g_K}+\tfrac{1}{2}g_K^{-1}H_K$ be the connection with skew-symmetric torsion $H_K$. Then, for left-invariant vector fields $U, V$, we have:
    $$
    \nabla^+_U V =[U,V].
    $$
    In particular, $\nabla^+$ is flat.
\end{lemma}

We turn next to the geometry on the total space of the bundle. Given a connection $A$ on $P$, the \emph{Chern-Simons $3$-form} $CS(A) \in \Lambda^3_P$ is given by:
\begin{align}\label{f:CS}
CS(A)=-\tfrac{1}{6} \langle A\wedge [A\wedge A]\rangle_{\mathfrak{k}} + \langle F_A \wedge A\rangle_{\mathfrak{k}}.
\end{align}
Observe that $CS(A)$ restricts to $H_K$ on the fibres of $P$ and provides a potential for the representative of the first Pontryagin class determined by $A$:
\begin{align}\label{f:dCS}
dCS(A)=\langle F_A \wedge F_A \rangle_{\mathfrak{k}}.
\end{align}

\begin{defn} \label{d:invariantdata} Let $A_0$ be a background principal connection on $P$. Given $(g,H,A)$ a Riemannian metric on $M$, $3$-form on $M$, and principal connection on $P$ as above, we set $a = A - A_0$ and define
\begin{align}\label{f:gtot}
\gtot =&\ p^* g + g_K(A, A),\\
\label{f:btot}
\overline{b}= &\ p^*b-\langle a\wedge A\rangle_{\mathfrak{k}},\\
\label{f:Htot}
\Htot =&\ p^* H - CS(A).
\end{align}
\end{defn}

The following result is straightforward, by construction.

\begin{lemma} The data $(\gtot, \Htot)$ is $K$-invariant.  Furthermore
\begin{align} \label{f:anomalyterm}
    d \Htot =p^* ( d H - \IP{ F_A \wedge F_A}_{\mathfrak{k}}).
\end{align}
\end{lemma}

We fix $(\gtot, \Htot)$ as above and adopt the following notation. Consider the orthogonal splitting
$$
TP \cong_A p^*T \oplus VP,
$$
where $p^*T$ is identified with the horizontal distribution of $A$. We denote by $X,Y,Z$ horizontal lifts with respect to $A$ of vector fields on $M$, while $U,V,W$ denote canonical vector fields. Furthermore, $\{v_i\}$ will denote an orthonormal frame of $p^*T$. Though the following results are valid for any signature, the proof will be given for simplicity in the definite case, hence also $\{U_j\}$ will denote an orthonormal frame of $VP$. 

By \eqref{eq:LeviCK} it is easy to see that $H_K$ is parallel with respect to $\nabla^{g_K}$. In particular, it holds:
\begin{align*} \label{f:HKeqs}
dH_K = d^\star_{g_K} H_K = 0.
\end{align*}
The Levi-Civita connection $\overline{\nabla}=\nabla^{\gtot}$ is decomposed \cite[Definition 9.25]{Besse} as:
\begin{equation}\label{f:conndecomposition}
\begin{split}
\overline{\nabla}_U V & = \nabla^{g_K}_U V+T_U V\\
\overline{\nabla}_U X & =T_U X + A^{\perp} \overline{\nabla}_U X\\
\overline{\nabla}_X U & = A \overline{\nabla}_X U + I\!I_X U \\
\overline{\nabla}_X Y & = I\!I_X Y + A^{\perp} \nabla_X Y,
\end{split}
\end{equation}
where we have introduced the tensors \cite[Definitions (9.17), (9.20)]{Besse}:
\begin{equation*}
\begin{split}
T & = A^{\perp} \overline{\nabla}_{A} A + A \overline{\nabla}_{A} A^{\perp}\\
I\!I
 & = A^{\perp} D_{A^{\perp}} A + A D_{A^{\perp}} A^{\perp}.
\end{split}
\end{equation*}
%These tensors satisfy the identities \cite[(9.18), (9.21)]{Besse} which we will make extensive use of, and 
Moreover one has:
\begin{align}\label{f:AtensorF}
I\!I_X Y =\tfrac{1}{2}A([X,Y])=-\tfrac{1}{2}F_A(X,Y)
\end{align}
under the identification $VP \cong P\times \mathfrak{k}$.

Now consider the connection with skew-symmetric torsion (see \eqref{eq:nablapm})
\begin{align}\label{def:BismutconngH}
\overline{\nabla}^+=\overline{\nabla} + \tfrac{1}{2}\gtot^{-1}\Htot.
\end{align}
Given a smooth function $\phi \in C^{\infty}(M)$, we set $\bar{\phi} = p^* \phi$ and compute the \emph{weighted Ricci tensor} of $\bar{\N}^+$, defined by:
\begin{equation}\label{eq:LC-wBRic}
    \begin{split}
\Rc^{\bar{H},\bar{\phi}}= \Rc^{\bar{g}}-\tfrac{1}{4}\bar{H}^2 -\tfrac{1}{2}d^{\star}_{\bar{g}} \bar{H} + \overline{\nabla}^+ (d\bar{\phi}).
    \end{split}
\end{equation}
We next give an explicit formula for this tensor in the ansatz above.  The statement below uses the conventions for the symmetric tensors $H^2$ and $F_A^2$ given by \eqref{f:sqtens}.

\begin{prop}\label{prop:BismutRic} With the previous notations, one has
\begin{align}\label{f:RcgHcomponents}
\begin{split}
\Rc^{\Htot,\overline{\phi}}(X,Y) & = \left( \Rc^g -\tfrac{1}{4}H^2 - F_A^2 -\tfrac{1}{2}d^{\star}_gH + \nabla^+(d\phi) \right)(X,Y), \\
\Rc^{\Htot,\overline{\phi}}(U,X) & = -\langle i_X(d^{\star}_A F_A - F_A \ \lrcorner \ H+ i_{\nabla \phi} F_A) , U \rangle_{\mathfrak{k}},\\
\Rc^{\Htot,\overline{\phi}}(X,U) & = 0,\\
\Rc^{\Htot,\overline{\phi}}(U,V) & =0.
\end{split}
\end{align}
\end{prop}

\begin{proof} We first observe as a direct consequence of \cite[Proposition 9.36]{Besse} that
\begin{gather} \label{f:LCRcdecomp}
\begin{split}
\Rc^{\gtot}(X,Y) =&\ (\Rc^g-\tfrac{1}{2}F_A^2)(X,Y),\\
\Rc^{\gtot}(U,X) =&\ -\tfrac{1}{2}\langle d_A^{\star}F_A(X),U \rangle_{\mathfrak{k}}, \\
\Rc^{\gtot}(U,V) =&\ {\Rc}^{g_K}(U,V)+\tfrac{1}{4}\textstyle \sum_{j,k} \langle F_A(v_j,v_k),U\rangle_{\mathfrak{k}} \langle F_A(v_j,v_k),V\rangle_{\mathfrak{k}}.
\end{split}
\end{gather}
Building on this, we first compute the components of $\bar{H}^2$:
\begin{equation*}
    \begin{split}
\Htot^2(X,Y) = & \textstyle \sum_{j,k} \Htot(v_j,v_k,X)\Htot(v_j,v_k,Y) + 2\Htot(v_j,U_k,X)\Htot(v_j,U_k,Y)\\
 = & \ H^2(X,Y)+2 \textstyle \sum_{j,k} \langle F_A(v_j,X),U_k \rangle_{\mathfrak{k}} \langle F_A(v_j,Y),U_k \rangle_{\mathfrak{k}}\\
 = & \ (H^2 + 2F_A^2)(X,Y),\\
\Htot^2(U,X) = & \ \textstyle \sum_{j,k} \Htot(v_j,v_k,U)\Htot(v_j,v_k,X) = - \textstyle \sum_{j,k} H(v_j,v_k,X)\langle F_A(v_j,v_k),U \rangle,\\
 \Htot^2(U,V) = & \textstyle \sum_{j,k} \Htot(v_j,v_k,U)\Htot(v_j,v_k,V)+\Htot(U_j,U_k,U)\Htot(U_j,U_k,V)\\
 = & \textstyle \sum_{j,k} \langle F_A(v_j,v_k),U\rangle_{\mathfrak{k}} \langle F_A(v_j,v_k),V \rangle_{\mathfrak{k}} + \textstyle \sum_{j,k} \langle U_j,[U_k,U] \rangle_{\mathfrak{k}} \langle U_j,[U_k,V] \rangle_{\mathfrak{k}},
\end{split}
\end{equation*}
Now we turn to the calculation of $- \tfrac{1}{2} d^{\star}_{\bar{g}} \bar{H}$. For its horizontal component, we have:
\begin{equation*}
\begin{split}
d^{\star}_{\bar{g}}\Htot(X,Y) & = \textstyle \sum_j - (\overline{\nabla}_{v_j}\Htot)(v_j,X,Y)+ \textstyle \sum_k -(\overline{\nabla}_{ U_k}\Htot)(U_k,X,Y).
\end{split}
\end{equation*}
A computation shows that the second term vanishes.  The first term is
\begin{equation*}
\begin{split}
- (\overline{\nabla}_{v_j}& \Htot)(v_j,X,Y)\\
=&\ \textstyle \sum_j - v_j(\Htot(v_j,X,Y))+\Htot(\overline{\nabla}_{v_j}v_j,X,Y) + \Htot(v_j,\overline{\nabla}_{v_j}X,Y)+\Htot(v_j,X,\overline{\nabla}_{v_j}Y)\\ 
= & \  d^{\star}_gH(X,Y)+\textstyle \sum_j \Htot(v_j,I\!I_{v_j}X,Y)+\Htot(v_j,X,I\!I_{v_j}Y)\\
 =&\ d^{\star}_g H(X,Y)-\tfrac{1}{2}\textstyle \sum_{j,k} \langle F_A(v_j,Y),F_A(v_j,X) \rangle_{\mathfrak{k}} +\tfrac{1}{2} \textstyle \sum_{j,k} \langle F_A(v_j,X),F_A(v_j,Y)\rangle_{\mathfrak{k}}\\
=&\ d^{\star}_g H(X,Y).
\end{split}
\end{equation*}
Thus observe $d^\star_{\bar{g}}\bar{H}(X,Y) = d^\star_g H(X,Y)$. Next we have
\begin{equation*}
\begin{split}
d^{\star}_{\bar{g}} \Htot (U,X) =&\ \textstyle \sum_{j} -v_j(\Htot(v_j,U,X))+ \Htot(\overline{\nabla}_{v_j}v_j,U,X) + \textstyle \sum_{j} \Htot(v_j,\overline{\nabla}_{v_j}U,X) +\\ &\ +\Htot(v_j,U,\overline{\nabla}_{v_j}X)
     + \textstyle \sum_{k} -U_k(\Htot(U_k,U,X))+ \Htot(\overline{\nabla}_{U_k}U_k,U,X) +\\
     &\ + \textstyle \sum_{j} \Htot(U_k,\overline{\nabla}_{U_k}U,X) + \Htot(U_k,U,\overline{\nabla}_{U_k}X)\\
     =&\ \textstyle \sum_{j} v_j(-\langle F_A(v_j,X),U \rangle_{\mathfrak{k}})+\langle F_A(\nabla_{v_j}v_j,X),U \rangle_{\mathfrak{k}} + \langle F_A(v_j,\nabla_{v_j}X),U \rangle_{\mathfrak{k}}\\
     &\ -H(v_j,X,I\!I_{v_j}U)\\
    = &\ \langle d^{\star}_A F_A(X),U \rangle_{\mathfrak{k}} -\tfrac{1}{2}\textstyle \sum_{j,k} H(v_j,v_k,X)\langle F_A(v_j,v_k),U \rangle_{\mathfrak{k}}.
\end{split}
\end{equation*}
Finally a further computation shows that $d^{\star}_{\bar{g}} \bar{H}(U,V) = 0$.
Regarding the components of $\overline{\nabla}^+(d\phi)$, using the explicit description of $\bar{H}$ we obtain nonvanishing terms, while the remaining components are easily seen to vanish:
\begin{align*}
    \overline{\nabla}^+(d\phi)(X,Y) = & X(d\phi(Y))- d\phi(\overline{\nabla}_X Y+\gtot^{-1}\Htot(X,Y,\cdot)\\
    = & \ X(d\phi(Y))- d\phi(\nabla_X Y+g^{-1}H(X,Y,\cdot))\\
    = & \ \nabla^+(d\phi)(X,Y),\\
    \overline{\nabla}^+(d\phi)(U,X) = & \ -d\phi(\overline{\nabla}_U X + \tfrac{1}{2}\gtot^{-1}\Htot(U,X,\cdot))\\
    = & \ -\gtot(\overline{\nabla}_U X,\nabla\phi)+\tfrac{1}{2}d\phi\left(\langle g^{-1}i_X F_A,U\rangle_{\mathfrak{k}}\right)\\
    = & \ -\langle U, F_A(\nabla \phi, X) \rangle_{\mathfrak{k}}.
    \end{align*}
Combining all the previous computations yields the first three stated formulae.  For the fourth, we also take into account:
\begin{align*}
\Rc^{\Htot,\overline{\phi}}(U,V) = & \;  (\mathrm{Rc}^{\gtot}-\tfrac{1}{2}d^\star_{\gtot}\Htot-\tfrac{1}{4}\Htot^2 + \overline{\nabla}^+(d\phi))(U,V)\\
= & \; \mathrm{Rc}^{g_K}(U,V) + \tfrac{1}{4}\textstyle \sum_{j,k}\langle F_A(v_j,v_k),U\rangle_{\mathfrak{k}}\langle F_A(v_j,v_k),V\rangle_{\mathfrak{k}}\\
& -\tfrac{1}{4}\left(\textstyle \sum_{j,k}(U_j,[U_k,U])(U_j,[U_k,V]) +  \langle F_A(v_j,v_k),U\rangle_{\mathfrak{k}}\langle F_A(v_j,v_k),V\rangle_{\mathfrak{k}}\right)\\
= & \; \left(\Rc^{g_K}-\tfrac{1}{2}d_{g_K}^\star H_K-\tfrac{1}{4}H_K^2\right)(U,V)\\
= & \; 0,
\end{align*}
where the last step follows from the fact that the fibre Bismut connection is flat.
\end{proof}

We turn next to the calculation of the relevant scalar quantities:

\begin{prop}\label{prop:Bismutscalar} Given $(\bar{g}, \bar{H})$ and $\bar{\phi} = p^* \phi \in C^{\infty}(P)$, one has
\begin{align}
\label{f:nH2}
    \brs{\bar{H}}^2 =&\ 
     \brs{H}^2 + 3\brs{F_A}^2 + \brs{[\cdot,\cdot]}^2 \\
\label{f:BaEmR}
    R^{\bar{H},\bar{\phi}} = & \ R_g - \tfrac{1}{12} \brs{H}^2 -\tfrac{1}{2}|F_A|^2+\tfrac{1}{6}|[\cdot,\cdot]|^2+2\Delta_g \phi - |\nabla \phi|.
\end{align}

\end{prop}
\begin{proof}
    Observe that the three natural components of $\Htot = p^*H - CS(A)$ are $\gtot$-orthogonal to each other.  
    %To compute, note first that $\overline{g}$ is in general indefinite in the bundle directions.
    We choose a frame $\{v_i\}$ on $P$ by combining $\{e_j\}$ orthonormal in the horizontal distribution and $\{u_k\}$ orthogonal on the vertical one. With this, we have:
    \begin{align*}
        |\langle F_A\wedge A\rangle_{\mathfrak{k}}|^2 = & \ \textstyle \sum_{i_1,i_2,i_3} \left(\langle F_A\wedge A\rangle_{\mathfrak{k}}(v_{i_1},v_{i_2},v_{i_3})\right)^2 |v_{i_1}\wedge v_{i_2}\wedge v_{i_3}|^2\\        
        = & \ 3\textstyle \sum_{i,j,k} \langle F_A(e_i,e_j),u_k\rangle_{\mathfrak{k}}\langle F_A(e_i,e_j),u_k\rangle_{\mathfrak{k}}\overline{g}(u_k,u_k)\\
        = & \ 3|F_A|^2,\\
        |\tfrac{1}{6}\langle A\wedge [A\wedge A]\rangle_{\mathfrak{k}}|^2 = & \ \tfrac{1}{36} \textstyle \sum_{i,j,k} \left(\langle A\wedge[A \wedge A]\rangle_{\mathfrak{k}}(u_i,u_j,u_k)\right)^2 |u_i\wedge u_j \wedge u_k |^2\\
        = & \ -\textstyle \sum_{i,j,k} \langle u_i,[u_j,u_k]\rangle_{\mathfrak{k}}^2\langle u_i, u_i \rangle_{\mathfrak{k}} \langle u_j, u_j \rangle_{\mathfrak{k}} \langle u_k, u_k \rangle_{\mathfrak{k}}\\
        = & \ |[\cdot,\cdot]|^2.
    \end{align*}
For the second item, we compute separately the first, third and last terms in the formula
$$
R^{\Htot,\bar{\phi}}=R_{\gtot} - \tfrac{1}{12}|\Htot|^2 + 2\Delta_{\gtot} \phi - |\overline{\nabla}\phi|^2.
$$
First, from (\ref{f:LCRcdecomp}) we obtain
\begin{align*}
    R_{\gtot} = & \ \textstyle \sum_i \mathrm{Rc}_{\gtot}(v_i,v_i) + \textstyle \sum_{j} \mathrm{Rc}_{\gtot}(u_j,u_j)|u_j|^2\\
    = & \ R_{g} - \tfrac{1}{2}|F_A|^2 + R_{g_K} +\tfrac{1}{4}|F_A|^2\\
    = & \ R_{g} - \tfrac{1}{4}|F_A|^2 + R_{g_K}.
\end{align*}
Now, since $\nabla^+$ on $K$ is flat and $H_K$ is $g_K$-harmonic, in particular the Ricci curvature vanishes, i.e. $\mathrm{Rc}_{g_K} -\tfrac{1}{4}H_K^2 =0$. Hence, taking trace it follows from the above formulae
\begin{align*}
    R_{\gtot} = & \ R_{g} - \tfrac{1}{4}|F_A|^2 +\tfrac{1}{4}|H_K|_{g_K}^2 = R_{g} - \tfrac{1}{4}|F_A|^2 +\tfrac{1}{4}|[\cdot,\cdot]|^2.
\end{align*}
An elementary computation shows that $\gD_{\bar{g}} \phi = \gD_g \phi$ and $\brs{\overline{\nabla} \phi}^2 = \brs{\N \phi}^2$.  Therefore, combining the above computations and the first item, the result follows.
\end{proof}

\subsection{Dimensional reduction and flow equations}\label{s:dimredGRF}

Let $E$ be a real string algebroid over $M$ with underlying principal bundle $P$ and quadratic Lie algebra $(\mathfrak{k},\langle\cdot,\cdot\rangle_{\mathfrak{k}})$ (see Section \ref{sec:background}). In this section we show that string algebroid generalized Ricci flow \eqref{eq:GRFgbA} is equivalent to generalized Ricci flow \eqref{f:GRF} on $P$, with a suitable invariant ansatz. We will then turn to the situation of string algebroid pluriclosed flow \eqref{eq:CPFunitary} on a complex manifold, for which the setting is more involved and we will obtain a partial answer.

%Recall here that given a triple $(g,H,A)$, we can define on the total space of $P$ a metric, a $3$-form and a $2$-form $(\gtot,\Htot,\overline{b})$ as in Definition \ref{d:invariantdata}. Then, 
We start in the case of a real string algebroid $E$ over a smooth manifold $M$.  Upon fixing a generalized metric on $E$, we obtain a triple $(g_0,H_0,A_0)$, and associated $K$-invariant data $(\gtot_0,\Htot_0,\overline{b}_0)$ as in Definition \ref{d:invariantdata}.

\begin{prop} \label{p:cGRFtoGRF} Let $(g_t, b_t, A_t)$ denote the unique solution to generalized Ricci flow \eqref{eq:GRFgbA} on $E$  with initial data $(g_0,H_0,A_0)$. Then $(\bar{g}_t, \bar{b}_t)$, defined in \eqref{f:gtot} and \eqref{f:btot}, is the unique solution to generalized Ricci flow \eqref{f:GRF} on $P$ with initial data $(\bar{g}_0, \bar{b}_0)$.  
\end{prop}

\begin{proof} 
Combining \eqref{f:gtot} with the flow equations \eqref{eq:GRFgbA} gives:
\begin{equation*}
    \begin{split}
        \frac{\partial}{\partial t} \gtot = & \ p^*\frac{\partial}{\partial t} g + 2g_K\left(\frac{\partial}{\partial t} A \odot A\right)\\
        = & \ p^*(-2\Rc^g + \tfrac{1}{2}H^2 + 2F_A^2) + 2\langle (d_A^\star F_A -  F_A \ \lrcorner \ H)\odot A \rangle_{\mathfrak{k}},
    \end{split}
\end{equation*}
where $\odot$ denotes symmetric product. Comparing against Proposition \ref{prop:BismutRic}, the metric $\bar{g}$ satisfies the required equation for generalized Ricci flow.  Next we compute
\begin{equation*}
    \begin{split}
        \frac{\partial}{\partial t}\bar{b} = & \ p^*\frac{\partial}{\partial t} b - g_K\left(\frac{\partial}{\partial t} A\wedge a\right)- g_K\left(A\wedge \frac{\partial}{\partial t} a\right)\\
        = & \ p^*(-d_g^\star H + \langle (d_A^\star F_A - F_A \lrcorner \ H )\wedge a\rangle_{\mathfrak{k}})- \langle(d_A^\star F_A- F_A \ \lrcorner \ H)\wedge a\rangle_{\mathfrak{k}}\\
        & -\langle A\wedge (d_A^\star F_A - F_A \ \lrcorner \ H)\rangle_{\mathfrak{k}}\\
         = & \ p^*(-d_g^\star H) - \langle A\wedge (d_A^\star F_A- F_A \ \lrcorner \ H)\rangle_{\mathfrak{k}}.
         \end{split}
    \end{equation*}
Comparing against the computations in Proposition \ref{prop:BismutRic}, the flow equation for $\bar{b}$ follows.
\end{proof}

\begin{remark}
    The correspondence of generalized Ricci flow lines of Proposition \ref{p:cGRFtoGRF} hints at a Courant algebroid dimensional reduction point of view in the spirit of \cite{BarHek,GF14} (see also \cite{BursztynCavalcantiGualtieri}). This perspective is however not the main focus of this work and will be investigated elsewhere. 
\end{remark}
 
We turn next to the dimensional reduction of pluriclosed flow.  As it turns out, although we are able to set a correspondence of flow lines between pluriclosed flow and its reduced version (cf. Proposition \ref{prop:UPCFexp}), the flow will no longer preserve the symmetry, in contrast with the case of generalized Ricci flow. The reason for this may be partly because the natural idea of reducing pluriclosed flow on a holomorphic principal bundle is lacking a good notion of invariant metrics, unless the group is abelian.  Therefore, here we resort to a different geometric set-up. Following the previous notation, we assume that $M$ admits a complex structure $J$ and furthermore that the group $K$ is  compact and even-dimensional.  We endow $K$ with a left-invariant, integrable complex structure $J_K$, as studied in \cite{Samelson,Pit88}.  Then, the bundle $P$ is naturally endowed with a Hermitian structure (cf. \cite{PoddarThankur}):

\begin{definition}
Let $\omega$ be a Hermitian metric on $(M,J)$ and $A$ a principal connection on $P$ satisfying $F_A^{0,2}=0$. Then the total space of $P$ is a Hermitian manifold with the complex structure and Hermitian metrics given by
\begin{align}\label{Jtot}
    \Jtot & =A^\perp Jdp + J_K \circ A,\\
    \label{otot}
    \otot & = p^*\omega + \frac{1}{2}\langle \Jtot A\wedge A\rangle_{\mathfrak{k}}.
\end{align}
\end{definition}
\noindent We note that $\otot$ is the Hermitian metric defined by the total space metric \eqref{f:gtot} and $\Jtot$. The fact that $\Jtot$ is integrable is well-known. Nonetheless, we include here a short proof for the benefit of the reader. 

\begin{lemma}\label{l:Jtotintegrable}
    The complex structure $\bar{J}$ is integrable.
\end{lemma}
\begin{proof}
    By Newlander-Nirenberg's theorem, it is enough to show that $T^{0,1}P$ is an involutive distribution. In the following computation, vertical fields $U,V$ are canonical and horizontal fields $X,Y$ are $K$-equivariant:
    \begin{equation*}
    \begin{split}
        [U^{0,1},V^{0,1}]^{1,0} = & \  \tfrac{1}{4} N_{J_K}(U,V)^{1,0} =  0 ,\\
        [X^{0,1},U^{0,1}]^{1,0} = & \ ([X,U]+i[JX,U]+i[X,J_K U]-[JX,J_K U])^{1,0} = 0, \\
        [X^{0,1},Y^{0,1}]^{1,0} = & A^\perp([X^{0,1},Y^{0,1}]^{1,0}) -F_A(X^{0,1},Y^{0,1})^{1,0} = 0.
    \end{split}
    \end{equation*}
    We have used that $J$ and $J_K$ are integrable, and $F_A^{0,2}=0$. By taking linear combinations it follows that the Lie bracket preserves $T^{0,1}P$, as required.
\end{proof}

Now, we provide the dimensional reduction of some Hermitian quantities that will be used in the sequel. Here, we just indicate the main steps in the computations, while the details can be found in \cite{GMolina}.

\begin{proposition}\label{prop:hermred}
    With the conventions above, we have
    \begin{align}\label{eq:pctoBianchi}
        dd^c\otot & = p^*(dd^c\omega + \langle F_A \wedge F_A \rangle_{\mathfrak{k}}).
    \end{align}
    Moreover, in case $(\omega,A)$ satisfy the Bianchi identity, we also have
    \begin{align}\label{eq:rhotot}
        \overline{\rho}_B & = \rho_B +\langle \Lambda_{\omega} F_A,F_A\rangle_{\mathfrak{k}} + \langle d_A(\Lambda_{\omega} F_A)\wedge A\rangle_{\mathfrak{k}} + \tfrac{1}{2}\langle [\Lambda_{\omega}F_A, A]\wedge A\rangle_{\mathfrak{k}}.
    \end{align}
\end{proposition}
\begin{proof}
The first item follows from \eqref{f:anomalyterm} together with
\begin{equation*}
    \begin{split}
        -d^c\overline{\omega} & =p^*(-d^c\omega) + \tfrac{1}{2}\langle [\overline{J} A \wedge \overline{J} A]\wedge A \rangle_{\mathfrak{k}} - \langle F_A \wedge A \rangle_{\mathfrak{k}})\\
        & = p^*H +\tfrac{1}{6}\langle A\wedge [A\wedge A]\rangle_{\mathfrak{k}} - \langle F_A\wedge A\rangle_{\mathfrak{k}}\\
        & = p^*H -CS(A),
    \end{split}
\end{equation*}
where we have used the integrability of $\overline{J}$.
The final computation of $\overline{\rho}_B$ combines the previous items with Proposition \ref{prop:BismutRic} and \cite[Formula (3.16)]{Ivanovstring}.
\end{proof}

Let $A_t$ be a one-parameter family of principal connections on $P$, and consider associated family of complex structures $\overline{J}_t$ given by \eqref{Jtot}:
    \begin{align}\label{Jtott}
    \overline{J}_t = A_t^\perp J dp + J_K \circ A_t.
    \end{align}
Then, one has that
\begin{align}\label{eq:evolJtot}
        \partial_t \overline{J} = J\partial_t A+ J_K \circ \partial_t A.
    \end{align}   
It is therefore clear that along string algebroid pluriclosed flow \eqref{eq:CPFunitary}, the complex structure on $P$ given by \eqref{Jtott} is not fixed. The main difficulty in producing a natural, fixed complex structure on $P$ is the fact that \eqref{Jtott} is not $K$-equivariant, unless $(K, J_K)$ is a complex Lie group. In our setting, where we demand that moreover $K$ is compact, this reduces to just torus bundles. In the abelian case, the next result solves this issue.

\begin{lemma}\label{l:Jtotcte}
    Assume $K=T^{2k}$ is a torus endowed with a bi-invariant complex structure, denoted by $J_K = J_T$. Moreover, let $(\omega_t,A_t)$ be a one-parameter family of pairs consisting of a Hermitian metric on $(M,J)$ and a connection on $P$. Suppose moreover that
    \begin{align}\label{eq:evolAred}
    \partial_t A_t = Jd_A(\Lambda_\omega F_A)- d_A(J_T\Lambda_\omega F_A).
    \end{align}
    Then, the complex structure $\overline{J}_t$ given by \eqref{Jtott} is constant along the flow. In particular, if $F_{A_0}^{0,2}=0$, then $A_t$ remains integrable along the flow.
\end{lemma}
\begin{proof}
Using \eqref{eq:evolJtot}, we have
\begin{equation*}
    \begin{split}
        \partial_t \overline{J}_t = & \ J\partial_t A_t + J_T \circ \partial_t A_t\\
        = & -d_A(\Lambda_\omega F_A)-Jd_A(J_T\Lambda_\omega F_A) + J_T \circ Jd_A(\Lambda_\omega F_A)-J_T \circ d_A(J_T \Lambda_\omega F_A) = 0,
    \end{split}
\end{equation*}
where we have used $[J_T,d_A]=0$ on $\mathrm{ad} P$, as the structure group is abelian.
\end{proof}

\begin{remark}\label{rmk:unitaryhol}
After identifying $(\mathrm{ad} P, J_T) \cong (\mathrm{ad} P^{1,0},\i)$, the flow equation \eqref{eq:evolAred} reads
$$
\partial_t A_t = -\partial_A(\i\Lambda_\omega F_A).
$$
Then, as discussed in Section \ref{sec:PCFstring}, if $A_0$ is compatible with a Hermitian structure $h_0$ on $P$, then $A_t$ is the Chern connection of $h_t$ satisfying
$$
h^{-1}\partial_t h = -S^h_g,
$$
where $S^h_g = \i \Lambda_\omega F_{A_h}$ for $A_h=A$, with the previous notations. Therefore, the flow in Lemma \ref{l:Jtotcte} is compatible with pluriclosed flow in Definition \ref{def:PCFQ}. 
\end{remark}

We are now ready to prove the main result of this section, which generalizes previously known results \cite{Streetssolitons, SRYM2} about pluriclosed flow on torus bundles over Riemann surfaces.

\begin{proposition}
    Let $T\hookrightarrow P \rightarrow (M,J)$ be a torus bundle with complex fibre isomorphic to $(T^{2k},J_T)$, where $J_T$ is any invariant complex structure. Moreover, let $(\omega_t,A_t)$ be a one-parameter family of Hermitian metrics on $(M,J)$ and principal connections on $P$ such that 
    \begin{align}\label{eq:incond}
dd^c\omega_0+\langle F_{A_0}\wedge F_{A_0}\rangle_{\mathfrak{k}}=0, \; \;  \; F_{A_0}^{0,2}=0,
    \end{align}
    and satisfying:   \begin{align}\label{eq:PCFcorr}
        \begin{split}
            \partial_t \omega_t & = -(\rho_B(\omega) - \langle F_A, \Lambda_\omega F_A \rangle_{\mathfrak{k}})^{1,1},\\
\partial_t  A_t & = \frac{1}{2}Jd_A(\Lambda_\omega F_A)-\frac{1}{2}d_A(J_T\Lambda_\omega F_A).
        \end{split}
    \end{align}
    Then, the one-parameter family $\overline{\omega}_t$ given by \eqref{otot} remains pluriclosed along the flow and satisfies
    $$
    \partial_t \overline{\omega}_t = -\overline{\rho}_B^{1,1}
    $$
\end{proposition}
\begin{proof}
    Arguing as in the proof of Proposition \ref{p:PCFhol}, the Bianchi identity is preserved along the flow by Remark \ref{rmk:unitaryhol}.
     %let $\phi_t$ be the family of gauge transformations integrating $\tfrac{1}{2}J_T\Lambda_\omega F_A$. Then, $A'_t = \phi_t \cdot A_t$ satisfies 
     %$$
     %\partial_t A'_t = \frac{1}{2}Jd_{A'}(\Lambda_\omega F_{A'}).
     %$$
     %Moreover, we also have
     %$$
     %\partial_t \omega_t = -(\rho_B(\omega) - \langle F_A, \Lambda_\omega F_A \rangle_{\mathfrak{k}})^{1,1} = -(\rho_B(\omega) - \langle F_{A'}, \Lambda_\omega F_{A'} \rangle_{\mathfrak{k}})^{1,1},
     %$$
     %hence $(\omega_t, A'_t)$ satisfy string algebroid pluriclosed flow \eqref{eq:CPFunitary}. Arguing as in the proof of Proposition \ref{p:PCFhol}, it then follows that
     %$$
     %dd^c\omega_t + \langle F_{A'_t}\wedge F_{A'_t}\rangle_{\mathfrak{k}} = dd^c\omega_t + \langle F_{A_t}\wedge F_{A_t}\rangle_{\mathfrak{k}} = 0.
     %$$
     Note also that $\overline{J}_t$ is constant by  Lemma \ref{l:Jtotcte}. Then, by \eqref{eq:pctoBianchi}, the metric $\overline{\omega}_t$ is pluriclosed. Now, given \eqref{otot}, we compute for $\overline{\omega_t}$:
    \begin{equation*}
        \begin{split}
            \partial_t \overline{\omega_t} = & \ p^*(\partial_t \omega_t)+\tfrac{1}{2}\langle \Jtot \partial_t A_t \wedge A\rangle_{\mathfrak{k}} +\tfrac{1}{2}\langle \Jtot A \wedge \partial_t A_t\rangle_{\mathfrak{k}}\\
            = & \ p^*(-\rho_B-\langle \Lambda_\omega F_A , F_A\rangle_{\mathfrak{k}})^{1,1} + \tfrac{1}{2}\langle \left(-\tfrac{1}{2}d_A(\Lambda_\omega F_A)-\tfrac{1}{2}Jd_A(J_T \Lambda_\omega F_A)\right)\wedge A\rangle_{\mathfrak{k}}\\
            & +\tfrac{1}{2}\langle \Jtot A \wedge \left(\tfrac{1}{2}Jd_A(\Lambda_\omega F_A)-\tfrac{1}{2}d(J_T\Lambda_\omega F_A)\right)\rangle_{\mathfrak{k}}\\
            = & \ p^*(-\rho_B-\langle \Lambda_\omega F_A , F_A\rangle_{\mathfrak{k}})^{1,1} - \tfrac{1}{2}\langle d_A(\Lambda_\omega F_A)\wedge A\rangle_{\mathfrak{k}} - \tfrac{1}{2}\langle Jd_A(\Lambda_\omega F_A)\wedge \Jtot A\rangle_{\mathfrak{k}}\\
            = & \ p^*(-\rho_B-\langle \Lambda_\omega F_A , F_A\rangle)^{1,1} - \langle d_A(\Lambda_\omega F_A)\wedge A\rangle^{1,1} \\
            = & \ -\overline{\rho}_B^{1,1},
        \end{split}
    \end{equation*}
    where in the last step we have used \eqref{eq:rhotot}, and the fact the structure group is abelian.
\end{proof}

\begin{remark}
Notice that we have only explicitly recovered the pluriclosed flow equation for the Hermitian metric.  For any pluriclosed flow on the total space $P$ there is an accompanying flow of $(2,0)$-forms $\bar{\gb}_t$, which can also be described using the dimension reduced data, although we do not do this explicitly here.
\end{remark}

One can try to extend this picture to the case of non-abelian structure group by formally considering the same evolution equation as in \eqref{eq:evolAred}. Note however that in this case the second term breaks the equivariance of the connection and the resulting flowing connection is no longer principal. Therefore, it does not fit with the dimensional reduction ansatz of this section. Nevertheless, in the next result, we argue that there is in any case a pluriclosed flow line induced by a solution of string algebroid pluriclosed flow \eqref{eq:CPFunitary}.

\begin{proposition}\label{prop:redPCF}
    Let $K$ be a real Lie group of even dimension with left-invariant complex structure $J_K$, and $p \colon P\rightarrow (M,J)$ be a smooth principal $K$-bundle, as above. Let $(\omega_t,b_t,A_t)$ satisfy string algebroid pluriclosed flow  \eqref{eq:CPFunitary}. Then, the solution to generalized Ricci flow on $P$ with initial condition $(\overline{g}_0,\overline{b}_0)$ is given by $\left(\overline{\phi}_t^*\overline{g}_t, \overline{\phi}_t^*\overline{b}_t-p^*B_t-\langle \overline{\phi}_t^*A_0 \wedge A_0\rangle_{\mathfrak{k}}\right)$, where 
      \begin{equation*}
        \begin{split}
        \frac{d}{dt} B_t = & \ \phi_t^*\left(-d\left(\theta+i_{\theta^\sharp}b-\langle a, i_{\theta^\sharp}a\rangle_{\mathfrak{k}} \right)+i_{\theta^\sharp} H_0 +2\langle i_{\theta^\sharp} a, F_{A_0}\rangle_{\mathfrak{k}}\right)
        \end{split}
        \end{equation*}
        and $\overline{\phi}_t, \phi_t$ are the families of diffeomorphisms integrating $A_t^\perp\theta_t^\sharp$ and $\theta_t^\sharp$ respectively. In consequence, they are gauge-equivalent to a one-parameter family solving pluriclosed flow on $(P,\Jtot)$, where
        $$
        \Jtot = A_0^\perp J dp + J_K \circ A_0.
        $$
\end{proposition}
\begin{proof}
Assume that $(\omega_t,b_t,A_t)$ satisfy \eqref{eq:CPFunitary}. In particular, we have that the Bianchi identity \eqref{eq:BIGHM} is satisfied
along the flow by a straightforward calculation (cf. the proof of Proposition \ref{p:PCFhol}). Now, by Proposition \ref{prop:UPCFRiemannian}, the family $(g_t,b_t,A_t)$ satisfies \eqref{eq:UPCFGRF}. Following the gauge fixing of Proposition \ref{p:GFGRFinGG+}, the triple $(g'_t,b'_t,A'_t)$ given by (cf. \eqref{eq:gauge_pars}):
\begin{equation}
    g_t = \phi{_t}_*g'_t, \qquad b_t = \phi{_t}_*(b'_t + B_t - \langle a'^{\overline{\phi}}\wedge a'\rangle_{\mathfrak{k}} ), \qquad A_t = \overline{\phi}{_t}_*  A'_t
\end{equation}
satisfy generalized Ricci flow \eqref{eq:GRFgbA}, where $B_t$ is obtained from \eqref{f:Bstrgauge}. Then, using \eqref{f:gtot}, we compute
\begin{align*}
    \overline{g}'_t
    = &\ p^*g'_t - \langle A', A'\rangle_{\mathfrak{k}}
    = p^*\phi_t^*g_t - \overline{\phi}_t^*\langle A, A\rangle_{\mathfrak{k}} = \overline{\phi}_t^*\overline{g}_t
    \end{align*}
and similarly using \eqref{f:btot}:
\begin{align*}
    \overline{b}'_t = &\ \phi_t^*b_t - p^*B_t + \langle a{'}_t^{\overline{\phi}}\wedge a'_t \rangle_{\mathfrak{k}} - \langle a'_t \wedge A'_t \rangle_{\mathfrak{k}}\\
    = &\ \overline{\phi}_t^*\overline{b}_t -p^*B_t - \langle \overline{\phi}_t^*A_0 \wedge A_0\rangle_{\mathfrak{k}}.
\end{align*}
Then, by Proposition \ref{p:cGRFtoGRF}, the pair
$(\overline{g}'_t,\overline{b}'_t)$ satisfy generalized Ricci flow on $P$. For the final statement, first observe that at time $t=0$, the metric $\overline{g}_0$ is Hermitian for $\Jtot$ as defined in the statement.  Furthermore, using \eqref{eq:pctoBianchi} and the fact that the Bianchi identity \eqref{eq:BIGHM} holds, the associated Hermitian form $\overline{\omega}_0$ is pluriclosed.  Thus by \cite[Theorem 6.5]{PCFReg}, the data $(\overline{g}'_t,\overline{b}'_t)$ is gauge-equivalent to a solution to pluriclosed flow on $(P, \overline{J})$.
\end{proof}

\section{Uniformly parabolic pluriclosed flows on string algebroids} \label{s:higherreg}

In this section we prove Theorem \ref{t:EKthmintro} (cf. Theorem \ref{t:EKthm} below).  After introducing the relevant notation, in the next subsection we establish the key a priori $C^{\ga}$ estimate for the reduction $h$ by an adaptation of the method of Evans-Krylov.  Next, using this estimate we show a Liouville theorem for entire ancient solutions.  Then, by a standard blowup argument we use this to establish Theorem \ref{t:EKthmintro}.  We end by using Theorem \ref{t:EKthmintro} to establish a long-time existence criteria for the flow in the presence of uniform bounds on $g$.

\subsection{Statement of Theorem \ref{t:EKthm}}

We start by introducing some relevant notation for the statement of the result.

\begin{defn} \label{d:Phiquantity} Fix a complex manifold $(M,J)$ and a holomorphic vector bundle $V \to (M,J)$.  Given Hermitian metrics $h,\til{h}$ on $V$, let
\begin{align*}
    \gU(h,\til{h}) = \N_h - \N_{\til{h}},
\end{align*}
denote the difference of the Chern connections associated to $h$ and $\til{h}$. Given now Hermitian metrics $g,\til{g}$ on $(M,J)$ define
\begin{align*}
    \Phi_k(\gw,\til{\gw},h,\til{h}) := \sum_{j=0}^k \left( \brs{\N^j_g \gU(\gw,\til{\gw})}^2_{g} + \brs{\N^j_{g,h} \gU(h,\til{h})}_{g,h}^2 \right)^{\tfrac{1}{1 + j}}.
\end{align*}
The quantity $\Phi_k$ is a scale-invariant $C^k$ norm for the data $(\gw, h)$ with respect to $(\til{\omega},\til{h})$.
\end{defn}

Note that we have used the notation $\gU(\gw,\til{\gw})$ also to measure the difference of Chern connections of the Hermitian metrics $g$, $\til{g}$ on the holomorphic tangent bundle. Let $(P^c,\bar J)$ be a holomorphic principal $K^c$-bundle over $(M,J)$. We fix a faithful representation for $K^c$ and regard reductions $h \in \Gamma(P^c/K)$ as Hermitian metrics in the associated holomorphic vector bundle. Then, given reductions $h,\til{h} \in \Gamma(P^c/K)$ and Hermitian metrics $g$, $\til{g}$, we have associated quantities $ \gU(h,\til{h})$ and $\Phi_k(\gw,\til{\gw},h,\til{h})$, as before.

\begin{theorem} \label{t:EKthm} (cf. Theorem \ref{t:EKthmintro}) Let $(\gw_t, \gb_t, h_t)$ a solution to pluriclosed flow \eqref{eq:CPFhol} defined on $[0, \tau), \tau \leq 1$, such that the initial condition satisfies the Bianchi identity 
$$
d d^c \gw_0 + \IP{F_{h_0} \wedge F_{h_0}}_{\mathfrak k} = 0.  
$$  
Suppose there are background metrics $\til{\gw}$, $\til{h}$ and $\gl,\gL > 0$ so that for all $(x,t)$,
\begin{align*}
\gl^{-1} \til{\gw} \leq \gw \leq \gL \til{\gw}, \qquad \brs{\gb}_{\til{\gw}} \leq \gL, \qquad \gl^{-1} \til{h} \leq h \leq \gL \til{h}.
\end{align*}
Given $k \in N$ there exists a constant $C = C(n, k, \gl, \gL,\til{\gw},\til{h})$ such that for all $t \leq 1$ one has
\begin{align*}
\sup_{M \times [0,\tau)} t \Phi_k (\gw,\til{\gw},h,\til{h}) \leq C.
\end{align*}
\end{theorem}

\begin{remark} While we do not give a precise statement here, it is straightforward to localize these estimates following the arguments of \cite{JordanStreets}.
\end{remark}

\subsection{Evolution equations}

Given $(M, g, J)$ a Hermitian manifold, we let $\gD^C_{\gw}$ denote the Laplacian of the Chern connection of $\gw$.  Furthermore we let
\begin{align*}
\square = \left(\dt - \gD^C_{\gw_t} \right),
\end{align*}
denote the heat operator associated to the Chern connections of a family of Hermitian metrics.  First we recall the classic Weitzenb\"ock formula:

\begin{lemma} \label{l:Schwarzlemma} Given $(P^c,\bar J)$ a holomorphic principal $K^c$-bundle over a Hermitian manifold $(M,g,J)$ with reductions $h, \til{h} \in \Gamma(P^c/K)$, one has
\begin{align*}
\gD^C_{\gw} \tr_h \til{h} =&\ \brs{\gU(h,\til{h})}^2_{g,h^{-1},\til{h}} + \tr \til{h} h^{-1} \left( S_{g}^h - S_{g}^{\til{h}} \right).
\end{align*}
\end{lemma}

\noindent Direct computations using Lemma \ref{l:Schwarzlemma} to further evolution equations:

\begin{lemma} \label{l:hSchwarz} Fix $(\gw_t, \gb_t, h_t)$ a solution to pluriclosed flow \eqref{eq:CPFhol}, and $\til{h}$ a background reduction on $(P^c,\bar J)$.  Then
\begin{align*}
\square \tr_{\til{h}} {h} =&\ - \brs{\gU(h,\til{h})}^2_{g,\til{h}^{-1}, h} - \tr h \til{h}^{-1} S_{g}^{\til{h}},\\
\square \tr_{h} {\til{h}} =&\ - \brs{\gU(h,\til{h})}^2_{g,h^{-1}, \til{h}} + \tr \til{h} h^{-1} S_{g}^{\til{h}},\\
\square \log \det \left( h \til{h}^{-1} \right) =&\ - \tr S_{g}^{\til{h}}.
\end{align*}
\end{lemma}

Next we want to derive a key evolution equation for the square norm of the Chern connection of $h$ along string algebroid pluriclosed flow \eqref{eq:CPFhol}.  To derive this we first understand the evolution of the base metric in terms of the Chern curvature as opposed to the Bismut curvature.  While not necessary for the proof of Theorem \ref{t:EKthm}, a key point in the next section is that, in the case $\IP{,}_{\mathfrak{k}}$ is definite, the flow is a supersolution of the usual pluriclosed flow.

\begin{lemma} \label{prop:UPCFexp2} Let $(\omega_t,h_t)$ be a solution of the pluriclosed flow \eqref{eq:CPFhol}.  Then one has
\begin{gather} \label{f:redGRF}
\begin{split}
    \dt \gw =&\ - \Lambda_{\gw} F_g + (T^2 + F_h^2)(J \cdot, \cdot),
\end{split}
\end{gather}
where $T$ is the Chern torsion and $T^2(X,Y) = \IP{i_X T, i_Y T}$.
\end{lemma}
\begin{proof}
    By Definition \ref{def:PCFQ}, we only need to prove that
    $$
    -\rho_B^{1,1}+i\langle S_g^h,F_h\rangle_{\mathfrak{k}} = - \Lambda_{\gw} F_g + (T^2+F_h^2)(J \cdot,\cdot),
    $$
    where $T$ is the Chern torsion, given by $
    g(T(X,Y),Z) = -d\omega(JX,Y,Z),
    $ as the terms $T^2$ and $F_h^2$ are type $(1,1)$. With our notations, 
 by \cite[Proposition 3.3]{Ivanovstring}, we have
    $$
    - \Lambda_{\gw} F_g = -\rho_B^{1,1}(X,Y) + T^2(X,JY)+\textstyle \sum_i \tfrac{1}{4}dd^c\omega(X,Y,e_i,Je_i),
    $$
    where $T^2$ is computed using the tensor norm of $g$.  Using now the Bianchi identity
    $$
    dd^c\omega+\langle F_h\wedge F_h\rangle_{\mathfrak{k}}= 0
    $$
    it follows that
    \begin{equation*}
        \begin{split}
            \textstyle \sum_i \tfrac{1}{4}dd^c\omega(X,Y,e_i,Je_i) = & -\tfrac{1}{4}\textstyle \sum_i \langle F_h\wedge F_h\rangle_{\mathfrak{k}}(e_i,Je_i,X,Y)\\
            = & -\textstyle \sum_i \tfrac{1}{2}\langle i_{e_i}F_h\wedge F_h\rangle_{\mathfrak{k}}(Je_i,X,Y)\\
            = & \ i\langle S_g^h,F_h\rangle_{\mathfrak{k}}(X,Y) +F_h^2(X,JY)
        \end{split}
    \end{equation*}
    which combined with the previous identity yields the result.
\end{proof}

\begin{lemma} \label{l:hconnectionflow} Let $(\omega_t,h_t)$ be a solution of the pluriclosed flow \eqref{eq:CPFhol}.  Then one has
\begin{align*}
\square \brs{\gU(h,\til{h})}^2_{g,h^{-1},h} =&\ - \brs{\N \gU}^2_{g,h^{-1},h} - \brs{\bar{\N} \gU + \bar{T} \cdot \gU}^2_{g,h^{-1},h} - \IP{ F_h^2,  \gU^2}_{g}\\
&\ + T \star \gU \star \til{F} + \gU^2 \star \til{F} + \gU \star \til{\N} \til{F}.
\end{align*}
where
\begin{align*}
(\bar{T} \cdot \gU)_{\bi j \ga}^{\gb} =&\ g^{\bk l} T_{\bar{i} \bk j} \gU_{l \ga}^{\gb} , \qquad \gU^2_{i \bj} = \gU_{i \ga}^{\gb} \gU_{\bj \bgg}^{\bar{\gd}} h^{\bgg \ga} h_{\gb \bar{\gd}}.
\end{align*}
and $\star$ denotes algebraic contraction of tensors.

\begin{proof} To begin we observe
\begin{align*}
\dt \gU_{i \ga}^{\gb} =&\ - \N_i (S_g^h)_{\ga}^{\gb}.
\end{align*}
Also we have
\begin{align*}
\gD \gU_{i \ga}^{\gb} =&\ g^{\bl k} \N_k \N_{\bl} \gU(h,\til{h})_{i \ga}^{\gb}\\
=&\ g^{\bl k} \N_k \left( F_{\bl i \ga}^{\gb} - \til{F}_{\bl i \ga}^{\gb} \right)\\
=&\ - \N_i (S_g^h)_{\ga}^{\gb} - g^{\bl k} T_{i k}^p F_{p \bl \ga}^{\gb} - g^{\bl k} \N_k \til{F}_{\bl i \ga}^{\gb}.
\end{align*}
Also we have the general commutation formula
\begin{align*}
\bar{\gD} \bar{\gU}_{\bj \bga}^{\bgb} =&\ \gD \bar{\gU}_{\bj \bga}^{\bgb} - (S_g^g)_{\bj}^{\bp} \bar{\gU}_{\bp \bga}^{\bgb} - (S_{g}^h)_{\bga}^{\bgg} \bar{\gU}_{\bj \bgg}^{\bgb} + (S_{g}^h)_{\bgg}^{\bgb} \bar{\gU}_{\bj \bga}^{\bgg}.
\end{align*}
With these in place and using Lemma \ref{prop:UPCFexp2}, the computation is a direct adaptation of \cite[Proposition 5.7]{JFS}.
\end{proof}
\end{lemma}

\subsection{$C^{\ga}$-estimate for $h$}

In this subsection we prove a key a priori regularity property for $h$ under the assumption of uniform ellipticity.  We note that Donaldson's original proof of this regularity for Hermitian Yang-Mills flow crucially exploits the semilinear nature of the equation, and in particular the fact that the linearized operator is fixed in time, and comparing the metrics of a single flow line with a small time-shift (cf. \cite{DonaldsonHYM} Corollary 15).  This is no longer possible in our setting as the linearized operator will be the Chern Laplacian of $g$, which is also varying in time.  For this reason we must resort to a more delicate argument, adapting the method of Evans-Krylov to this setting.  The proof is similar in spirit to other results in the theory of pluriclosed flow already in \cite[Theorem 4.3]{StreetsWarren}, \cite[Proposition 5.16]{SBIPCF}.

\begin{defn} Given $(w,s) \in \mathbb C^n \times \mathbb R$, let
\begin{align*}
 Q(R) := &\ \{ (z,t) \in \mathbb C^n \times \mathbb R |\ t \leq s, \quad
\max \{ \brs{z - w}, \brs{t - s}^{\frac{1}{2}} \} < R \}\\
 \Theta(R) :=&\ Q((w,s - 4 R^2),R).
\end{align*}
\end{defn}

Since we are concerned with a local situation, we regard $h$ as a Hermitian-symmetric valued function defined on an open of $\CC^n$.

\begin{prop} \label{p:hosc} Suppose $(\gw_t, \gb_t, h_t)$ is a solution to (\ref{eq:CPFhol}) defined on $Q(R)$ such that
\begin{align*}
\gl \til{\gw} \leq \gw \leq \gL \til{\gw}, \qquad \gl \til{h} \leq h \leq \gL \til{h},
\end{align*}
for background metrics $\til{\gw}, \til{h}$.
Then there exist positive constants $\ga, C$
depending only on $n,\gl,\gL$ such that for all $\rho < R$,
\begin{align*}
\osc_{Q(\rho)} h \leq C(n,\gl,\gL) \left(
\frac{\rho}{R} \right)^{\ga} \osc_{Q(R)} h.
\end{align*}

\begin{proof} As the result is local, we can make the simplifying assumption without loss of generality that the background metrics $\til{\gw}$ and $\til{h}$ are constant with respect to chosen local coordinates and frames, and in particular are flat.  For notational simplicity we assume our solution is in fact defined on $Q(4R)$.  First note that the operator $\log \det$ is $(\gl,\gL)$-elliptic on a
convex set containing the range of $h$.  It follows that for any two points $(x,t_1),(y,t_2) \in Q(4R)$ there exists a matrix
$a^{ij}$, $\gl \gd_i^j \leq a^{i{\bj}} \leq \gL \gd_i^j$ such that
\begin{gather} \label{f:hEK20}
\begin{split}
\log \det h(x,t_1) - \log \det h(y,t_2) =&\ a^{i{\bj}}((x,t_1),(y,t_2)) \left( h_{i \bj}(x,t_1) - h_{i
\bj}(y,t_2) \right).
\end{split}
\end{gather}
It follows from (\cite{Lieberman} Lemma 14.5) that we can choose unit vectors
$v_{k}$ and functions $f_{k} = f_{i}((x,t_1),(y,t_2))$, such
that
\begin{align*}
a^{i\bj} = \sum_{k = 1}^N f_{k} v_{k}^i
\bar{v_{k}^j},
\end{align*}
and moreover $\gl_* \leq f_{k} \leq \gL_*$, where $\gl_*, \gL_*$ only depend
on $\gl,\gL$.  Now let $f_0 := 1$, $h_0 := - \log \det h$, $h_{k} := h_{v_{k} \bar{v_{k}}}$.  Thus in particular (\ref{f:hEK20}) is now expressed as
\begin{align} \label{wrel}
 \sum_{k=0}^N f_{k} \left(h_{k} (y,t_2) - h_{k}(x,t_1) \right) =&\ 0.
\end{align}

Now let
\begin{align*}
M_{s k} = \sup_{Q(sR)} h_{k}, \qquad m_{s k} = \inf_{Q(sR)} h_{k}, \qquad P(sR) =&\ \sum_{k=0}^N M_{s k} - m_{s k}.
\end{align*}
By direct calculations similar to Lemma \ref{l:hSchwarz} in the case $\til{h}$ is flat, one derives that
\begin{align*}
  \square \log \det h =&\ 0,\quad \square h \leq 0,
\end{align*}
where the last inequality holds in the sense of bilinear forms.  In particular, it follows that for every $0 \leq k \leq N$ the function $M_{2 k} - h_{k}$ is a
supersolution to a uniformly parabolic equation.  Thus by the weak Harnack estimate \cite[Theorem 7.37]{Lieberman} we conclude
\begin{align} \label{f:hEK30}
\left( R^{-n-2} \int_{\Theta(R)} (M_{2k} - h_{k})^p \right)^{\frac{1}{p}}
\leq C \left( M_{2k} - M_{k} \right).
\end{align}
Now observe that we can rearrange (\ref{wrel}) to yield for every $(y,t_2) \in
Q(2R)$,
\begin{align*}
 h_{k}(y,t_2) - m_{2k} \leq C \sum_{l \neq k} M_{2 l} -
h_{l}(y,t_2).
\end{align*}
Integrating this over $\Theta(R)$, applying Minkowski's inequality and
(\ref{f:hEK30})
yields
\begin{gather} \label{f:hEK40}
 \begin{split}
 \left( R^{-n-2} \int_{\Theta(R)} \left( h_{k} - m_{2 k} \right)^p
\right)^{\frac{1}{p}} \leq&\ \left( C R^{-n-2} \int_{\Theta(R)} \left( \sum_{l \neq k} M_{2 l} - h_{l} \right)^p \right)^{\frac{1}{p}}\\
\leq&\ C \sum_{l \neq k} \left( R^{-n-2} \int_{\Theta(R)} \left( M_{2 l} -
h_{l} \right)^p \right)^{\frac{1}{p}}\\
\leq&\ C \sum_{l \neq k} M_{2 l} - M_{l}.
 \end{split}
\end{gather}
Now we use (\ref{f:hEK30}) and (\ref{f:hEK40}) and Minkowski's inequality to yield
\begin{align*}
 M_{2l} - m_{2 l} =&\ \left( R^{-n-2} \int_{\Theta(R)} \left(M_{2 l} -
m_{2 l} \right)^p \right)^{\frac{1}{p}}\\
 =&\ \left( R^{-n-2} \int_{\Theta(R)} \left( M_{2l} - h_{l} + (h_{l} -
m_{2 l}) \right)^p \right)^{\frac{1}{p}}\\
\leq&\ C \sum_{k} M_{2 k} - M_{k}\\
\leq&\ C \sum_{k} M_{2 k} - M_{k} + m_{k} - m_{2 k}\\
=&\ C \left( P(2R) - P(R) \right).
\end{align*}
Summing over $l$ and rearranging yields for some constant $0 < \mu < 1$ the
inequality
\begin{align*}
 P(R) \leq \mu P(2R).
\end{align*}
A standard iteration argument now yields the statement of the theorem.
 \end{proof}
\end{prop}

\subsection{Flatness of ancient solutions}

We next present our Liouville type theorem for entire ancient solutions.

\begin{prop} \label{p:ancientrigidity} Suppose $(\gw_t, \gb_t, h_t)$ is a solution to (\ref{eq:CPFhol}) defined on $\mathbb C^n \times (-\infty, 0]$ which further satisfies
\begin{align*}
\gl \gw_F \leq \gw \leq \gL \gw_F, \qquad \brs{\gb}_{\gw_F} \leq \gL, \qquad \gl h_F \leq h \leq \gL h_F,
\end{align*}
where $\gw_F$ and $h_F$ are flat metrics.  Then $(\gw_t, \gb_t, h_t) \equiv (\gw_0, \gb_0, h_0)$, where $\gw_0$ and $h_0$ are flat metrics.
\begin{proof} We first claim that $h_t \equiv h_0$, where $h_0$ is flat.  Suppose there exists $\ga > 0$ and a point $(x', t') \in \mathbb C^n \times (-\infty]$ such that $\brs{h}_{C^{\ga}}(x',t') \neq 0$.  For large $L > 0$ define
\begin{align*}
{\gw}^L(x,t) =&\ \gw(x' + L^{-1} x, t' + L^{-2} t)\\
\gb^L(x,t) =&\ \gb(x' + L^{-1} x, t' + L^{-2} t)\\
{h}^L(x,t) =&\ h(x' + L^{-1} x, t' + L^{-2} t).
\end{align*}
Note that $(\gw^L, \gb^L, h^L)$ is a solution to (\ref{eq:CPFhol}), relative to the rescaled inner product $L^{2} \IP{,}_{\mathfrak k}$.  Furthermore, by construction, $\brs{h^L}_{C^{\ga}}(0,0) = L^{\ga}$.  On the other hand, the oscillation of $h^L$ on $Q(1)$ is bounded independently of $L$.  It thus follows from Proposition \ref{p:hosc} that there exists $\ga = \ga(n,\gl,\gL)$ such that $\brs{h^L}_{C^{\ga}}(0,0)$ has a uniform upper bound, which yields a contradiction for sufficiently large $L$.  Thus $\brs{h}_{C^{\ga}} \equiv 0$ and it follows that $h_t \equiv h_0$ where $h_0$ is a flat metric.  Now that $F_{h} \equiv 0$, note that the anomaly cancellation equation implies that the metrics $\gw_t$ are pluriclosed, and moreover $(\gw_t, \gb_t)$ is a solution to pluriclosed flow.  By the Evans-Krylov type estimate (cf. \cite{SBIPCF} Corollary 5.20) or the Calabi-type estimate (\cite{JordanStreets} Theorem 4.2) for pluriclosed flow, it follows that $(\gw_t, \gb_t) \equiv (\gw_0, \gb_0)$, where $\gw_0$ is flat.
\end{proof}
\end{prop}

\subsection{Proof of Theorem \ref{t:EKthm}}

We complete next the proof of the main result of this section:

\begin{proof}[Proof of Theorem \ref{t:EKthm}] For notational convenience set $\Phi_k := \Phi_k(\gw,\til{\gw},h,\til{h})$.  If the statement was false for some $k$, we can choose a sequence of times $(x_i, t_i)$ such that
\begin{align} \label{f:EKthm5}
    \limsup_{i \to \infty} t_i \Phi_k (x_i,t_i) = \infty.
\end{align}
By a standard point-picking argument we can replace $(x_i,t_i)$ with a new sequence (still denoted $(x_i,t_i)$) such that
\begin{align} \label{f:EKthm10}
    t_i \Phi_k(x_i,t_i) =&\ \sup_{M \times [0,t_i]} t \Phi_k(x,t).
\end{align}
Since $M$ is compact we may further assume that $\lim x_i = x$.  Pick local normal coordinates for $\til{\gw}$ on $B_{\gd}(x,\til{\gw})$, 
$$
\gd < \min\{ \tfrac{1}{2} \inj(x,\til{\gw}), \brs{\til{\Rm}}_{C^0}^{-\tfrac{1}{2}} \}.
$$
In these coordinates, the given solution $(\gw_t,\gb_t,  h_t)$ is defined on $B_{\gd}(0) \times [0,\tau)$.
We let $L_i := \Phi_k(x_i,t_i)$, and define
\begin{align*}
    \gw^i(x,t) =&\ \gw(x_i + L_i^{-1} x, t_i + L_i^{-2} t),\\
    \gb^i(x,t) =&\ \gb(x_i + L_i^{-1} x, t_i + L_i^{-2} t),\\
    h^i(x,t) =&\ h(x_i + L_i^{-1} x, t_i + L_i^{-2} t).
\end{align*}
Furthermore, let $\Phi_k^i := \Phi_k(\gw^i,\til{\gw}^i,h^i,\til{h}^i)$, where $(\til{\gw}_i, \til{h}_i)$ are scaled as above.  The data $(\gw^i, \gb^i, h^i)$ is a solution to (\ref{eq:CPFhol}) relative to the rescaled inner product $L_i^2 \IP{,}_{\mathfrak k}$, defined on $$
(-t_i L_i^2, 0] \times B_{\tfrac{1}{2} \gd L_i}(0).  
$$
Using (\ref{f:EKthm5}) and (\ref{f:EKthm10}) it follows for any compact subset of $(-\infty,0] \times \mathbb C^n$, the solution $(\gw^i,\gb^i, h^i)$ is defined on this set for arbitrarily large $i$ and moreover satisfies that $\Phi_k^i \leq 1$ on this domain, while $\Phi_k^i(0,0) = 1$.  By Schauder estimates, it follows that on all compact subsets of $(-\infty, 0] \times \mathbb C^n$, for all $l > k$ one has that $\Phi^i_l$ is uniformly bounded.  It follows that the sequence $(\gw^i,\gb^i, h^i)$ converges to a solution satisfying the hypotheses of Proposition \ref{p:ancientrigidity}, and also satisfies $\Phi_k^{\infty}(0,0) = 1$, which is a contradiction.
\end{proof}

\subsection{Long-time existence criteria}

\begin{prop} \label{p:hboundglower} Fix $(\gw_t, \gb_t, h_t)$ a solution to pluriclosed flow \eqref{eq:CPFhol} on $[0, T]$, and background metrics $\til{\gw}, \til{h}$.  Suppose there exists $\mu > 0$ such that for all $(x,t) \in M \times [0,T]$,
\begin{align*}
    \mu \til{\gw} \leq \gw_t, \qquad \brs{\gb}_{\til{\gw}} \leq \gL.
\end{align*}
Then there exists $\gL > 0$ depending on $\mu, h_0, \til{h}$ and $T$ such that
\begin{align*}
    \gL^{-1} \til{h} \leq h \leq \gL \til{h}.
\end{align*}
\begin{proof} Applying the assumed lower bound for $\gw$ we obtain the elementary estimate
\begin{align*}
    \tr {h}^{-1} \til{h} S_{g,\til{h}} \leq&\ C \mu^{-1} \tr h^{-1} \til{h}
\end{align*}
for some $C$ depending on $\til{h}$.  Thus from Lemma \ref{l:hSchwarz} we obtain
\begin{align*}
\square \tr_{h} {\til{h}} \leq&\ C \mu^{-1} \tr_h \til{h}.
\end{align*}
From the maximum principle we conclude
\begin{align*}
    \sup_{M \times \{t\}} \tr_h \til{h} \leq&\ e^{C \mu^{-1} t} \sup_{M \times \{0\}} \tr_h \til{h}.
\end{align*}
Similarly, the lower bound on $\gw$ implies that $\tr S_{g,\til{h}} \geq - C \mu^{-1}$ for a constant depending on $\til{h}$.  It then follows from Lemma \ref{l:hSchwarz} that
\begin{align*}
\square \log \det \left( h \til{h}^{-1} \right) \geq&\ - C \mu^{-1}.
\end{align*}
From the maximum principle we conclude
\begin{align*}
    \inf_{M \times \{t\}} \log \det (h \til{h}^{-1}) \geq&\ \inf_{M \times \{0\}} \log \det (h \til{h}^{-1}) - C \mu^{-1} t.
\end{align*}
Thus we have shown a uniform upper bound for $\tr_h \til{h}$ and a uniform lower bound for $\det h \til{h}^{-1}$ on $[0,T]$, and the existence of the claimed constant $\gL$ follows.
\end{proof}
\end{prop}

\begin{corollary} \label{c:flowregularitycor} Fix $(\gw_t, \beta_t, h_t)$ a solution to pluriclosed flow \eqref{eq:CPFhol} on $[0, T)$, and a background metric $\til{\gw}$.  Suppose there exists $\gL > 0$ such that for all $(x,t) \in M \times [0,T)$,
\begin{align*}
    \gL^{-1} \til{\gw} \leq \omega \leq \gL \til{\gw}, \qquad \brs{\gb}_{\til{\gw}} \leq \gL.
\end{align*}
Then the flow extends smoothly past time $T$.
\begin{proof} By Proposition \ref{p:hboundglower}, there exist uniform ellipticity bounds for $h$ with respect to some background metric $\til{h}$ on $[0, T)$.  By Theorem \ref{t:EKthm} there exists a uniform bound on $\Phi_k(\gw, \til{\gw}, h, \til{h})$ on $[T/2, T)$.  It is a standard argument to then integrate the flow equation in time to produce smooth limiting data $(\gw_T, \gb_T, h_T)$.  There exists a short-time solution with initial data $(\gw_T, \gb_T, h_T)$, thus extending the flow past time $T$ as claimed.
\end{proof}
\end{corollary}

\section{Global existence results}\label{sec:Global}

In this section we turn to analyzing the existence problem for string algebroid pluriclosed flow \eqref{eq:CPFhol}. We first prove that it is a well-posed evolution equation, for arbitrary choice of pairing $\IP{,}_\mathfrak{k}$. Assuming negative-definite pairing, using the flow of Hermitian metrics \eqref{eq:DFlow} on the string algebroid, we prove the higher regularity result of Theorem \ref{t:GGEKintro} for solutions with uniform bounds on $\mathbf{G}$.  Further building on this we establish Theorem \ref{t:flatspaceconvintro}, yielding global existence and convergence of the flow for arbitrary initial data on some special backgrounds.

\subsection{Short-time existence and higher regularity}

We first prove that the string algebroid pluriclosed flow \eqref{eq:CPFhol} is a well-posed evolution equation. We should stress that the next result applies for arbitrary signature of the pairing $\IP{,}_\mathfrak{k}$, as it follows from the proof.

\begin{prop} \label{p:STE} 
Fix $(P^c,\bar J)$ a holomorphic principal $K^c$-bundle over a compact complex manifold $(M,J)$. Given data $(\gw_0, h_0)$ there exists $\ge > 0$ and a unique solution to string algebroid pluriclosed flow \eqref{eq:CPFhol} with initial data $(\gw_0,h_0)$ on $[0,\ge)$.
\begin{proof} 
We provide a brief sketch which is an straightforward adaptation of previous arguments.  Fixing a background metric $\til{\gw}$, it suffices to prove short time existence for the system of equations for pairs $(\xi_t,h_t)$ in Lemma \ref{lem:oneformred}.
%\begin{equation*}
%\begin{split}
%\dt \xi =&\ \delb^*_{\gw} \gw - \frac{\i}{2} \del \log \frac{\gw^n}{\til{\gw}^n},\\
%h^{-1}\dt h =&\ - S^h_g,\\
%\dt \hat{\omega} =&\ - \rho_C(\til{\gw}) + \i \IP{S^h_g, F_h}_{\mathfrak k}, \qquad \gw = \hat{\gw} + \delb \xi + \del \overline{\xi}.
%\end{split}
%\end{equation*}
The equation for $\xi$ is degenerate parabolic due to the gauge redundancy $\delb (\xi + \del f) + \del (\bar{\xi} + \delb f) = \delb \xi + \del \bar{\xi}$ for a real function $f$.  After addition of $\tfrac{1}{2} \del \operatorname{Re} \N^{\star}_{g_0} \xi$ in the evolution of $\xi$, this becomes a strictly parabolic system for $\xi$ and $h$ coupled to an ODE describing $\hat{\omega}$ (cf. the proof of \cite[Lemma 3.5]{SBIPCF}) and the short-time existence follows from standard methods. Note that this modification of the flow for $\xi$ does not change the flow for $\omega$, and hence Lemma \ref{lem:oneformred} still applies.
\end{proof}
\end{prop}

%
%
%\subsection{Higher Regularity}
%
%

%In this subsection we 
We turn next to the proof of Theorem \ref{t:GGEKintro} (cf. Theorem \ref{t:GGEK}) in the case of a negative-definite pairing $\IP{,}_\mathfrak{k} < 0$, a higher regularity result related to Theorem \ref{t:EKthmintro}.  Whereas Theorem \ref{t:EKthmintro} relied on a separate analysis of $h$ using ideas related to the Evans-Krylov theorem, the proof of Theorem \ref{t:GGEKintro} is comparatively simpler, relying on a direct maximum principle estimate for the Chern connection associated to the Hermitian metric $\mathbf{G}$, which is positive-definite by our assumption on $\IP{,}_\mathfrak{k}$.  It is thus a generalization of Yau's $C^3$ estimate for the complex Monge-Amp\`ere equation \cite{YauCC}, as well as the higher regularity for pluriclosed flow established in \cite{JFS,JordanStreets,SBIPCF}.  We emphasize that this argument will only hold for $\IP{,}_\mathfrak{k} < 0$.  

\begin{defn} \label{d:GGPhiquantity} Fix a holomorphic string algebroid $\mathcal{Q} \to (M,J)$ and a Hermitian metric $g$ on $M$. Given Hermitian metrics $\mathbf{G},\mathbf{\til{G}}$ on $\mathcal{Q}$, let
\begin{align*}
    \gU(\mathbf{G},\til{\mathbf{G}}) = \N_{\mathbf{G}} - \N_{\til{\mathbf{G}}},
\end{align*}
denote the difference of the Chern connections associated to $\mathbf{G}$ and $\til{\mathbf{G}}$.  Now define 
\begin{align*}
    \Phi_k(\mathbf{G},\til{\mathbf{G}}) := \sum_{j=0}^k \left( \brs{\N^j_{g,\mathbf{G}} \gU(\mathbf{G},\til{\mathbf{G}})}^2_{g,\mathbf{G}} \right)^{\tfrac{1}{1 + j}},
\end{align*}
where $\N$ denotes the relevant Chern connection.  The quantity $\Phi_k$ is a scale-invariant $C^{k+1}$ norm of $\mathbf{G}$.
\end{defn}

We fix now $(\gw_t,\beta_t,h_t)$ a solution to pluriclosed flow \eqref{eq:CPFhol} with initial condition $(\gw_0,0,h_0)$ satisfying the Bianchi identity \eqref{eq:BIGHM}. By Lemma \ref{p:PCFhol}, it induces a family of Hermitian metrics $\mathbf{G}_t = \mathbf{G}(\gw_t,\beta_t,h_t)$ on the holomorphic string algebroid $\mathcal{Q} = \mathcal{Q}_{2\i\partial \omega_0,A^{h_0}}$ satisfying the coupled HYM flow equation \eqref{eq:DFlow}.

\begin{lemma} \label{l:GGfloweqns} 

Fix $\mathbf{\til{G}}$ a background Hermitian metric on $\mathcal{Q}$. With the setup above, let $\gU := \gU(\mathbf{G}_t, \til{\mathbf{G}})$ denote the difference of Chern connections associated to $\mathbf{G}_t$ and $\til{\mathbf{G}}$.  Then
\begin{align*}
\square \tr_{\til{\mathbf{G}}} {\mathbf{G}} =&\ - \brs{\gU}^2_{g,\til{\mathbf{G}}^{-1}, \mathbf{G}} - \tr \mathbf{G} \til{\mathbf{G}}^{-1} S_g^{\til{\mathbf{G}}},\\
\square \tr_{\mathbf{G}} {\til{\mathbf{G}}} =&\ - \brs{\gU}^2_{g,\mathbf{G}^{-1}, \til{\mathbf{G}}} + \tr \til{\mathbf{G}} \mathbf{G}^{-1} S_g^{\til{\mathbf{G}}},\\
\square \log \det \left( \mathbf{G} \til{\mathbf{G}}^{-1} \right) =&\ - \tr S_g^{\til{\mathbf{G}}},\\
\square \brs{\gU(\mathbf{G},\til{\mathbf{G}})}^2_{g,\mathbf{G}^{-1},\mathbf{G}} =&\ - \brs{\N \gU}^2_{g,\mathbf{G}^{-1},\mathbf{G}} - \brs{\bar{\N} \gU + \bar{T} \cdot \gU}^2_{g,\mathbf{G}^{-1},\mathbf{G}} - \IP{ F^2,  \gU^2}_{g}\\
&\ + T \star \gU \star F_{\til{\mathbf{G}}} + \gU^2 \star F_{\til{\mathbf{G}}} + \gU \star \til{\N} F_{\til{\mathbf{G}}},
\end{align*}
where $F_{\til{\mathbf{G}}}$ is the Chern curvature of $\til{\mathbf{G}}$.
\begin{proof} These claims follow directly from Lemmas \ref{l:hSchwarz} and \ref{l:hconnectionflow}, noting that $\mathbf{G}_t$ itself satisfies the HYM flow with respect to the time-dependent metric $g_t$.
\end{proof}
\end{lemma}

We are ready to prove the main result of this section.

\begin{theorem}\label{t:GGEK} (cf. Theorem \ref{t:GGEKintro}) 
Assume that the pairing $\IP{,}_{\mathfrak k}$ is negative-definite. Let $(\gw_t,\beta_t,h_t)$ a solution to \eqref{eq:CPFhol} defined on $[0, \tau), \tau \leq 1$, such that the initial condition satisfies the Bianchi identity \eqref{eq:BIGHM}, with associated holomorphic string algebroid $\mathcal{Q}$. Suppose that $\mathcal{Q}$ admits a background metric $\til{\mathbf{G}}$ and constants $\gl,\gL > 0$ so that, for all $(x,t)$, the family of Hermitian metrics $\mathbf{G}_t = \mathbf{G}(\omega_t,\beta_t,h_t)$ satisfies
\begin{align*} 
\gl \til{\mathbf{G}} \leq \mathbf{G} \leq \gL \til{\mathbf{G}}.
\end{align*}
Given $k \in \mathbb N$ there exists $C = C(k,n,\gl,\gL,\til{\mathbf{G}})$ such that
\begin{align*}
    \sup_{M \times [0,\tau)} t \Phi_k(\mathbf{G},\til{\mathbf{G}}) \leq C.
\end{align*}
\end{theorem}

%\begin{theorem} \label{t:GGEK} (cf. Theorem \ref{t:GGEKintro}) 
%Fix $\mathcal{Q} \to (M, J)$ a holomorphic string algebroid with negative-definite pairing $\IP{,}_{\mathfrak k}$. Given $\mathbf{G}_t = \mathbf{G}(\gw_t,\beta_t,h_t)$ a solution to the coupled HYM flow equation \eqref{eq:DFlow} on $\mathcal{Q}$ defined on $[0,\tau), \tau \leq 1$, suppose there exists a background Hermitian metric $\til{\mathbf{G}}$ and constants $\gl,\gL > 0$ so that for all $(x,t)$,
%\begin{align*} 
%\gl \til{\mathbf{G}} \leq \mathbf{G} \leq \gL \til{\mathbf{G}}.
%\end{align*}
%Given $k \in \mathbb N$ there exists $C = C(k,n,\gl,\gL,\til{\mathbf{G}})$ such that
%\begin{align*}
%    \sup_{M \times [0,\tau)} t \Phi_k(\mathbf{G},\til{\mathbf{G}}) \leq C.
%\end{align*}
%\end{theorem}
\begin{proof} We start with the final equation of Lemma \ref{l:GGfloweqns}.  As the pairing $\IP{,}_{\mathfrak k}$ is definite, it follows that $F^2 \geq0$.  Using this fact and Lemma \ref{l:GGfloweqns}, it follows that
\begin{align*}
    \square \brs{\gU(\mathbf{G},\til{\mathbf{G}})}^2_{g,\mathbf{G}^{-1},\mathbf{G}} \leq&\ C \left(1 + \brs{\gU(\mathbf{G},\til{\mathbf{G}})}_{g,\mathbf{G}^{-1},\mathbf{G}}^2 \right),
\end{align*}
where the constant $C$ depends on the Chern curvature of $\til{\mathbf{G}}$.  Using again Lemma \ref{l:GGfloweqns} together with the assumed bounds on $\mathbf{G}$, it follows that there is a constant $A$ depending on $n,\gl,\gL$, and $\til{\mathbf{G}}$ such that for all times $t \leq 1$,
\begin{align*}
    \square \left( t \brs{\gU(\mathbf{G},\til{\mathbf{G}})}^2_{g,\mathbf{G}^{-1},\mathbf{G}} + A \tr_{\til{\mathbf{G}}} \mathbf{G} \right) \leq&\ C,
\end{align*}
where now $C$ depends on $n,\gl,\gL$ and $\til{\mathbf{G}}$.  By the maximum principle, the claimed estimate for $k = 0$ follows.  Using this, the case of general $k$ follows by a standard blowup argument, as detailed in the case of pluriclosed flow in \cite[Theorem 4.2]{JordanStreets}.
\end{proof}

\subsection{Convergence on flat backgrounds}

Building on the arguments of the previous subsection, we obtain global existence and convergence results on special backgrounds.  We begin with an algebraic lemma. %Recall from Lemma \ref{t:Ggeneralized1} that any generalized Hermitian metric $(\omega,\beta,h)$ on a holomorphic string algebroid $\mathcal{Q}$ determines a Hermitian metric Hermitian $\mathbf{G} = \mathbf{G}(\omega,\beta,h)$.

\begin{lemma}\label{l:genmetricbounds}
    Let $\mathbf{G}=\mathbf{G}(\omega,\beta,h)$ and $\til{\mathbf{G}}=\mathbf{G}(\til{\omega},0,\til{h})$ be Hermitian metrics on $\mathcal{Q}_{2\i\partial\til{\omega},A^{\til{h}}}$ induced via Lemma \ref{t:Ggeneralized1}.  Then
    \begin{equation}
    \begin{split}
        \mathrm{tr}_{\mathbf{G}} \til{\mathbf{G}} = & \ \mathrm{tr}_{\til{\mathbf{G}}} \mathbf{G}\\
        = & \ \mathrm{tr}_g \til{g} + \mathrm{tr}_{\til{g}} g +|\alpha|^2_{g} + |\alpha|^2_{\til{g}} +\tfrac{1}{4}|\beta - \langle \ga \otimes \ga \rangle_{\mathfrak k} |^2_{g,\til{g}} + \mathrm{rk}(\mathrm{ad} \; P),
    \end{split}
    \end{equation}
where $\alpha = A^{\til{h}}-A^h$ and $(\IP{ \alpha \otimes \alpha}_{\mathfrak k})_{ij} = \IP{\ga_i, \ga_j}_{\mathfrak k}$.  In particular,
\begin{align*}
    \tr_\mathbf{G} \til{\mathbf{G}} \geq 2 (\dim M) + \mathrm{rk}(\mathrm{ad} \; P),
\end{align*}
with equality if and only if $g = \til{g}, \ga = 0, \gb = 0$.
\end{lemma}
\begin{proof}
Let $\mathbf{G}'=\mathbf{G}(\omega,0,h)$ be the Hermitian metric on $\mathcal{Q}=\mathcal{Q}_{2\i\partial\omega,A^h}$ as in Lemma \ref{t:Ggeneralized1}.  Then, in the splitting $\mathcal{Q}\cong T^{1,0}\oplus \mathrm{ad} \; P \oplus T^*_{1,0}$, we have
    \begin{align*}
        \mathbf{G}' = \left(\begin{array}{c c c}
        g &  & \\
         & -\langle \cdot,\cdot \rangle_{\mathfrak k} & \\
         &  & \tfrac{1}{4}g^{-1}
\end{array}\right).
    \end{align*}
An orthonormal basis of $\mathbf{G}'$ is given by $\{ v_i, s_j,2g(\overline{v_i})\}$ for basis $\{v_i\}$ of $T^{1,0}$ and $\{s_j\}$ of $\mathrm{ad} \; P$ such that
$$
g(v_i,\overline{v_j}) =  -\langle s_i,\overline{s_j}\rangle_{\mathfrak k} = \delta_{ij}.
$$
Therefore, an orthonormal basis of $\mathbf{G}=\varphi_*\mathbf{G}'$ is given by $\{ \varphi(v_i), \varphi(s_j),2\varphi(g(\overline{v_i}))\}$, where $\varphi$ is the holomorphic endomorphism of $\mathcal{Q}$ given by
$$
\varphi = \left( \begin{array}{c c c}
\mathrm{Id} & 0 & 0\\
\alpha & \mathrm{Id} & 0\\
\beta-\langle \alpha\otimes  \alpha \rangle_{\mathfrak k} & -2\langle \alpha, \cdot\rangle_{\mathfrak k} & \mathrm{Id}
\end{array}\right).
$$
Now, we compute
\begin{align*}
    \mathrm{tr}_\mathbf{G} \til{\mathbf{G}} = & \ \textstyle \sum_i \til{\mathbf{G}}(\varphi(v_i),\varphi(v_i)) + 4\til{\mathbf{G}}(\varphi(g(\overline{v_i})),\varphi(g(\overline{v_i})) + \textstyle \sum_j \til{\mathbf{G}}(\varphi(s_j),\varphi(s_j))\\
    = & \ \textstyle \sum_i \til{g}(v_i,\overline{v_i}) - \langle \alpha(v_i),\overline{\alpha(v_i)}\rangle_{\mathfrak k} + \tfrac{1}{4}\til{g}^{-1}\left(\beta(v_i)-\langle \alpha(v_i),\alpha\rangle_{\mathfrak k},\overline{\beta(v_i)}-\overline{\langle \alpha(v_i),\alpha\rangle_{\mathfrak k}}\right)\\
    & + \til{g}^{-1}\left(g(\overline{v_i}),g(v_i)\right)+\textstyle \sum_j -\langle s_j, s_j\rangle_{\mathfrak k} +\til{g}^{-1}\left(\langle\alpha,s_j\rangle_{\mathfrak k},\overline{\langle\alpha,s_j\rangle_{\mathfrak k}}\right)\\
    = & \ \mathrm{tr}_g \til{g} + \mathrm{tr}_{\til{g}} g +|\alpha|^2_{g} +|\alpha|^2_{\til{g}} +\tfrac{1}{4}|\beta - \IP{\ga \otimes \ga}_{\mathfrak k}|^2_{g,\til{g}} + \mathrm{rk}(\mathrm{ad} \; P).
\end{align*}
\end{proof}

We proof next the main result of this section.

\begin{theorem} \label{t:flatspaceconv} (cf. Theorem \ref{t:flatspaceconvintro})  Let $\mathcal{Q} \to (M, J)$ denote a holomorphic string algebroid with negative-definite pairing $\IP{,}_{\mathfrak k}$.  Suppose that $\mathcal{Q}$ admits a Hermitian metric $\mathbf{G}_F$ which is Chern-flat. Given a generalized Hermitian metric $(\omega_0,\beta_0,h_0)$ on $\mathcal{Q}$, the solution to pluriclosed flow \eqref{eq:CPFhol} with this initial data exists on $[0, \infty)$ and converges to a generalized Hermitian metric $(\omega_\infty,\beta_\infty,h_\infty)$ on $\mathcal{Q}$, such that $\mathbf{G}_{\infty} = \mathbf{G}(\omega_\infty,\beta_\infty,h_\infty)$ is Chern-flat.
\begin{proof} 
Since $\mathbf{G}_F$ is Chern-flat, Lemma \ref{l:GGfloweqns} implies
\begin{align} \label{f:flatconv5}
    \square \tr_{\mathbf{G}_t} \mathbf{G}_F =&\ - \brs{ \gU(\mathbf{G}, \mathbf{G}_F)}^2_{g, \mathbf{G}^{-1}, \mathbf{G}_F} \leq 0.
\end{align}
Thus by the maximum principle, for any smooth existence time $t$ one has
\begin{align*}
    \sup_{M \times \{t\}} \tr_{\mathbf{G}_t} \mathbf{G}_F \leq \sup_{M \times \{0\}} \tr_{\mathbf{G}_0} \mathbf{G}_F.
\end{align*}
Furthermore, again by Lemma \ref{l:GGfloweqns}, using that $\mathbf{G}_F$ is Chern-flat,
\begin{align} \label{f:flatconv15}
\square \brs{\gU(\mathbf{G}_t, \mathbf{G}_F)}^2_{g^{-1},\mathbf{G}^{-1},\mathbf{G}} \leq&\ 0.
\end{align}
It follows that for $C$ chosen large with respect to the initial data,
\begin{align*}
    \square \left( t \brs{\gU(\mathbf{G}_t, \mathbf{G}_F)}^2_{g^{-1},\mathbf{G}^{-1},\mathbf{G}} + C \tr_{\mathbf{G}_t} \til{\mathbf{G}} \right) \leq 0.
\end{align*}
Thus there exists a uniform constant $C$ such that
\begin{align} \label{f:flatconv20}
    \sup_{M \times \{t\}} t \brs{\gU(\mathbf{G}_t, \mathbf{G}_F)}_{g^{-1},\mathbf{G}^{-1},\mathbf{G}}^2 \leq C.
\end{align}
By Theorem \ref{t:GGEKintro} there are a priori estimates on all derivatives of $\mathbf{G}_t$ away from time zero.  The long time existence of the flow then follows combining Lemma \ref{l:genmetricbounds} with Corollary \ref{c:flowregularitycor}.  Furthermore, by (\ref{f:flatconv20}), for any sequence of times $\{t_i\} \to \infty$ there exists a subsequence such that $\mathbf{G}_{t_i} \to \mathbf{G}_{\infty}$, where $\mathbf{G}_{\infty}$ is Chern-flat.  Repeating the estimate of (\ref{f:flatconv5}) using the background metric $\mathbf{G}_{\infty}$ it follows that $\lim_{t \to \infty} \mathbf{G}_t = \mathbf{G}_{\infty}$.  By Lemma \ref{l:genmetricbounds}, the convergence of $\mathbf{G}_t$ implies convergence of the associated metrics $g_t$ and connections $A^{h_t}$. This implies that the limiting connection $A_{\infty}$ has holonomy contained in $K$, and is thus associated to a reduction $h_{\infty}$.
\end{proof}
\end{theorem}

An interesting question is to classify string algebroids admitting Chern-flat Hermitian metrics. Following \cite{SGFLS,GFGM}, the existence of such a metric on a holomorphic string algebroid $\mathcal{Q}$ reduces to find a solution $(\gw,h)$ of the coupled system
\begin{equation}
\label{eq:CoupledGinstantonExp}
\begin{split}
    R_{\nabla^{-}}-\mathbb{F}^{\dagger}\wedge \mathbb{F}  & = 0,\\
    \nabla^{+,h}F_h  & = 0,\\
    [F_h,\cdot] - \mathbb{F}\wedge \mathbb{F}^{\dagger} 
    & = 0,\\
	dd^c\omega + \langle F_h \wedge F_h \rangle_{\mathfrak{k}} 
    & = 0,
\end{split}
\end{equation}
where $\mathbb{F}^{\dagger}\wedge \mathbb{F}$ and $\mathbb{F}\wedge \mathbb{F}^{\dagger}$ denote endomorphism-valued two-forms, defined by
\begin{align*}
i_Yi_X\mathbb{F}^{\dagger}\wedge \mathbb{F}(Z) & = g^{-1}\IP{i_Y F_h,F_h(X,Z)}_\mathfrak{k} - g^{-1}\IP{i_X F_h,F_h(Y,Z)}_\mathfrak{k},\\
i_Yi_X\mathbb{F}\wedge \mathbb{F}^{\dagger}(r) & = F_h(Y,g^{-1}\IP{i_XF_h,r}_\mathfrak{k}) - F_h(X,g^{-1}\IP{i_YF_h,r}_\mathfrak{k}),
\end{align*}
and $R_{\nabla^{-}}$ denotes the curvature of the metric connection $\nabla^- = \nabla + \tfrac{1}{2}g^{-1}d^c\omega$. Given a solution $(\omega_F,h_F)$ of this system, a choice of generalized Hermitian metric $\mathbf{G}_0$ on the holomorphic vector bundle underlying $\mathcal{Q}$ is equivalent to the data of a pair $(\omega_0,h_0)$ satisfying the following Dolbeault version of the anomaly cancellation equation
\begin{equation*}\label{eq:structuralGHMintro}
[\partial(\omega_0 - \omega_F + R(h_0,h_F))] \in H^{2,1}_{\dbar}(M),
\end{equation*}
where $R(h_0,h_F)$ denotes the Bott-Chern secondary characteristic class associated to $(h_0,h_F)$ (see Section \ref{sec:Aeppli} and Remark \ref{rmk:twist}).  A basic example occurs when the base manifold admits a Bismut-flat metric and the bundle is flat and we show in our next result that in this case one obtains convergence to a geometry of this type.

\begin{corollary} \label{c:Bismutflatconv}
Let $(M, J, g_{F})$ denote a Bismut-flat manifold, and let $(P^c,\bar J) \to (M,J)$ be a holomorphic principal $K^c$-bundle with flat reduction $h_F$.  Let $\mathcal{Q} \to (M,J)$ denote the holomorphic string algebroid with negative-definite pairing $\IP{,}_{\mathfrak k}$ determined by this data.  Given a generalized Hermitian metric $(\omega_0,\beta_0,h_0)$ on $\mathcal{Q}$, the solution to pluriclosed flow \eqref{eq:CPFhol} with this initial data exists on $[0, \infty)$ and converges to a generalized Hermitian metric $(\omega_\infty,\beta_\infty,h_\infty)$ on $\mathcal{Q}$, where $g_{\infty}$ is Bismut-flat and $h_{\infty}$ is flat.
\end{corollary}

\begin{proof}
We first note that the data $(g_F, h_F)$ determines a Hermitian metric $\mathbf{G}_F$ on $\mathcal{Q}$ which is Chern-flat (see \eqref{eq:CoupledGinstantonExp}). The global existence and convergence of the flow to a Chern-flat limit $\mathbf{G}_{\infty}(g_{\infty}, h_{\infty})$ follows from Theorem \ref{t:flatspaceconv}.  We furthermore claim that $g_{\infty}$ is Bismut flat and $h_{\infty}$ is flat.  Note that it follows from \cite[Lemma 4.7]{GFGM} that if $h_{\infty}$ is Chern-flat then $g_{\infty}$ is Bismut-flat. By Lemma \ref{l:hSchwarz}, since $h_F$ is flat along the flow we obtain
\begin{align*}
\square \left( \tr_{h_F} h + \tr_h h_F \right) =&\ - \brs{\gU(h,h_F)}^2_{g,h_F^{-1},h} - \brs{\gU(h,h_F)}^2_{g,h^{-1}, h_F^{-1}} \leq 0,\\
\square \brs{\gU(h,h_F)}^2_{g,h^{-1},h} =&\ - \brs{\N \gU}^2_{g,h^{-1},h} - \brs{\bar{\N} \gU + \bar{T} \cdot \gU}^2_{g,h^{-1},h} - \IP{ F^2,  \gU^2}_{g} \leq 0.
\end{align*}
Arguing as in Theorem \ref{t:flatspaceconv}, we obtain for large enough $A > 0$
\begin{align*}
    \square \left( t \brs{\gU(h, h_F)}^2_{g,h^{-1},h} + A \tr_{h_F} h \right) \leq&\ \brs{\gU(h, h_F)}^2_{g,h^{-1},h} - A \brs{\gU(h,h_F)}^2_{g,h_F^{-1},h} \leq 0.
\end{align*}
It follows from the maximum principle that $h_{\infty}$ is Chern-flat, as claimed.
\end{proof}

\section{Pluriclosed flow and geometrization} \label{s:geometrization}

As discussed in the introduction, we may expect the string algebroid pluriclosed flow to be a tool for addressing uniformization problems in complex geometry. In this section, we 
record a number of corollaries, including a scalar curvature monotonicity result and two gradient flow interpretations in the case of negative-definite Lie algebra pairing, related to geometrization. In particular, we show in the presence of a holomorphic volume form that the flow \eqref{f:PCFSA} is the gradient flow of the \emph{dilaton functional}, as introduced in \cite{garciafern2018canonical}. Finally, for complex threefolds with holomorphically trivial canonical bundle and indefinite Lie algebra pairing, we discuss a conjectural picture for string algebroid pluriclosed flow and its relationship to `Reid's fantasy' \cite{Reid1987moduli}.

\subsection{Scalar curvature monotonicity}\label{sec:Scalarmon}

In \cite{Streetsscalar} a family of scalar curvature monotonicity formulas was established for generalized Ricci flow on exact Courant algebroids.  The key ingredient is an auxiliary scalar field along generalized Ricci flow, called the \emph{dilaton flow}. Given a smooth manifold $M$ and closed three-form $H_0$, fix $(g_t, b_t)$ a solution of generalized Ricci flow \eqref{f:GRF}. Given an arbitrary smooth function $\phi_0$, there exists a unique solution to the dilaton flow
\begin{align} \label{f:exactdilflow}
\square \phi = \tfrac{1}{6} \brs{H}^2
\end{align}
with initial data $\phi_0$.  Associated to the data $(g, b, \phi)$ we have the weighted (Bakry-\'Emery) Ricci and scalar curvatures (cf. Section \ref{sec:dimensionalred}):
\begin{align*}
    \Rc^{H,\phi} := \Rc^{\N^+} + \ \N^+ d\phi, \quad R^{H,\phi} := R - \tfrac{1}{12} \brs{H}^2 + 2 \gD \phi - \brs{d \phi}^2,
\end{align*}
where $\N^+ = \N + \tfrac{1}{2} g^{-1} H$. Then by \cite[Proposition 1.1]{Streetsscalar}  we have
\begin{align} \label{f:exactscalar}
    \square R^{H,\phi} = 2 \brs{\Rc^{H,\phi}}^2.
\end{align}
In particular, the previous formula implies that a lower bound on $R^{H,\phi}$ is preserved on a compact manifold.

In view of Proposition \ref{p:cGRFtoGRF}, this implies such a monotonicity for generalized Ricci flow and the generalized scalar curvature on string algebroids (see Section \ref{sec:GScalar}). To fix notation, recall that the generalized scalar curvature associated to $(G,\operatorname{div})$, a generalized metric and an exact divergence on a string algebroid $E$, is given by (see Definition \ref{def:GScalar})
\begin{equation}\label{eq:Genscalarbis}
\mathcal{S}^+ = R_g-\tfrac{1}{12}|H|^2 - \tfrac{1}{2}|F_A|^2 + 2 \Delta\phi - \left|d\phi\right|^2,
\end{equation}
in terms of the associated data $(g, H, A,\phi)$ (cf. \eqref{f:BaEmR}).

%\begin{align} \label{f:scalar}
%    R^{H,A,\phi} := R_g - \tfrac{1}{12} |H|^2 - \tfrac{1}{2} \brs{F_A}^2 + \tfrac{1}{6} \brs{ [\cdot,\cdot]}^2 + 2 \gD_g \phi - \brs{\N \phi}^2.
%\end{align}

\begin{prop} \label{p:scalarmon} 
Fix  $(g_t, b_t, A_t)$ is a solution to string algebroid generalized Ricci flow \eqref{eq:GRFgbA} such that the initial condition satisfies the Bianchi identity \eqref{eq:bianchitrans}.  Fix $\phi_0 \in C^{\infty}(M)$ and let $\phi_t$ be the unique solution to
\begin{align} \label{f:stringdilatonflow}
\square \phi =&\ \tfrac{1}{6} \left( \brs{H}^2 + 3 \brs{F_A}^2 +\brs{[ \cdot, \cdot]}^2 \right).
\end{align}
Then
\begin{align*}
    \square \mathcal{S}^+ =&\ 2 \brs{\Rc^{H,\phi}
 - F_A^2 }^2 + 2 \brs{d^{\star}_A F_A - F_A \lrcorner H + i_{d\phi^\sharp} F_A}^2.
\end{align*}

%\begin{align*}
%    \square \mathcal{S}^+ =&\ 2 \brs{\Rc^g -\tfrac{1}{4}H^2 - F_A^2 -\tfrac{1}{2}d^{\star}_gH + \nabla^+ d\phi }^2 + 2 \brs{d^{\star}_A F_A - F_A \lrcorner H + i_{\nabla \phi} F_A)\rangle}^2.
%\end{align*}

\begin{proof} By Proposition \ref{p:cGRFtoGRF}, the data $(g_t, b_t, A_t)$ induces a solution $(\bar{g}_t, \bar{b}_t)$ of generalized Ricci flow on an exact Courant algebroid $\bar{E}$ over the principal bundle $p: P \to M$.  We claim that the one-parameter family of functions $\bar{\phi}_t = p^* \phi_t$ is a solution to the dilaton flow (\ref{f:exactdilflow}) on $P$ with initial data $p^*\phi_0$.  Since $X p^* f = 0$ for any vertical vector field $X$ and $f \in C^{\infty}(M)$, it follows that $\gD_{\bar{g}_t} \bar{\phi}_t = p^*\gD_{g_t} \phi_t$ for all $t$. The equation (\ref{f:exactdilflow}) for $\bar{\phi}_t$ follows from this and Proposition \ref{prop:Bismutscalar}.  Using this, equation \eqref{f:exactscalar}, and the form of the Ricci curvature on $P$ as expressed in Proposition \ref{prop:BismutRic}, the statement follows from  \cite[Proposition 1.1]{Streetsscalar}.
\end{proof}
\end{prop}

%\begin{remark}
%We note the constant term $\brs{[ \cdot, \cdot]}^2$ in the string algebroid version of the dilaton flow, which is not apparent from our discussion in Section \ref{sec:GScalar} (cf. \cite[Section 2.4]{SGFLS}).
%\end{remark}

\begin{remark}
Alternatively, the generalized scalar curvature monotonicity in our setting follows from the recent work \cite{SCSV}, which applies to a general class of Ricci flows on transitive Courant algebroids.
\end{remark}

\subsection{Generalized Ricci flow on string algebroids as a gradient flow}

Recall that the generalized Ricci flow on exact Courant algebroids \eqref{f:GRF} is the gradient flow of a Perelman-type functional \cite{OSW}. More precisely, we fix a closed three-form $H_0$ on a smooth compact manifold $M$ and, given a pair $(g, b)$, we set $H = H_0 + db$ and define
\begin{align*}
    \mathcal F(g, b, f) =&\ \int_M \left( R_g - \tfrac{1}{12} \brs{H}^2 + \brs{\N f}^2 \right) e^{-f} dV_g,\\
    \gl(g, b) =&\ \inf_{\{ f \in C^\infty(M)\ |\ \int_M e^{-f} dV_g = 1 \}} \FF(g,b,f).
\end{align*}
It was shown in \cite{OSW} that generalized Ricci flow (\ref{f:GRF}) is the gradient flow of $\gl$. Applying now Proposition \ref{p:cGRFtoGRF}, we obtain that the string algebroid generalized Ricci flow \eqref{eq:GRFgbA} is a symmetry reduction of \eqref{f:GRF}. We claim that this implies that \eqref{eq:GRFgbA} is also a gradient flow. Given $E$ a string algebroid over $M$, with underlying principal bundle $P$, fix a background generalized metric, which determines a three-form $H_0$ and principal connection $A_0$ satisfying the Bianchi identity \eqref{eq:bianchitrans0}. For the data $(g,b,A)$ defining a generalized metric as in \S \ref{s:GRFSA} and $f \in C^{\infty}(M)$ we define
\begin{equation}\label{eq:FFlambdastring}
\begin{split}
\FF(g,b,A,f) =&\ \int_M \left( R_g - \tfrac{1}{12} \brs{H}^2 - \tfrac{1}{2} \brs{F_A}^2 + \brs{\N f}^2 \right) e^{-f} dV_g\\
\gl(g,b,A) =&\ \inf_{\{ f \in C^\infty(M) \ |\ \int_M e^{-f} dV_g = 1 \}} \FF(g,b,A,f),
\end{split}
\end{equation}
where $H$ is given by \eqref{eq:Htdef}.

\begin{prop} \label{p:gradient} Fix  $(g_t, b_t, A_t)$ is a solution to string algebroid generalized Ricci flow \eqref{eq:GRFgbA} such that the initial condition satisfies the Bianchi identity \eqref{eq:bianchitrans}. Then, the functional $\gl$ in \eqref{eq:FFlambdastring} is monotone nondecreasing and constant if and only if $(g_t, b_t, A_t)$ is a gradient soliton, that is, there exists $f \in C^{\infty}(M)$ such that
\begin{equation}\label{eq:GRsoliton}
\begin{split}
\operatorname{Rc} - \frac{1}{4}H^2 - F_A^2 + \frac{1}{2} L_{df^\sharp}g & = 0,\\
d^\star H +  i_{df^\sharp} H & = 0,\\ d_A^\star F_A - F_A \lrcorner \ H + i_{df^\sharp} F_A & = 0.
\end{split}
\end{equation}
Furthermore, $e^{-f}$ is an eigenvector for the minimum eigenvalue of the Schr\"odinger operator
$$
- 4 \Delta +  R - \tfrac{1}{12} |H|^2 - \tfrac{1}{2} \brs{F_A}^2 + \tfrac{1}{6} \brs{ [\cdot,\cdot]}^2.
$$

\begin{proof} 
Consider the flow line $(\bar{g}_t, \bar{b}_t)$ of generalized Ricci flow on the total space of the principal bundle $P \to M$, induced by the data $(g_t, b_t, A_t)$ via Proposition \ref{p:cGRFtoGRF}. The corresponding functional 
\begin{align*}
    %\mathcal F(g, b, f) =&\ \int_M \left( R_g - \tfrac{1}{12} \brs{H}^2 + \brs{\N f}^2 \right) e^{-f} dV_g,\\
    \gl(\bar{g}, \bar{b}) =&\ \inf_{\{ \bar{f} \in C^\infty(P)\ |\ \int_P e^{-\bar{f}} dV_{\bar{g}} = 1 \}} \FF(\bar{g},\bar{b},\bar{f}).
\end{align*}
is monotone nondecreasing along $(\bar{g}_t, \bar{b}_t)$ by \cite[Corollary 6.10]{GRFBook}, and its value is attained by the unique normalized eigenvector $e^{-\bar{f}}$ for the minimum eigenvalue of the Schr\"odinger operator (see \cite[Lemma 6.3]{GRFBook})
$$
- 4 \Delta_{\bar{g}} +  R_{\bar{g}} - \tfrac{1}{12} |\bar{H}|^2_{\bar{g}}.
$$
Since this operator is $K$-invariant and the lowest eigenspace is $1$-dimensional, it follows that $\bar{f} \in C^\infty(P)^K \cong C^\infty(M)$. The statement follows now combining Proposition \ref{prop:BismutRic}, Proposition \ref{prop:Bismutscalar} and \cite[Corollary 6.10]{GRFBook}.
%The functional which appears via reduction is
%\begin{align*}
%\FF(g,b,A,f) =&\ \int_M \left( R_g - \tfrac{1}{12} |H|^2 - \tfrac{1}{2} \brs{F_A}^2 + \tfrac{1}{6} \brs{ [\cdot,\cdot]}^2) e^{-f}dV_g,
%\end{align*}
%but the corresponding $\lambda$-functional is equivalent to the one above because of the normalization of the dilaton $\int_M e^{-f}e^{-f}dV_g = 1$.
\end{proof}
\end{prop}

The formal computations underlying the evolution of $\gl$ follow independent of the signature of $\IP{,}_{\mathfrak k}$, but we we will only get a monotonicity formula in the case that it is definite.  For emphasis we note that there exist nontrivial gradient solitons on exact Courant algebroids over compact manifolds $M$ (\cite{Streetssolitons, SU1}).  These are actually solitons for the generalized Ricci flow on string algebroids, and in fact were originally constructed using a special case of the dimension reduction principle established here.

\subsection{Dilaton functional and pluriclosed flow}\label{sec:dilaton}

Fix $(P^c,\bar J)$ a holomorphic principal $K^c$-bundle over a compact complex manifold $(M,J)$. We assume that $(M,J)$ admits a holomorphic volume form $\Omega$ and that $P^c$ has vanishing first Pontyagin class with respect to a fixed pairing $\IP{,}_\mathfrak{k}$ in the Lie algebra of a maximal compact $K \subset K^c$. Given a Hermitian metric $g$ on $(M,J)$, we define the norm $\|\Omega\|_ \omega$ by
$$
\|\Omega\|_ \omega^2\frac{\omega^n}{n!} = (-1)^{\frac{n(n-1)}{2}} (\i)^n \Omega \wedge \overline \Omega.
$$
On the set of pairs $(\omega,h)$ satisfying the Bianchi identity \eqref{eq:BIGHM}, consider the \emph{dilaton functional}
$$
\cM(\omega,h) =  \int_M \|\Omega\|_ \omega \frac{\omega^n}{n!}.
$$
Restricted to an Aeppli class of such pairs $(\omega,h)$ (see Section \ref{sec:Aeppli}), it was proved in \cite{garciafern2018canonical} that the critical points of $\cM$ are solutions of the Hull-Strominger system \eqref{f:HS}. Assuming that $\IP{,}_\mathfrak{k}$ is negative-definite, in this section we prove that the one-form reduction of the pluriclosed flow in Lemma \ref{lem:oneformred} is the gradient flow of $\mathcal{M}$ restricted to a suitable space of potentials for a fixed Aeppli class.

Let $\mathcal{Q}$ be a holomorphic string algebroid over $(M,J)$ with underlying principal bundle $(P^c,\bar J)$ and pairing $\IP{,}_\mathfrak{k}$. We assume that the space of generalized Hermitian metrics $B_{\mathcal{Q}}^+$ is non-empty (see Proposition \ref{propo:Chernclassic}).  Without loss of generality, we can therefore take $\mathcal{Q} = \mathcal{Q}_{2\i\partial \omega_0,A^{h_0}}$ for a pair $(\omega_0,h_0)$ satisfying the Bianchi identity \eqref{eq:BIGHM}. Consider the space of \emph{Aeppli potentials}
$$
\mathcal{K}^+ = \{(\xi,h) \; | \; \omega_{\xi,h}:= \omega_0 + \dbar \xi + \partial \bar \xi - \tilde R(h,h_0) > 0\} \subset \Lambda^{1,0} \times \Gamma(P^c/K),
$$
where, for the unique $s_h \in \Gamma(\ad P_{h_0})$ such that $h = h_0 e^{\i s_h}$, we set
$$
\tilde R(h,h_0) := \i \int_0^1 \IP{h_t^{-1}\dot h_t ,F_{h_t}}_{\mathfrak{k}} dt \in \Lambda^{1,1}_\RR, \qquad h_t := h_0 e^{\i t s_h}.
$$

\begin{lemma}\label{lem:potential}
Let $B_0^+ \subset B_{\mathcal{Q}}^+$ denote the space of generalized Hermitian metrics on $\mathcal{Q}$ with vanishing Aeppli class (see Definition \ref{def:Aeppli}). Then, there is a well-defined map
$$
\mathcal{K}^+ \to B_0^+ \colon (\xi,h) \mapsto \mathbf{h}(\xi,h):=(\omega_{\xi,h},\partial \xi + B^{2,0},h).
$$
where, for $\alpha_t = A^{h_t} - A^{h_0}$,
$$
B^{2,0} := \frac{\i}{2}\int_0^1 \IP{\alpha_t \wedge \dot \alpha_t} dt \in \Lambda^{2,0}.
$$
\end{lemma}

\begin{proof}
By \cite[Lemma 4.2]{garciafern2020gauge}, we have
\begin{align*}
2\i\partial \omega_{\xi,h} & = 2\i\partial \omega_0 - 2 \i d(\partial \xi) - 2 \i \partial \tilde R(h,h_0) \\
& = 2\i\partial \omega_0 - 2 \i d(\partial \xi + B^{2,0}) + 2 \langle \alpha \wedge F_{h_0} \rangle_{\mathfrak{k}} + \langle \alpha \wedge d_{h_0}\alpha \rangle_{\mathfrak{k}} + \frac{1}{3}\langle \alpha \wedge [\alpha \wedge \alpha] \rangle_{\mathfrak{k}},
\end{align*}
for $A^h = A^{h_0} + \alpha$. The proof follows from Proposition \ref{propo:Chernclassic} and Definition \ref{def:Aeppli}.
\end{proof}

The tangent space at $(\xi,h)\in \mathcal{K}^+$ to the space of Aeppli potentials can be identified with a subspace of the kernel of the anchor map $\pi \colon \mathcal{Q} \to T^{1,0}$
$$
T_{(\xi,h)} \mathcal{K}^+ \cong \Lambda^{1,0} \oplus \i \ad P_h \subset \Ker \pi
$$
via 
$$
(\dot \xi, \dot h) \mapsto \gamma(\dot \xi, \dot h):= \sqrt{2} \dot \xi + \langle \alpha,h^{-1}\dot h \rangle_{\mathfrak{k}} + \tfrac{1}{2}h^{-1}\dot h.
$$
Consequently, using Definition \ref{d:generalizedmetriHerm}, $\mathcal{K}^+$ inherits an $L^2$-Hermitian metric associated to the natural measure $\|\Omega\|_{\omega} \frac{\omega^n}{n!}$. The next result is a direct consequence of Lemma \ref{t:Ggeneralized1}.

\begin{lemma}\label{lem:L2metric}
The hermitian metric on $\mathcal{K}^+$ is given by 
\begin{equation*}
\begin{split}
\IP{(\dot \xi, \dot h),(\dot \xi, \dot h)} & := \int_M\mathbf{G}_{\mathbf{h}(\xi,h)}(\gamma(\dot \xi,\dot h),\gamma(\dot \xi,\dot h)) \|\Omega\|_{\omega} \frac{\omega^n}{n!}\\
& = \frac{1}{2}\int_M \IP{\dot \xi,\overline{\dot \xi}}_g \|\Omega\|_{\omega} \frac{\omega^n}{n!} - \frac{1}{4}\int_M\IP{h^{-1}\dot h,\overline{h^{-1}\dot h}^h}_\mathfrak{k} \|\Omega\|_{\omega} \frac{\omega^n}{n!},
\end{split}
\end{equation*}
for $(\dot \xi, \dot h) \in T_{(\xi,h)} \mathcal{K}^+$, where $\omega := \omega_{\xi,h}$ and $\alpha = A^h - A^{h_0}$.
\end{lemma}

%\begin{proof}
%\begin{equation*}
%\begin{split}
%\mathbf{G}(r + \xi,r + \xi) & = - \IP{r ,\overline{r}^h}_\mathfrak{k} + \tfrac{1}{4}\IP{\xi - 2\langle \alpha,r\rangle_{\mathfrak{k}},\overline{\xi} - 2\overline{\langle \alpha,r\rangle}_{\mathfrak{k}}}_g,
%\end{split}
%\end{equation*}
%$$
%(\dot \xi, \dot h) \mapsto \gamma(\dot \xi, \dot h):= \sqrt{2} \dot \xi + \langle \alpha,h^{-1}\dot h \rangle_{\mathfrak{k}} + \tfrac{1}{2}h^{-1}\dot h.
%$$
%\begin{equation*}
%\begin{split}
%\mathbf{G}(\gamma,\gamma) & = - \frac{1}{4}\IP{r ,\overline{r}^h}_\mathfrak{k} + \frac{1}{2}\IP{\xi,\overline{\xi} }_g,
%\end{split}
%\end{equation*}
%\end{proof}

Applying Lemma \ref{lem:potential}, we can now regard the dilaton functional above as a map on the space $\mathcal{K}^+$ of Aeppli potentials 
$$
\cM \colon \mathcal{K}^+ \to \RR \colon (\xi,h) \mapsto \int_M \|\Omega\|_{\omega_{\xi,h}} \frac{\omega_{\xi,h}^n}{n!}.
$$
We calculate next the gradient flow of $\cM$ with respect to the Hermitian metric in Lemma \ref{lem:L2metric}. Following Lemma \ref{lem:oneformred}, we consider the Hermitian metric on the canonical bundle %such that $\|\Omega\| \equiv 1$ .
determined by the volume form $(-1)^{\frac{n(n-1)}{2}} (\i)^n \Omega \wedge \overline \Omega$. The corresponding representative of $c_1(M,J) \in H^{1,1}_A(M,\RR)$ identically vanishes, and hence the one-form reduction of pluriclosed flow in Lemma \ref{lem:oneformred} simplifies to the following expression
\begin{equation}\label{eq:oneformPCFCY}
\begin{split}
\dt \xi =&\ \delb^{\star}_{\gw} \gw + \i \del \log \|\Omega\|_\omega,\\
h^{-1}\dt h =&\ - S^h_g,\\
\dt \hat{\omega} =&\ \i \IP{S^h_g, F_h}_{\mathfrak k}, \qquad \gw = \hat{\gw} + \delb \xi + \del \overline{\xi}.
\end{split}
\end{equation}

\begin{proposition}\label{prop:dilatonfuncmonotone}
The gradient flow of the dilaton functional $\cM \colon \mathcal{K}^+ \to \RR$ in the space of Aeppli potentials is given by \eqref{eq:oneformPCFCY}.   Furthermore, if $(\xi_t,h_t)$ is a solution of \eqref{eq:oneformPCFCY}, one has
\begin{equation}\label{eq:dtM}
\begin{split}
\dt \cM (\xi_t,h_t) & = \frac{1}{2}\int_M \|\Omega\|_\omega \|\theta_\omega + d \log \|\Omega\|_\omega\|^2_\omega  \frac{\omega^{n}}{n!} - \frac{1}{2}\int_M \langle \Lambda_\omega F_A, \Lambda_\omega F_A \rangle_{\mathfrak{k}}\|\Omega\|_ \omega \frac{\omega^{n}}{n!}.
\end{split}
\end{equation}
Consequently, the dilaton functional is non-decreasing along the flow provided that $\langle , \rangle_{\mathfrak{k}}$ is negative semi-definite.
\end{proposition}

\begin{proof}
Applying \cite[Lemma 4.5]{garciafern2018canonical} and integration by parts it follows that the variation of $\cM$ at $(\xi,h) \in \mathcal{K}^+$ is
\begin{align*}
\delta \cM(\dot \xi, \dot h) % & = \frac{1}{2(n-1)!} \int_M \(\dbar\dot\xi + \partial \bar{\dot\xi} - \i \IP{h^{-1}\dot h ,F_h}_\mathfrak{k}\) \wedge \|\Omega\|_ \omega \omega^{n-1}\\
& = \frac{1}{2(n-1)!} \int_M \(d(\dot\xi + \bar{\dot\xi}) - \i \IP{h^{-1}\dot h ,F_h}_\mathfrak{k}\) \wedge \|\Omega\|_ \omega \omega^{n-1}\\
%& = \frac{1}{2(n-1)!}\int_M (\dot\xi + \bar{\dot\xi}) \wedge d(\|\Omega\|_ \omega \omega^{n-1}) - \frac{1}{2} \int_M \IP{h^{-1}\dot h ,S_g^h}_\mathfrak{k} \|\Omega\|_ \omega \frac{\omega^{n}}{n!} \\
& = \frac{1}{2}\int_M (\dot\xi + \bar{\dot\xi}) \wedge (d\log \|\Omega\|_ \omega + \theta_\omega) \wedge  \|\Omega\|_ \omega \frac{\omega^{n-1}}{(n-1)!} - \frac{1}{2} \int_M \IP{h^{-1}\dot h ,S_g^h}_\mathfrak{k} \|\Omega\|_ \omega \frac{\omega^{n}}{n!} \\
%& = - \frac{1}{2}\int_M \|\Omega\|_ \omega (\dot\xi + \bar{\dot\xi}) \wedge * J (d\log \|\Omega\|_ \omega + \theta_\omega) - \frac{1}{2} \int_M \IP{h^{-1}\dot h ,S_g^h}_\mathfrak{k} \|\Omega\|_ \omega \frac{\omega^{n}}{n!}\\
%& = - \frac{1}{2}\int_M \|\Omega\|_ \omega (\dot\xi + \bar{\dot\xi}) \wedge * (d^c\log \|\Omega\|_ \omega - d^{\star}\omega) - \frac{1}{2} \int_M \IP{h^{-1}\dot h ,S_g^h}_\mathfrak{k} \|\Omega\|_ \omega \frac{\omega^{n}}{n!}\\
& = \frac{1}{2}\int_M  \IP{\dot\xi + \bar{\dot\xi}, d^{\star}\omega - d^c\log \|\Omega\|_ \omega}_g \|\Omega\|_ \omega \frac{\omega^n}{n!} - \frac{1}{2} \int_M \IP{h^{-1}\dot h ,\overline{S_g^h}^h}_\mathfrak{k} \|\Omega\|_ \omega \frac{\omega^{n}}{n!},
\end{align*}
for $\omega = \omega_{\xi,h}$, where we have used that $(n-1)!\eta = J \eta \wedge  \omega^{n-1}$ for $\eta \in \Lambda^1$ arbitrary. The statement follows now combining \eqref{eq:oneformPCFCY} with Lemma \ref{lem:L2metric}.
\end{proof}

%To finish this section, we obtain a rigidity result for fixed points of the pluriclosed flow in the case that $\langle , \rangle_{\mathfrak{k}}$ is negative semi-definite, by application of Lemma \ref{prop:UPCFexp2}.

%
%
\subsection{Long-time behaviour}

In this section we present some conjectures regarding the long-time behaviour of the string algebroid pluriclosed flow \eqref{eq:CPFhol} on compact complex manifolds which admit a holomorphic volume form. We provide evidence in the case that the pairing $\langle , \rangle_{\mathfrak{k}}$ is negative semi-definite, and speculate on a similar behaviour for the case of general pairing.

We fix a compact complex manifold $(M,J)$. % with holomorphic volume form $\Omega$. 
To begin, our long-time existence results in Section \ref{sec:Global} require a solution of the Bianchi identity equation \eqref{eq:BIGHM}, thus a natural first question is to identify holomorphic principal bundles which admit such a solution. Fix $(P^c,\bar J)$ a holomorphic principal $K^c$-bundle over $(M,J)$, a maximal compact $K \subset K^c$, and a pairing $\IP{,}_\mathfrak{k}$ in the Lie algebra of $K$. We assume that $P^c$ has vanishing first Pontyagin class in Bott-Chern cohomology with respect to $\IP{,}_\mathfrak{k}$, which is a necessary condition for solving \eqref{eq:BIGHM}. Given this, for any reduction $h$ of $P^c$ to $K$, there exists a potential $\tau \in \Lambda^{1,1}_\RR$ for the representative of the first Pontryagian class associated to $h$, that is,
\begin{equation}\label{eq:BIextra}
dd^c\tau = - \IP{F_h \wedge F_h}_\mathfrak{k}.
\end{equation}

\begin{question} 
Given a holomorphic principal bundle over $(M,J)$ with vanishing first Pontryagin class in Bott-Chern cohomology, does it admit a solution $(\tau,h)$ of \eqref{eq:BIextra} such that $\tau$ is a positive-definite Hermitian form?
\end{question}

For emphasis, note that this question is often posed in concert with other geometric conditions, as part of solving the full Hull-Strominger system \eqref{f:HS}.  Here we are asking to solve \emph{only} the Bianchi identity equation, to obtain initial data for the string algebroid pluriclosed flow (cf. \cite{GFGM}).  As described in Section \ref{s:dimredpcf}, in some cases a solution of the Bianchi identity corresponds precisely to an invariant Hermitian structure on the total space of $P_h$ being pluriclosed, and hence the problem is obstructed by the existence of a positive $dd^c$-exact $(1,1)$ current by the Hahn-Banach Theorem (cf. \cite{egidi2001special}).  More generally one expects to be able to give a more precise obstruction in terms of the geometry of the associated holomorphic string algebroid \cite{GFGM,GFGM2}. Pluriclosed manifolds automatically admit solutions for any bundle with vanishing first Pontryagin class.  %Also note that Clemens-Friedman manifolds have vanishing Bott-Chern cohomology in degree $(2,2)$ \cite{}, so the anomaly cancellation equation is always solvable for some $\gw$, although the positivity of $\gw$ remains open.

Given starting data for the flow satisfying the Bianchi identity equation, we can make formal conjectures on the existence and convergence behavior of string algebroid pluriclosed flow.  We first observe a basic lemma regarding the case of the dimension reduced ansatz of Section \ref{s:dimredpcf}.
\begin{lemma}
    Let $(M,J)$ be a compact complex manifold satisfying $c_1(M)=0 \in H^2(M,\RR)$, and $p \colon P\rightarrow M$ a principal $K$-bundle with $K$ even-dimensional, as in Section \ref{s:dimredpcf}. Then, $c_1(P,\overline{J})=0\in H^2(P,\RR)$, where $\overline{J}$ is given by \eqref{Jtot}.
\end{lemma}
\begin{proof}
    By the exact sequence on $P$
    $$
    0 \longrightarrow VP \longrightarrow TP \longrightarrow p^*TX \longrightarrow 0
    $$
    we have that:
    $$
    c_1(P,\overline{J}) = c_1(TP) = c_1(VP) + p^*c_1(X) = c_1(VP).
    $$
    Now, there is a global smooth trivialization of the complex bundle%\footnote{RGM: given $\{\xi_i\}\subset (\mathfrak{k},J_K)$ a cpx basis, observe $V_pP = \mathrm{span}_{\mathbb{R}}(X^{\xi_i}_p,\overline{J}X^{\xi_i}_p)$. Then, a global smooth trivialization over $P$ is given by the (cpx) bundle iso. $VP \rightarrow \underline{\mathbb{C}}^r$, $(p,\sum_i (a_i +b_i \overline{J})X^{\xi_i}_p) \mapsto (a_i + i b_i)_p$, where RHS is the trivial cpx vector bundle over $P$ $\Rightarrow c_1(VP,\overline{J})=0$.}
    $$
    (VP,\overline{J}) \cong VP^{1,0}\cong P \times \mathfrak{g}^{1,0},
    $$
    therefore $c_1(VP,\overline{J})=0$, and the result follows.
\end{proof}

\noindent Assume now that $(M,J)$ admits a holomorphic volume form $\Omega$. Using the previous lemma, the general `cone conjecture' for long-time existence of pluriclosed flow (cf. \cite[Conjecture 5.2]{PCFReg}) combined with Proposition \ref{prop:redPCF} predicts that all solutions to string algebroid pluriclosed flow with $\langle , \rangle_{\mathfrak{k}}$ negative-definite, in this setting, should be global, i.e. should exist on $[0, \infty)$.  This conjecture is in part motivated by the fact that the equation can be formally reduced to a strictly parabolic PDE for a $(1,0)$-form against a controlled background for all times in this setting (cf. Lemma \ref{lem:oneformred} and Proposition \ref{p:STE}).  More generally one expects global solutions for all possible compact gauge groups and pairings. %, and this would be well-motivated by a carrying out a similar reduction.  
We state this conjecture for emphasis:

\begin{conjecture} \label{c:PCFglobal} Given a solution of the Bianchi identity \eqref{eq:BIGHM} on a compact complex manifold with holomorphic volume form, the solution to string algebroid pluriclosed flow \eqref{eq:CPFhol} with this initial data exists on $[0, \infty)$.
\end{conjecture}

When $\langle , \rangle_{\mathfrak{k}}$ is negative-definite, this conjecture can be verified for a large class of backgrounds: those admitting a background metric with nonpositive Chern curvature operator (cf. \cite{LeeStreets, yang2019real}).

\begin{theorem} \label{t:Chernflatglobal} 
Let $\mathcal{Q} \to (M, J)$ denote a holomorphic string algebroid with negative-definite pairing $\IP{,}_{\mathfrak k}$. Assume further that $(M,J)$ admits a Hermitian metric $g'$ with non-positive Chern curvature operator. Given $(\omega_0,\beta_0,h_0)$ a generalized Hermitian metric on $\mathcal{Q}$, the solution to string algebroid pluriclosed flow with this initial data exists on $[0,\infty)$.
\begin{proof} 
Using the evolution equation for $g_t$ expressed in Lemma \ref{prop:UPCFexp2} combined with a standard computation (cf. \cite[Lemma 6.2]{SBIPCF}), one obtains
\begin{align*}
    \square \tr_g g' =&\ - \brs{\gU(g,g')}^2_{g^{-1},g^{-1},g'} - \IP{g', T^2 + F_h^2}_g + F_{g'} (g, g) \leq 0,
\end{align*}
where the last line follows since $g'$ has non-positive Chern curvature operator and the terms $T^2$ and $F_h^2$ are nonnegative.  By the maximum principle there is a uniform a priori upper bound for $\tr_{g_t} g'$, thus an a priori lower bound for $g_t$.  Using Lemma \ref{l:GGfloweqns}, we then obtain, for any choice of background metric $\til{\mathbf{G}}$ on $\mathcal{Q}$,
\begin{align*}
    \square \tr_\mathbf{G} \til{\mathbf{G}} =&\ - \brs{\gU(\mathbf{G},\til{\mathbf{G}})}^2_{g,\mathbf{G}^{-1},\til{\mathbf{G}}} + \tr \til{\mathbf{G}} \mathbf{G}^{-1} S_g^{\til{\mathbf{G}}} \leq C \tr_\mathbf{G} \til{\mathbf{G}},
\end{align*}
where $\mathbf{G} = \mathbf{G}(\omega_t,\beta_t,h_t)$ and the last inequality follows using the a priori lower bound for $g$.  It follows that there is an exponentially growing upper bound for $\tr_\mathbf{G} \til{\mathbf{G}}$.  Returning to Lemma \ref{l:GGfloweqns} a similar argument gives an upper bound for $\tr_{\til{\mathbf{G}}} \mathbf{G}$.  The regularity of the flow then follows from Theorem \ref{t:GGEK}.
\end{proof}
\end{theorem}

Despite the simple expectation of Conjecture \ref{c:PCFglobal}, there should be a delicate behavior of the flow at infinity. Our next result shows that, when $\langle , \rangle_{\mathfrak{k}}$ is negative semi-definite, we can only expect convergence on K\"ahler backgrounds with trivial canonical bundle, and it is natural to expect that on these backgrounds the flow converges.  Simple homogeneous examples, for instance on tori, show that one can obtain different blowdown limits at infinity depending on the choice of bundle.  In the case of line bundles over Riemann surfaces a complete description of the flow is given in \cite{SRYM2}, where it is shown that the topology of the bundle qualitatively affects the behavior of the flow. 

\begin{prop}\label{p:rigidity} 
Assume that $\langle , \rangle_{\mathfrak{k}}$ is negative semi-definite. Given $(\omega,h)$ a fixed point of pluriclosed flow \eqref{eq:CPFhol} satisfying the Bianchi identity \eqref{eq:BIGHM} on a compact complex manifold with holomorphic volume form, the associated metric $\gw$ is K\"ahler Ricci-flat and $h$ is flat.
\begin{proof} 
%MGF: not sure about this other proof.
%We give two alternative proofs. Firstly, by Proposition \ref{p:STE} the one-form reduction \eqref{eq:oneformPCFCY} is constant and by Proposition \ref{prop:dilatonfuncmonotone} it follows that $\omega$ is conformally balanced and $\Lambda_\omega F_h = 0$, that is, we have a solution of the Hull-Strominger system \eqref{f:HS} with $\langle , \rangle_{\mathfrak{k}}$ negative semi-definite. The statement follows now by application of \cite[Proposition 2.15]{GFGM2}.

%Alternatively, 
Using the equivalent form of the pluriclosed flow expressed in Lemma \ref{prop:UPCFexp2}, we see that a fixed point satisfies
\begin{align*}
    \Lambda_\omega F_g - (T^2 + F_A^2)(\cdot,J\cdot) =&\ 0.
\end{align*}
Taking the trace, we see that $s_C - \brs{T}^2 - \brs{F_A}^2 \equiv 0$, where $s_C$ denotes the Chern scalar curvature.  In particular, $s_C \geq 0$.  Since $c_1(M)$ vanishes in Bott-Chern cohomology, it follows from \cite{Gauduchonfibres} (cf. also \cite{GRFBook} Proposition 8.30) that $s_C \equiv 0$.  From this it easily follows that $T \equiv 0$ and $F_A \equiv 0$.  Returning to the displayed equation above it is clear now that $\gw$ is K\"ahler with vanishing Ricci curvature, as claimed.
\end{proof}
\end{prop}

\subsection{Geometrization of Reid's Fantasy}

A key example of a uniformization problem in low-dimensional complex geometry is `Reid's fantasy' \cite{Reid1987moduli}, which suggests that the moduli space of complex threefolds with holomorphically trivial canonical bundle may be connected, after allowing for passage through conifold transitions.  As noted by Reid, a natural strategy is to associate canonical geometric structures to such manifolds to better understand the moduli space.  Several works \cite{Collins2024stability,FuYau2,FuYau1,FuLiYau2012} have proposed using the Hull-Strominger system of equations as a tool for this geometrization problem.  Pluriclosed flow on string algebroids presents a new approach to the `geometrization of Reid's fantasy' through the Hull-Strominger system which has certain key conceptual differences from these proposals. % these approaches.

Prior PDE approaches ask for a conformally balanced metric on $(M,J)$, and seek to construct a solution of the Bianchi identity equation \eqref{eq:BIGHM}. %, whereas our approach, 
Following closely the variational analysis in \cite{garciafern2018canonical} (see Section \ref{sec:dilaton}), our approach is to ask for initial data satisfying the Bianchi identity equation, and then flow while preserving this condition to produce a solution of the coupled Hermitian-Einstein system, which, under natural topological assumptions, is equivalent to the Hull-Strominger system \eqref{f:HS}. Another key difference is that we allow for arbitrary choices of the bundle $P^c$ with vanishing first Pontryagin class, whereas the approaches above demand that $P^c$ is the bundle of split frames of $T^{1,0} \oplus V$, for some auxiliary  holomorphic vector bundle $V$. Furthermore, in e.g. \cite{FuYau2,LiYau2005,PPZ0} the Bianchi identity equation is strongly coupled to the choice of metric on the base, as one chooses the Chern connection of this metric as the connection on $T^{1,0}$, whereas in our setup this connection is a free parameter, as proposed e.g. in \cite{GFGM,garciafern2018canonical,garciafern2020gauge}.

In the remainder of this section we present a speculative discussion of the use of string algebroid pluriclosed flow \eqref{eq:CPFhol} in the geometrization of Reid's fantasy. Our overall expectation is that conifold transitions arise naturally as singularities at infinity for the flow on a Calabi-Yau threefold $(M,J)$, for a suitable choice of holomorphic bundle. Recall that the basic setup of a conifold transition is given by a disjoint union of holomorphic rational curves $\bigsqcup_{j=1}^k C_j \subset M$ which are rigid, that is $C_j \cong \mathbb{CP}^1$ and has normal bundle
$$
N_{C_j} \cong \mathcal{O}(-1)\oplus \mathcal{O}(-1),
$$
and a contraction map $(M,J) \to X_0$ to a singular Calabi-Yau threefold with an ordinary double point replacing each $C_j$. From work of Friedman \cite{friedman1983global}, Kawamata \cite{kawamata1992unobstructed}, and Tian \cite{tian1998smoothing}, $X_0$ admits a deformation to a smooth Calabi-Yau threefold $X_t$, provided that there exist positive integers $n_1^-, \ldots, n_s^-,m_{s+1}^+, \ldots m_k^+ \in \mathbb{N}$, such that
$$
\sum_{j=1}^s n_j^-[C_j] - \sum_{j=s+1}^k m_j^+[C_j] = 0
$$
in homology. With this setup, we propose to choose a collection of holomorphic vector bundles $V_1, \ldots, V_k$ with vanishing first Chern class and second Chern class Poincare dual to $[C_j] \in H_2(M,\ZZ)$. Choosing now $P^c$ to be the bundle of split frames of $\bigoplus_{j=1}^kV_k$, with Lie algebra pairing
$$
\IP{,}_{\mathfrak{k}} = - \sum_{j=1}^s n_j^- \tr_{V_j} + \sum_{j=s+1}^k m_j^+ \tr_{V_j},
$$
given that $(M,J)$ is K\"ahler, we can find a solution of the Bianchi Identity \eqref{eq:BIGHM} with $\omega$ positive-definite, and hence initial data to run the string algebroid pluriclosed flow. A natural expectation with this choice of bundle and pairing is that the blowdown limits of the flow at infinity will produce a canonical metric on a singular space related to the conifold transition. More specifically, we speculate the string algebroid pluriclosed flow to contract the curves $C_1, \ldots, C_s$, while preserving $C_{s+1}, \ldots, C_k$, in relation to the choice of pairing. The blowdown limit is then expected to produce a solution of the Hull-Strominger system \ref{f:HS} on the singular space given by the contraction of the first $s$ curves, possibly with distributional sources located at $C_{s+1}, \ldots, C_k$ (cf. \cite[Section 6]{garcia-fernandez2012}), which can serve as the model limit for the attendant conifold transition.

%\bibliographystyle{abbrv}
%\bibliography{Streets_Master_Bib}

\end{document}